\documentclass[10pt]{amsart}

\usepackage[margin=1in,headheight=13.6pt]{geometry}

\usepackage{setspace}
\usepackage{textcmds} 
\usepackage{amsmath}
\usepackage{amsxtra}
\usepackage{amscd}
\usepackage{amsthm}
\usepackage{amsfonts}
\usepackage{amssymb}
\usepackage{eucal}
\usepackage[all]{xy}
\usepackage{graphicx}
\usepackage{comment}
\usepackage{epsfig}
\usepackage{psfrag}
\usepackage{mathrsfs}
\usepackage{amscd}
\usepackage{rotating}
\usepackage{lscape}
\usepackage{amsbsy}
\usepackage{verbatim}
\usepackage{moreverb}
\usepackage{url}
\usepackage{extarrows}

\usepackage[hidelinks]{hyperref}

\usepackage{cancel}
\usepackage{bbm}

\usepackage{tikz-cd}
\usepackage{tikz}
\usetikzlibrary{matrix,arrows,decorations.pathmorphing}

\DeclareMathAlphabet{\mathpzc}{OT1}{pzc}{m}{it}

\newtheorem{mainthm}{Theorem}

\newtheorem{cor}[subsubsection]{Corollary}
\newtheorem{lem}[subsubsection]{Lemma}
\newtheorem{prop}[subsubsection]{Proposition}

\newtheorem{thm}[subsubsection]{Theorem}

\newtheorem{defin}[subsubsection]{Definition}

\theoremstyle{remark}
\newtheorem{rem}[subsubsection]{Remark}
\newtheorem{example}[subsubsection]{Example}

\numberwithin{equation}{section}

\newcommand{\nc}{\newcommand}
\nc{\renc}{\renewcommand}
\nc{\ssec}{\subsection}
\nc{\sssec}{\subsubsection}
\nc{\on}{\operatorname}

\nc\ol{\overline}
\nc\wt{\widetilde}
\nc\tboxtimes{\wt{\boxtimes}}
\nc\tstar{\wt{\star}}
\nc{\alp}{\alpha}

\nc{\ZZ}{{\mathbb Z}}
\nc{\NN}{{\mathbb N}}
\nc{\OO}{{\mathbb O}}
\renc{\SS}{{\mathbb S}}
\nc{\DD}{{\mathbb D}}
\nc{\GG}{{\mathbb G}}
\renewcommand{\AA}{{\mathbb A}}
\nc{\Fq}{{\mathbb F}_q}
\nc{\Fqb}{\ol{{\mathbb F}_q}}
\nc{\Ql}{\ol{{\mathbb Q}_\ell}}
\nc{\id}{\text{id}}
\nc\X{\mathcal X}

\nc{\Hom}{\on{Hom}}
\nc{\Lie}{\on{Lie}}
\nc{\Loc}{\on{Loc}}
\nc{\Pic}{\on{Pic}}
\nc{\Bun}{\on{Bun}}
\nc{\IC}{\on{IC}}
\nc{\Aut}{\on{Aut}}
\nc{\rk}{\on{rk}}
\nc{\Sh}{\on{Sh}}
\nc{\Perv}{\on{Perv}}
\nc{\pos}{{\on{pos}}}
\nc{\Conv}{\on{Conv}}
\nc{\Sph}{\on{Sph}}
\nc{\Sym}{\on{Sym}}
\nc{\BunBb}{\overline{\Bun}_B}
\nc{\BunNb}{\overline{\Bun}_N}
\nc{\BunTb}{\overline{\Bun}_T}
\nc{\BunBbm}{\overline{\Bun}_{B^-}}
\nc{\BunBbel}{\overline{\Bun}_{B,el}}
\nc{\BunBbmel}{\overline{\Bun}_{B^-,el}}
\nc{\Buno}{\overset{o}{\Bun}}
\nc{\BunPb}{{\overline{\Bun}_P}}
\nc{\BunBM}{\Bun_{B(M)}}
\nc{\BunBMb}{\overline{\Bun}_{B(M)}}
\nc{\BunPbw}{{\widetilde{\Bun}_P}}
\nc{\BunBP}{\widetilde{\Bun}_{B,P}}
\nc{\GUb}{\overline{G/U}}
\nc{\GUPb}{\overline{G/U(P)}}

\nc\syminfty{\on{Sym}^{\infty}}
\nc\lal{\ol{\lambda}}
\nc\xl{\ol{x}}
\nc\thl{\ol{\theta}}
\nc\nul{\ol{\nu}}
\nc\mul{\ol{\mu}}
\nc\Sum\Sigma
\nc{\oX}{\overset{\circ}{X}{}}
\nc{\hl}{\overset{\leftarrow}h{}}
\nc{\hr}{\overset{\rightarrow}h{}}
\nc{\M}{{\mathcal M}}
\nc{\N}{{\mathcal N}}
\nc{\F}{{\mathcal F}}
\nc{\D}{{\mathcal D}}
\nc{\Y}{{\mathcal Y}}
\nc{\G}{{\mathcal G}}
\nc{\E}{{\mathcal E}}
\nc{\CalC}{{\mathcal C}}
\nc\Dh{\widehat{\D}}
\renewcommand{\O}{{\mathcal O}}
\nc{\K}{{\mathcal K}}

\nc{\T}{{\mathcal T}}
\nc{\V}{{\mathcal V}}
\renc{\P}{{\mathcal P}}
\nc{\A}{{\AA}}
\nc{\B}{{\BB}}
\nc{\U}{{\mathcal U}}
\renewcommand{\L}{{\mathcal L}}

\nc{\frn}{{\check{\mathfrak u}(P)}}

\nc{\fC}{\mathfrak C}
\nc\f{{\mathfrak f}}

\nc{\qo}{{\mathfrak q}}
\nc{\po}{{\mathfrak p}}
\nc{\s}{{\mathfrak s}}
\nc\w{\text{w}}
\renewcommand{\r}{{\mathfrak r}}

\renewcommand{\mod}{{\on{-}\mathsf{mod}}}

\newcommand{\bimod}{{\on{-}\mathsf{bimod}}}

\nc\mathi\iota
\nc\Spec{\on{Spec}}
\nc\Mod{\on{Mod}}
\nc{\tw}{\widetilde{\mathfrak t}}
\nc{\pw}{\widetilde{\mathfrak p}}
\nc{\qw}{\widetilde{\mathfrak q}}
\nc{\jw}{\widetilde j}

\nc{\grb}{\overline{\Gr_{X^{\fset}}}}
\nc{\I}{\mathcal I}
\renewcommand{\i}{\mathfrak i}
\renewcommand{\j}{\mathfrak j}

\nc{\lambdach}{{\check\lambda}}
\nc{\Lambdach}{{\check\Lambda}{}}
\nc{\much}{{\check\mu}}
\nc{\omegach}{{\check\omega}}
\nc{\nuch}{{\check\nu}}
\nc{\etach}{{\check\eta}}
\nc{\alphach}{{\check\alpha}}

\nc{\rhoch}{{\check\rho}}

\nc{\Hb}{\overline{\H}}

\nc{\BA}{{\mathbb{A}}}
\nc{\BB}{\mathbb{B}}
\nc{\BC}{{\mathbb{C}}}
\nc{\BD}{{\mathbb{D}}}
\nc{\BE}{{\mathbb{E}}}
\nc{\BF}{{\mathbb{F}}}
\nc{\BG}{{\mathbb{G}}}
\nc{\BH}{{\mathbb{H}}}
\nc{\BI}{{\mathbb{I}}}
\nc{\BM}{{\mathbb{M}}}
\nc{\BN}{{\mathbb{N}}}
\nc{\BO}{{\mathbb{O}}}
\nc{\BP}{{\mathbb{P}}}
\nc{\BQ}{{\mathbb{Q}}}
\nc{\BR}{{\mathbb{R}}}
\nc{\BS}{{\mathbb{S}}}
\nc{\BT}{{\mathbb{T}}}
\nc{\BV}{{\mathbb{V}}}
\nc{\BZ}{{\mathbb{Z}}}

\nc{\bbone}{\mathbbm{1}}
\nc{\bbA}{{\mathbb{A}}}
\nc{\bbB}{\mathbb{B}}
\nc{\bbC}{{\mathbb{C}}}
\nc{\bbD}{{\mathbb{D}}}
\nc{\bbE}{{\mathbb{E}}}
\nc{\bbF}{{\mathbb{F}}}
\nc{\bbG}{{\mathbb{G}}}
\nc{\bbH}{{\mathbb{H}}}
\nc{\bbI}{{\mathbb{I}}}
\nc{\bbL}{{\mathbb{L}}}
\nc{\bbM}{{\mathbb{M}}}
\nc{\bbN}{{\mathbb{N}}}
\nc{\bbO}{{\mathbb{O}}}
\nc{\bbP}{{\mathbb{P}}}
\nc{\bbQ}{{\mathbb{Q}}}
\nc{\bbR}{{\mathbb{R}}}
\nc{\bbS}{{\mathbb{S}}}
\nc{\bbT}{{\mathbb{T}}}
\nc{\bbU}{{\mathbb{U}}}
\nc{\bbV}{{\mathbb{V}}}
\nc{\bbW}{{\mathbb{W}}}
\nc{\bbX}{{\mathbb{X}}}
\nc{\bbY}{{\mathbb{Y}}}
\nc{\bbZ}{{\mathbb{Z}}}

\nc{\CA}{{\mathcal{A}}}
\nc{\CB}{{\mathcal{B}}}

\nc{\CE}{{\mathcal{E}}}
\nc{\CF}{{\mathcal{F}}}
\nc{\CH}{{\mathcal{H}}}

\nc{\CL}{{\mathcal{L}}}
\nc{\CC}{{\mathcal{C}}}
\nc{\CG}{{\mathcal{G}}}
\nc{\CM}{{\mathcal{M}}}
\nc{\CN}{{\mathcal{N}}}
\nc{\CK}{{\mathcal{K}}}
\nc{\CO}{{\mathcal{O}}}
\nc{\CP}{{\mathcal{P}}}
\nc{\CQ}{{\mathcal{Q}}}
\nc{\CR}{{\mathcal{R}}}
\nc{\CS}{{\mathcal{S}}}
\nc{\CU}{{\mathcal{U}}}
\nc{\CV}{{\mathcal{V}}}
\nc{\CW}{{\mathcal{W}}}
\nc{\CX}{{\mathcal{X}}}
\nc{\CY}{{\mathcal{Y}}}
\nc{\CZ}{{\mathcal{Z}}}
\nc{\CI}{{\mathcal{I}}}

\nc{\csM}{{\check{\mathcal A}}{}}
\nc{\oM}{{\overset{\circ}{\mathcal M}}{}}
\nc{\obM}{{\overset{\circ}{\mathbf M}}{}}
\nc{\oCA}{{\overset{\circ}{\mathcal A}}{}}
\nc{\obA}{{\overset{\circ}{\mathbf A}}{}}
\nc{\ooM}{{\overset{\circ}{M}}{}}
\nc{\osM}{{\overset{\circ}{\mathsf M}}{}}
\nc{\vM}{{\overset{\bullet}{\mathcal M}}{}}
\nc{\nM}{{\underset{\bullet}{\mathcal M}}{}}
\nc{\oD}{{\overset{\circ}{\mathcal D}}{}}
\nc{\obC}{{\overset{\circ}{\mathbf C}}{}}
\nc{\obD}{{\overset{\circ}{\mathbf D}}{}}
\nc{\oA}{{\overset{\circ}{\mathbb A}}{}}
\nc{\op}{{\overset{\bullet}{\mathbf p}}{}}
\nc{\oU}{{\overset{\bullet}{\mathcal U}}{}}
\nc{\oZ}{{\overset{\circ}{\mathcal Z}}{}}
\nc{\ofZ}{{\overset{\circ}{\mathfrak Z}}{}}
\nc{\oF}{{\overset{\circ}{\fF}}}

\nc{\fa}{{\mathfrak{a}}}
\nc{\fb}{{\mathfrak{b}}}
\nc{\fc}{{\mathfrak{c}}}
\nc{\fd}{{\mathfrak{d}}}
\nc{\ff}{{\mathfrak{f}}}
\nc{\fg}{{\mathfrak{g}}}
\nc{\fgl}{{\mathfrak{gl}}}
\nc{\fh}{{\mathfrak{h}}}
\nc{\fj}{{\mathfrak{j}}}
\nc{\fl}{{\mathfrak{l}}}
\nc{\fm}{{\mathfrak{m}}}
\nc{\fn}{{\mathfrak{n}}}
\nc{\fu}{{\mathfrak{u}}}
\nc{\fp}{{\mathfrak{p}}}
\nc{\fr}{{\mathfrak{r}}}
\nc{\fs}{{\mathfrak{s}}}
\nc{\ft}{{\mathfrak{t}}}
\nc{\fz}{{\mathfrak{z}}}
\nc{\fsl}{{\mathfrak{sl}}}
\nc{\hsl}{{\widehat{\mathfrak{sl}}}}
\nc{\hgl}{{\widehat{\mathfrak{gl}}}}
\nc{\hg}{{\widehat{\mathfrak{g}}}}
\nc{\chg}{{\widehat{\mathfrak{g}}}{}^\vee}
\nc{\hn}{{\widehat{\mathfrak{n}}}}
\nc{\chn}{{\widehat{\mathfrak{n}}}{}^\vee}

\nc{\fA}{{\mathfrak{A}}}
\nc{\fB}{{\mathfrak{B}}}
\nc{\fD}{{\mathfrak{D}}}
\nc{\fE}{{\mathfrak{E}}}
\nc{\fF}{{\mathfrak{F}}}
\nc{\fG}{{\mathfrak{G}}}
\nc{\fK}{{\mathfrak{K}}}
\nc{\fL}{{\mathfrak{L}}}
\nc{\fM}{{\mathfrak{M}}}
\nc{\fN}{{\mathfrak{N}}}
\nc{\fP}{{\mathfrak{P}}}
\nc{\fU}{{\mathfrak{U}}}
\nc{\fV}{{\mathfrak{V}}}
\nc{\fX}{{\mathfrak{X}}}
\nc{\fY}{{\mathfrak{Y}}}
\nc{\fZ}{{\mathfrak{Z}}}

\nc{\bb}{{\mathbf{b}}}
\nc{\bc}{{\mathbf{c}}}
\nc{\bd}{{\mathbf{d}}}
\nc{\bbf}{{\mathbf{f}}}
\nc{\be}{{\mathbf{e}}}
\nc{\bg}{{\mathbf{g}}}
\nc{\bi}{{\mathbf{i}}}
\nc{\bj}{{\mathbf{j}}}
\nc{\bn}{{\mathbf{n}}}
\nc{\bo}{{\mathbf{o}}}
\nc{\bp}{{\mathbf{p}}}
\nc{\bq}{{\mathbf{q}}}
\nc{\bt}{{\mathbf{t}}}
\nc{\bu}{{\mathbf{u}}}
\nc{\bv}{{\mathbf{v}}}
\nc{\bx}{{\mathbf{x}}}
\nc{\bs}{{\mathbf{s}}}
\nc{\by}{{\mathbf{y}}}
\nc{\bw}{{\mathbf{w}}}
\nc{\bA}{{\mathbf{A}}}
\nc{\bK}{{\mathbf{K}}}
\nc{\bB}{{\mathbf{B}}}
\nc{\bC}{{\mathbf{C}}}
\nc{\bG}{{\mathbf{G}}}
\nc{\bD}{{\mathbf{D}}}
\nc{\bH}{{\mathbf{H}}}
\nc{\bM}{{\mathbf{M}}}
\nc{\bN}{{\mathbf{N}}}
\nc{\bO}{{\mathbf{O}}}
\nc{\bT}{{\mathbf{T}}}
\nc{\bV}{{\mathbf{V}}}
\nc{\bW}{{\mathbf{W}}}
\nc{\bX}{{\mathbf{X}}}
\nc{\bZ}{{\mathbf{Z}}}
\nc{\bS}{{\mathbf{S}}}

\nc{\sA}{{\mathsf{A}}}
\nc{\sB}{{\mathsf{B}}}
\nc{\sC}{{\mathsf{C}}}
\nc{\sD}{{\mathsf{D}}}
\nc{\sF}{{\mathsf{F}}}
\nc{\sG}{{\mathsf{G}}}
\nc{\sK}{{\mathsf{K}}}
\nc{\sM}{{\mathsf{M}}}
\nc{\sO}{{\mathsf{O}}}
\nc{\sW}{{\mathsf{W}}}
\nc{\sQ}{{\mathsf{Q}}}
\nc{\sP}{{\mathsf{P}}}
\nc{\sV}{{\mathsf{V}}}
\nc{\sS}{{\mathsf{S}}}
\nc{\sT}{{\mathsf{T}}}
\nc{\sZ}{{\mathsf{Z}}}
\nc{\sfp}{{\mathsf{p}}}
\nc{\sll}{{\mathsf{l}}}
\nc{\sr}{{\mathsf{r}}}
\nc{\bk}{{\mathsf{k}}}
\nc{\sg}{{\mathsf{g}}}

\nc{\sff}{{\mathsf{f}}}
\nc{\sfb}{{\mathsf{b}}}
\nc{\sfc}{{\mathsf{c}}}
\nc{\sd}{{\mathsf{d}}}
\nc{\se}{{\mathsf{e}}}

\nc{\BK}{{\bar{K}}}

\nc{\tA}{{\widetilde{\mathbf{A}}}}
\nc{\tB}{{\widetilde{\mathcal{B}}}}
\nc{\tg}{{\widetilde{\mathfrak{g}}}}
\nc{\tG}{{\widetilde{G}}}
\nc{\TM}{{\widetilde{\mathbb{M}}}{}}
\nc{\tO}{{\widetilde{\mathsf{O}}}{}}
\nc{\tU}{{\widetilde{\mathfrak{U}}}{}}
\nc{\TZ}{{\tilde{Z}}}
\nc{\tx}{{\tilde{x}}}
\nc{\tbv}{{\tilde{\bv}}}
\nc{\tfP}{{\widetilde{\mathfrak{P}}}{}}
\nc{\tz}{{\tilde{\zeta}}}
\nc{\tmu}{{\tilde{\mu}}}

\nc{\urho}{\underline{\rho}}
\nc{\uB}{\underline{B}}
\nc{\uC}{{\underline{\mathbb{C}}}}
\nc{\ui}{\underline{i}}
\nc{\uj}{\underline{j}}
\nc{\ofP}{{\overline{\mathfrak{P}}}}
\nc{\oB}{{\overline{\mathcal{B}}}}
\nc{\og}{{\overline{\mathfrak{g}}}}
\nc{\oI}{{\overline{I}}}

\nc{\eps}{\varepsilon}
\nc{\hrho}{{\hat{\rho}}}

\nc{\one}{{\mathbf{1}}}
\nc{\two}{{\mathbf{t}}}

\nc{\Rep}{{\mathop{\operatorname{\rm Rep}}}}
\nc{\Tot}{{\mathop{\operatorname{\rm Tot}}}}
\nc{\Ker}{{\mathop{\operatorname{\rm Ker}}}}
\nc{\Hilb}{{\mathop{\operatorname{\rm Hilb}}}}
\nc{\Ext}{{\mathop{\operatorname{\rm Ext}}}}
\nc{\CHom}{{\mathop{\operatorname{{\mathcal{H}}\it om}}}}
\nc{\GL}{{\mathop{\operatorname{\rm GL}}}}
\nc{\gr}{{\mathop{\operatorname{\rm gr}}}}
\nc{\Id}{{\mathop{\operatorname{\rm Id}}}}
\nc{\de}{{\mathop{\operatorname{\rm def}}}}
\nc{\length}{{\mathop{\operatorname{\rm length}}}}
\nc{\supp}{{\mathop{\operatorname{\rm supp}}}}

\nc{\Cliff}{{\mathsf{Cliff}}}
\nc{\Fl}{\on{Fl}}
\nc{\Fib}{{\mathsf{Fib}}}
\nc{\Coh}{{\on{Coh}}}
\nc{\QCoh}{{\on{QCoh}}}
\nc{\IndCoh}{{\on{IndCoh}}}
\nc{\FCoh}{{\mathsf{FCoh}}}

\nc{\reg}{{\text{\rm reg}}}

\nc{\cplus}{{\mathbf{C}_+}}
\nc{\cminus}{{\mathbf{C}_-}}
\nc{\cthree}{{\mathbf{C}_*}}
\nc{\Qbar}{{\bar{Q}}}
\nc\Eis{\on{Eis}}
\nc\Eisb{\ol\Eis{}}
\nc\Eisr{\on{Eis}^{rat}{}}
\nc\wh{\widehat}
\nc{\Def}{\on{Def_{\check{\fb}}(E)}}
\nc{\barZ}{\overline{Z}{}}
\nc{\barbarZ}{\overline{\barZ}{}}
\nc{\barpi}{\overline\pi}
\nc{\barbarpi}{\overline\barpi}
\nc{\barpip}{\overline\pi{}^+}
\nc{\barpim}{\overline\pi{}^-}

\nc{\fq}{\mathfrak q}

\nc{\fqb}{\ol{\fq}{}}
\nc{\fpb}{\ol{\fp}{}}
\nc{\fpr}{{\fp^{rat}}{}}
\nc{\fqr}{{\fq^{rat}}{}}

\nc{\hattimes}{\wh\otimes}

\nc{\bh}{{\bar{h}}}
\nc{\bOmega}{{\overline{\Omega(\check \fn)}}}

\nc{\seq}[1]{\stackrel{#1}{\sim}}

\nc{\cT}{{\check{T}}}
\nc{\cG}{{\check{G}}}
\nc{\cM}{{\check{M}}}
\nc{\cB}{{\check{B}}}
\nc{\cP}{{\check{P}}}

\nc{\ct}{{\check{\mathfrak t}}}
\nc{\cg}{{\check{\fg}}}
\nc{\cb}{{\check{\fb}}}
\nc{\cn}{{\check{\fn}}}
\nc{\cp}{{\check{\fp}}}
\nc{\cm}{{\check{\fm}}}

\nc{\cLambda}{{\check\Lambda}}

\nc{\cla}{{\check\lambda}}
\nc{\cmu}{{\check\mu}}
\nc{\cnu}{{\check\nu}}
\nc{\ceta}{{\check\eta}}

\nc{\DefbE}{{\on{Def}_{\cB}(E_\cT)}}

\nc{\imathb}{{\ol{\imath}}}
\nc{\rlr}{\overset{\longrightarrow}{\underset{\longrightarrow}\longleftarrow}}

\nc{\oBun}{\overset{\circ}\Bun}
\nc{\BunBbb}{\ol{\ol{Bun}}_B}
\nc{\BunBr}{\Bun_B^{rat}}
\nc{\BunBrsg}{\Bun_B^{rat,\on{s.g.}}}
\nc{\BunBrp}{\Bun_B^{rat,polar}}
\nc{\BunBrpbg}{\Bun_B^{rat,polar,\on{b.g.}}}
\nc{\BunBrpsg}{\Bun_B^{rat,polar,\on{s.g.}}}
\nc{\BunTrp}{\Bun_T^{rat,polar}}
\nc{\BunTrpbg}{\Bun_T^{rat,polar,\on{b.g.}}}
\nc{\BunTrpsg}{\Bun_T^{rat,polar,\on{s.g.}}}
\nc{\BunNr}{\Bun_N^{rat}}
\nc{\BunNre}{\Bun_N^{enh,rat}}
\nc{\BunTr}{\Bun_T^{rat}}
\nc{\Vect}{\on{Vect}}
\nc{\Whit}{\on{Whit}}
\nc{\bTb}{\ol{\on{CT}}}

\nc{\bTr}{\on{CT}^{rat}{}}
\nc\jmathr{\jmath^{rat}{}}
\nc{\ux}{\underline{x}}
\nc{\clambda}{{\check\lambda}}
\nc{\calpha}{{\check\alpha}}

\nc{\inftyGrpd}{{\mathsf{Grpd}_\infty}}

\nc{\fset}{\mathsf{fSet}}
\nc{\LocSysG}{\LocSys_{\cG}}
\nc{\Sing}{{\on{Sing}}}
\nc{\dr}{{\on{dR}}}
\nc{\Ind}{\on{Ind}}
\nc{\Sat}{\on{Sat}}
\nc{\Ho}{\on{Ho}}
\nc{\Res}{\on{Res}}
\nc{\sotimes}{\overset{!}\otimes}
\nc{\mmod}{{\on{-}}{\mathbf{mod}}}
\nc{\Maps}{\on{Maps}}
\nc{\CMaps}{{\mathcal Maps}}
\nc{\bMaps}{{\mathbf{Maps}}}

\nc{\dgSch}{\on{DGSch}}
\nc{\dgindSch}{\on{DGindSch}}
\nc{\indSch}{\on{indSch}}
\nc{\Sch}{\mathsf{Sch}}
\nc{\affdgSch}{\on{DGSch}^{\on{aff}}}
\nc{\affSch}{\on{Sch}^{\on{aff}}}

\nc{\Groupoids}{\on{Grpd}}
\nc{\inftypic}{\infty\on{-PicGrpd}}
\nc{\inftyCat}{{\mathsf{Cat}_{\infty}}}
\nc{\MoninftyCat}{\infty\on{-Cat}^{Mon}}
\nc{\SymMoninftyCat}{\infty\on{-Cat}^{\on{SymMon}}}
\nc{\SymMonStinftyCat}{\on{DGCat}^{\on{SymMon}}}
\nc{\MonStinftyCat}{\on{DGCat}^{Mon}}
\nc{\inftystack}{\on{Stk}}
\nc{\inftystackalg}{Stk^{1\text{-}alg}}
\nc{\inftyprestack}{\on{PreStk}}
\nc{\inftydgnearstack}{\on{NearStk}}
\nc{\inftydgstack}{\on{Stk}}
\nc{\inftydgstackalg}{DGStk^{1\text{-}alg}}
\nc{\inftydgprestack}{\on{PreStk}}

\nc{\HC}{\CH\bC}

\nc{\csupp}{\supp}
\nc{\Arth}{\on{Arth}}
\nc{\ArthG}{{\on{Arth}_\cG}}
\nc{\ul}{\underline}

\nc{\Z}{\mathcal{Z}}

\nc{\calN}{\N}
\nc{\calW}{\mathcal{W}}
\nc{\calF}{\mathcal{F}}
\nc{\calH}{\mathcal{H}}
\nc{\calO}{\mathcal{O}}
\nc{\calK}{\mathcal{K}}

\nc{\Ran}{\mathsf{Ran}}
\nc{\Jets}{\on{Jets}}
\nc{\act}{\mathsf{act}}
\nc{\Av}{\mathsf{Av}}
\nc{\Ad}{\on{Ad}}
\nc{\BGRan}{BG_{\Ran}}
\nc{\colim}{\on{colim}}
\nc{\codim}{\on{codim}}
\nc{\cpt}{{\on{cpt}}}
\nc{\dR}{{\on{dR}}}
\nc{\DGCat}{\mathsf{DGCat}}
\nc{\DGCatcont}{\on{DGCat}_{cont}}
\nc{\glob}{{\on{glob}}}
\nc{\loc}{{\on{loc}}}

\renewcommand{\op}{{\on{op}}}
\nc{\pt}{{\on{pt}}}
\nc{\PreStk}{{\mathsf{PreStk}}}
\nc{\Cat}{{\mathsf{Cat}}}

\nc{\ShvCat}{{\mathsf{ShvCat}}}

\nc{\restr}[2]{\left. #1 \right |_{#2}}
\nc{\uprestr}[2]{\left. #1 \right |^{#2}}

\nc{\bLoc}{{\mathbf{Loc}}}
\nc{\bGamma}{{\mathbf{\Gamma}}}
\nc{\bLocA}{\mathbf{Loc}^\A}
\nc{\bGammaA}{\mathbf{\Gamma}^\A}
\nc{\bLocB}{\mathbf{Loc}^\B}
\nc{\bGammaB}{\mathbf{\Gamma}^\B}
\nc{\bLocH}{\mathbf{Loc}^\H}
\nc{\bGammaH}{\mathbf{\Gamma}^\H}
\nc{\gen}{\mathsf{gen}}

\nc{\hto}{\hookrightarrow}

\nc{\ext}{\mathsf{ext}}

\nc{\ev}{\mathsf{ev}}
\nc{\rat}{\mathsf{rat}}

\nc{\usotimes}[1]{\underset{#1}{\otimes}}
\nc{\ustimes}[1]{\underset{#1}{\times}}
\nc{\uscolim}[1]{\underset{#1}{\colim}}

\nc{\ch}{{\mathfrak{ch}}}

\renc{\fD}{{\Dmod}}
\nc{\fH}{{\mathfrak{H}}}

\nc{\p}{{\mathfrak{p}}}
\renc{\r}{{\mathfrak{r}}}

\nc{\xto}{\xrightarrow}

\renc{\sec}{\section}
\nc{\enh}{\mathsf{enh}}

\renc{\gen}{\mathsf{gen}}
\nc{\BunGBgen}{\Bun_G^{B-\gen}}
\nc{\BunGHgen}{\Bun_G^{H-\gen}}
\nc{\BunGNgen}{\Bun_G^{N-\gen}}

\nc{\Fun}{\mathsf{Fun}}
\nc{\End}{\mathsf{End}}

\nc{\lr}{\xymatrix{ \ar@<-0.4ex>[r] \ar@<.5ex>[l]  & } }
\nc{\rr}{\xymatrix{ \ar@<-0.2ex>[r] \ar@<.7ex>[r]  & } }
\nc{\rrr}{\xymatrix{ \ar@<.0ex>[r] \ar@<.7ex>[r] \ar@<-0.7ex>[r] & } }

\nc{\Stab}{\mathsf{Stab}}
\nc{\Orb}{\mathsf{Orb}}

\renc{\exp}{\mathit{exp}}

\renc{\q}{\mathfrak{q}}

\nc{\virg}[1]{``#1"}

\renc{\bold}[1]{\boldsymbol{#1}}

\nc{\bigt}[1]{\big( #1 \big) }
\nc{\Bigt}[1]{\Big( #1 \Big) }

\nc{\extwhit}{{\CW h}(G,\mathsf{ext})}

\nc{\footcite}{\footnote}

\nc{\GA}{{G(\AA)}}
\nc{\GO}{{G(\OO)}}

\nc{\Shv}{\mathsf{Shv}}
\nc{\inc}{\mathsf{inc}}

\nc{\Par}{\mathsf{Par}}

\renc{\i}{\mathfrak{i}}

\nc{\NA}{N(\AA)}
\nc{\VA}{V(\AA)}

\nc{\Glue}{\mathsf{Glue}}
\nc{\laxlim}{\text{laxlim}}

\nc{\FT}{\mathsf{FT}}

\nc{\out}{\mathsf{out}}

\nc{\hol}{\mathsf{hol}}
\nc{\Hol}{\on{Hol}}
\nc{\add}{\mathsf{add}}

\nc{\sto}{\rightsquigarrow}
\nc{\squigto}{\rightsquigarrow}

\nc{\fW}{\mathfrak{W}}

\nc{\vrho}{\varrho}

\nc{\counit}{\mathsf{counit}}
\nc{\unit}{\mathsf{unit}}
\nc{\corr}{\mathsf{corr}}
\nc{\Corr}{\mathsf{Corr}}

\nc{\IndSch}{\mathsf{IndSch}}
\nc{\Tate}{{\mathsf{Tate}}}

\nc{\surjto}{\twoheadrightarrow}

\renc{\j}{\mathfrak{j}}

\nc{\J}{\mathcal{J}}

\nc{\pro}{\mathsf{pro}}
\nc{\fty}{\mathsf{ft}}
\nc{\Pro}{\mathsf{Pro}}

\nc{\coact}{\mathsf{coact}}
\nc{\aff}{\mathsf{aff}}

\nc{\Nilp}{\on{Nilp}}

\nc{\Gch}{{\check{G}}}
\nc{\Pch}{{\check{P}}}
\nc{\Mch}{{\check{M}}}
\nc{\Qch}{{\check{Q}}}

\nc{\LL}{\mathbb{L}}

\nc{\LS}{{\on{LS}}}

\nc{\x}{\varkappa} 

\nc{\Otimes}{\boldsymbol{\otimes}}
\nc{\Times}{\boldsymbol{\times}}
\nc{\flip}{\text{<}}

\nc{\coeffRan}{\mathsf{coeff}^{\Ran}}

\nc{\Ha}{H(\sA)}
\nc{\Groups}{\mathsf{Groups}}

\nc{\Groth}{\mathsf{Groth}}

\nc{\rlto}{\rightleftarrows}

\nc{\DGCatRan}{\ShvCatCrys(\Ran)}

\nc{\longto}{\longrightarrow}

\renc{\Jets}{\mathsf{Jets}}
\nc{\mer}{\mathsf{mer}}

\nc{\W}{\mathcal{W}}

\nc{\Sect}{\mathsf{Sect}}
\renc{\Maps}{\mathsf{Maps}}

\renc{\bf}{\mathbf{f}}

\nc{\y}{\mathtt{y}}
\renc{\x}{\mathtt{x}}

\nc{\un}{{\it un}}
\nc{\indep}{\mathsf{indep}}
\nc{\CoAlg}{\mathsf{CoAlg}}

\nc{\coeff}{\mathsf{coeff}}

\nc{\R}{\mathcal{R}}

\renc{\hat}{\widehat}

\nc{\TK}{T(\mathsf{K})} 
\nc{\TtKK}{\Tt(\mathpzc{K})} 
\nc{\TtK}{\Tt(\mathsf{K})} 
\nc{\KK}{\mathpzc{K}}

\nc{\Dmod}{\mathfrak{D}}

\nc{\curs}[1]{\mathpzc{#1}}

\nc{\Bshv}{\bold{\B}}
\nc{\Bind}{\H_{\indep}}
\nc{\BRan}{\H_{\Ran}}
\nc{\ARan}{\A_{\Ran}}
\nc{\Aind}{\A_{\indep}}

\nc{\GrRan}{\Gr}
\nc{\Gr}{\mathsf{Gr}}
\nc{\GrGRan}{\Gr_{G}}
\nc{\GrGind}{\Gr_{G}^{\indep}}
\nc{\Grind}[1]{\Gr_{#1}^{\indep} }
\nc{\GrGdom}{\curs{Gr}_G}

\nc{\GMapsRan}[1]{\mathsf{GMaps}(X,{#1})}
\nc{\GSectRan}[1]{\mathsf{GSect}({#1}/X)}
\nc{\GMapsind}[1]{\mathsf{GMaps}(X,{#1})^\indep}
\nc{\GSectind}[1]{\mathsf{GSect}({#1}/X)^\indep}
\nc{\GMapsdom}[1]{\curs{GMaps}(X,{#1})}
\nc{\GSectdom}[1]{\curs{GSect}({#1}/X)}

\nc{\chind}{\ch^{\indep}}
\nc{\chdom}{\curs{ch}}

\nc{\QSect}[1]{\curs{QSect}(#1/X)} 
\nc{\QMaps}[1]{\curs{QMaps}(X,#1)} 

\nc{\Zar}{\mathit{Zar}}

\nc{\loccit}{\textit{loc.$\,$cit.}}
\nc{\Crys}{\on{Crys}}
\nc{\ShvCatCrys}{\ShvCat^{\Crys}}

\nc{\BPE}{{\BP E}}
\nc{\BVE}{{\BV E}}
\nc{\BBE}{{\BB E}}

\nc{\Wh}{{{\CW}h}}

\nc{\ChiralCat}{\mathsf{ChiralCat}}
\nc{\RRep}{\mathfrak{R}ep}
\nc{\SSph}{\mathfrak{S}ph}

\nc{\tto}{\twoheadrightarrow}
\nc{\disj}{{\mathsf{disj}}}

\nc{\C}{\CC}
\nc{\Tch}{{\check{T}}}

\nc{\good}{\mathsf{good}}
\nc{\triv}{\mathsf{triv}}

\nc{\Alg}{\mathsf{Alg}}
\nc{\CAlg}{\mathsf{CAlg}}
\nc{\Spread}{\mathsf{Spread}}
\nc{\Dom}{\mathsf{Dom}}

\nc{\Jac}{\on{Jac}}
\renc{\CD}[1]{{#1}^{\on{CD}}}

\nc{\String}{\on{String}}

\renc{\min}{{\mathit{min}}}

\nc{\rrep}{\on-\!\mathbf{rep}}

\nc{\WWh}{\mathfrak{W}h}
\nc{\Grpd}{\mathsf{Grpd}}

\nc{\timesdisj}{\overset{\circ}\times}

\renc{\NA}{N(\sA)}
\nc{\chiral}{\mathsf{chiral}}

\nc{\Hopf}{\mathsf{Hopf}}

\nc{\heart}{\heartsuit}
\nc{\kk}{\mathbbm{k}} 

\nc{\HHom}{\CH{om}} 
\nc{\Cone}{\on{Cone}}

\nc{\EE}{\mathbb{E}}
\renc{\HC}{{\on{HC}}}
\nc{\HH}{{\on{HH}}}
\nc{\even}{{\on{even}}}
\nc{\SingSupp}{\on{SingSupp}}
\nc{\Supp}{\on{Supp}}

\nc{\temp}{{\mathit{temp}}}
\nc{\geom}{{\mathit{geom}}}
\nc{\ren}{{\mathit{ren}}}
\nc{\naive}{{\mathit{naive}}}
\nc{\conaive}{{\mathit{conaive}}}
\nc{\spec}{\mathit{spec}}

\nc{\gch}{\mathfrak{\check{g}}}

\nc{\Hecke}{\on{Hecke}}

\nc{\LSGch}{{\LS_\Gch}}

\nc{\Hsx}[2]{\H_{{#1} \leftarrow {#2}}}
\nc{\Hdx}[2]{\H_{{#1} \to {#2}}}
\nc{\Hcorr}[3]{ \H_{{#1} \leftarrow {#2} \to {#3}} }
\nc{\Hopcorr}[3]{ \H_{{#1} \to {#2} \leftto {#3}} }

\nc{\ICohsx}[2]{\ICohW_{{#1} \leftarrow {#2}}}
\nc{\ICohdx}[2]{\ICohW_{{#1} \to {#2}}}
\nc{\ICohcorr}[3]{ \ICohW_{{#1} \leftarrow {#2} \to {#3}} }
\nc{\ICohopcorr}[3]{ \ICohW_{{#1} \to {#2} \leftto {#3}} }

\nc{\QCohsx}[2]{\QCohW_{{#1} \leftarrow {#2}}}
\nc{\QCohdx}[2]{\QCohW_{{#1} \to {#2}}}
\nc{\QCohcorr}[3]{ \QCohW_{{#1} \leftarrow {#2} \to {#3}} }
\nc{\QCohopcorr}[3]{ \QCohW_{{#1} \to {#2} \leftto {#3}} }

\renc{\AA}{\bbA}
\nc{\Asx}[2]{\AA_{{#1} \leftarrow {#2}}}
\nc{\Adx}[2]{\AA_{{#1} \to {#2}}}
\nc{\Acorr}[3]{ \AA_{{#1} \leftarrow {#2} \to {#3}} }
\nc{\Aopcorr}[3]{ \AA_{{#1} \to {#2} \leftto {#3}} }

\nc{\Bsx}[2]{\B_{{#1} \leftarrow {#2}}}
\nc{\Bdx}[2]{\B_{{#1} \to {#2}}}
\nc{\Bcorr}[3]{ \B_{{#1} \leftarrow {#2} \to {#3}} }
\nc{\Bopcorr}[3]{ \B_{{#1} \to {#2} \leftto {#3}} }

\nc{\ICohzero}[3]{\ICoh_0 \bigt{#1 \times_{{#2}_\dR} #3}}
\nc{\IndCohzero}{\ICohzero}

\nc{\form}[3]{#1 \times_{{#2}_\dR} #3 }

\nc{\ind}{{\mathsf{ind}}}
\nc{\oblv}{{\mathsf{oblv}}}

\nc{\Aff}{\mathsf{Aff}}
\nc{\dgAff}{\Aff}

\nc{\deloop}{\mathsf{deloop}}
\renc{\loop}{\mathsf{loop}}

\nc{\coev}{\mathsf{coev}}

\nc{\bE}{\mathbf{E}}

\nc{\ShvCatH}{{\ShvCat^{\bbH}}}
\nc{\ShvCatQW}{\ShvCat^{\QCohW}}

\nc{\bbimod}{\on{-}\mathbf{bimod}}

\nc{\Tw}{\mathsf{Tw}}
\nc{\Arr}{\mathsf{Arr}}

\nc{\bDelta}{\bold\Delta}
\nc{\BiCat}{\mathsf{BiCat}}
\nc{\Seg}{\mathsf{Seg}}
\nc{\Cart}{\mathsf{Cart}}
\nc{\Bimod}{\mathsf{Bimod}}
\nc{\lax}{\mathit{lax}}

\nc{\pr}{\mathsf{pr}}

\nc{\zero}{ \{ 0 \}   }
\nc{\Perf}{\on{Perf}}

\nc{\leftto}{\leftarrow}
\nc{\lto}{\leftto}
\nc{\xlto}[1]{\xleftarrow{#1}}

\nc{\ltemp}{{}^\temp}
\nc{\TwCorr}{\mathsf{TwCorr}}

\nc{\Affover}[1]{{\Aff_{/#1}}}
\nc{\Affoverop}[1]{{( \Affover{#1})^\op}}

\nc{\AffOver}[2]{{(\Aff_{#1})_{/#2}}}

\nc{\AffOverop}[2]{{( \AffOver{#1}{#2})^\op}}

\nc{\aft}{{\mathit{aft}}}

\renc{\vert}{{\mathit{vert}}}
\nc{\horiz}{{\mathit{horiz}}}
\nc{\type}{{\mathit{type}}}
\nc{\adm}{{\mathit{adm}}}

\nc{\g}{\mathfrak{g}}

\nc{\free}{\mathsf{free}}

\nc{\Sform}{{S \times_{S_\dR} S}}
\nc{\Yform}{{\Y \times_{\Y_\dR} \Y}}
\nc{\SdR}{ {S_{\dR}}}

\nc{\laft}{{\mathit{laft}}}

\nc{\Affevcocaft}{\Aff_{\aft}^{< \infty}}
\nc{\Affaftevcoc}{\Aff_{\aft}^{< \infty}}

\nc{\Affevcoclfp}{\Aff_{\lfp}^{< \infty}}

\nc{\Schevcoclfp }{\Sch_{\lfp}^{< \infty}}
\nc{\Schevcocaft}{\Sch_{\aft}^{< \infty}}
\nc{\Schaftevcoc}{\Sch_{\aft}^{< \infty}}

\nc{\Stkevcoc}{\Stk^{< \infty}}
\nc{\Stkevcoclfp}{\Stk_{\lfp}^{< \infty}}
\nc{\Stkperfevcoclfp}{\Stk_{\mathit{perf},\lfp}^{< \infty}}
\nc{\Stkperfevcoc}{\Stk_{\mathit{perf}}^{< \infty}}
\nc{\Stkperflfp}{\Stk_{\mathit{perf},\lfp}}
\nc{\Stkperf}{\Stk_{\mathit{perf}}}
\nc{\Stklfp}{\Stk_{\lfp}}

\nc{\evcoc}{\mathit{bdd}}

\nc{\ICoh}{\IndCoh}

\nc{\citep}{\cite}

\renc{\H}{\bbH}

\nc{\uno}{\mathbbm{1}}

\nc{\Tang}{\mathbb{T}}

\nc{\LieAlg}{\mathsf{LieAlg}}

\nc{\Serre}{{\on{Serre}}}

\nc{\MPreStk}{\mathsf{MPreStk}}

\nc{\all}{{\on{all}}}

\nc{\QCohwedge}{\bbQ^\wedge}
\nc{\ICohwedge}{\bbI^\wedge}
\nc{\ICohW}{\ICohwedge}
\nc{\QCohW}{\QCohwedge}

\nc{\ShvCatA}{\ShvCat^{\AA}}
\nc{\ShvCatB}{{\ShvCat^\B}}

\nc{\naiveto}{{\xto{\naive}}}
\nc{\conaiveto}{{\xto{\conaive}}}

\nc{\strong}{\mathit{strong}}
\nc{\costrong}{\mathit{costrong}}

\nc{\conv}{\mathit{conv}}

\nc{\Q}{\bbQ}

\nc{\bY}{\mathbf{Y}}
\nc{\Loop}{\mathsf{LOOP}}

\nc{\DG}{{\on{DG}}}

\nc{\coind}{\mathsf{coind}}

\nc{\co}{\on{co}}

\nc{\laftdef}{{\mathit{laft-def}}}

\nc{\qsmooth}{{\mathit{qs.smooth}}}
\nc{\smooth}{{\mathit{smooth}}}

\nc{\LKE}{\on{LKE}}
\nc{\RKE}{\on{RKE}}

\nc{\ShvCatAco}{\ShvCatA_{\co}}
\nc{\ShvCatHco}{\ShvCatH_{\co}}

\nc{\Stk}{\mathsf{Stk}}

\renc{\2}{(\infty,2)}
\renc{\1}{(\infty,1)}

\nc{\doubleCat}{\mathsf{doubleCat}}
\nc{\Spaces}{\mathcal{S}\!\mathit{paces}}

\nc{\ALG}{\mathsf{ALG}}
\nc{\MAPS}{\mathsf{MAPS}}

\nc{\CAT}{\mathsf{CAT}}

\nc{\oneCat}{{\Cat_{\1}}}
\nc{\oneCAT}{{\CAT_{\1}}}

\nc{\twoCat}{{\Cat_{\2}}}
\nc{\twoCAT}{{\CAT_{\2}}}

\nc{\DGCAT}{\mathsf{DGCAT}}
\nc{\twoCatDG}{{\CAT_{\2}^\DG}}
\nc{\twoCATDG}{{\CAT_{\2}^\DG}}

\nc{\twoCATDGw}{{\CAT_{\2, w*}^\DG}}
\nc{\twoCATDGww}{{\CAT_{\2, ww*}^\DG}}

\nc{\AlgBimod}{\Alg^{\mathit{bimod}}}
\nc{\AlgBimodDGCat}{\AlgBimod(\DGCat)}
\nc{\ALGBimod}{\ALG^{\mathit{bimod}}}
\nc{\twoAlgBimod}{\ALGBimod}

\nc{\rev}{{\on{rev}}}

\nc{\lfp}{{\mathit{lfp}}}

\nc{\RBeck}{{\on{R-BC}}}
\nc{\LBeck}{{\on{L-BC}}}

\nc{\schem}{\mathit{schem}}
\nc{\proper}{\mathit{proper}}

\nc{\res}{{\mathit{res}}}

\nc{\UQCoh}{\U^{\QCoh}}
\nc{\UQ}{\UQCoh}

\nc{\LieAlgbd}{{\on{Lie-algbd}}}

\nc{\LY}{{L\Y}}

\nc{\TangQ}{\Tang^{\QCoh}}

\nc{\Fil}{{\on{Fil}}}
\nc{\AssGr}{\on{assoc-gr}}

\nc{\red}{{\mathit{red}}}

\nc{\Cech}{\on{Cech}}

\nc{\FormMod}{\mathsf{FormMod}}
\nc{\FormModunderStkevcoc} {\FormMod_{\Stk^{< \infty}/}^\lfp }

\nc{\WW}{{\on{Weyl}}}

\nc{\ULA}{\on{ULA}}

\nc{\vDmod}{\Dmod^{\on{der}}}

\nc{\Tr}{\T{r}}

\nc{\primef}{'\! f}
\nc{\primeDelta}{'\! \Delta}

\nc{\niliso}{\mathit{nil \on- iso}}
\nc{\cart}{\mathit{cart}}

\begin{document}

\onehalfspacing


\title{The center of the categorified ring of differential operators}
\author{Dario Beraldo}




\begin{abstract}
Let $\Y$ be a derived algebraic stack satisfying some mild conditions.
The purpose of this paper is three-fold. First, we introduce and study $\H(\Y)$, a monoidal DG category that might be regarded as a categorification of the ring of differential operators on $\Y$. 
When $\Y = \LS_G$ is the derived stack of $G$-local systems on a smooth projective curve, we expect $\H(\LS_G)$ to act on both sides of the geometric Langlands correspondence, compatibly with the conjectural Langlands functor.
Second, we construct a novel theory of D-modules on derived algebraic stacks. Contrarily to usual D-modules, this new theory, to be denoted by $\vDmod$, is sensitive to the derived structure.
Third, we identify the Drinfeld center of $\H(\Y)$ with $\vDmod(\LY)$, the DG category of $\vDmod$-modules on the loop stack $\LY := \Y \times_{\Y \times \Y} \Y$. 
\end{abstract}

\maketitle

\tableofcontents

\section{Introduction}

\ssec{Overview} \label{ssec: intro to intro}

\sssec{}

Let $\kk$ be a ground field of characteristic zero and $Y = \Spec(A)$ a global complete intersection scheme over $\kk$, for instance the zero locus of a dominant map $\A_{\kk}^n \to \A_{\kk}^1$. 
This is a scheme with very mild singularities (if at all): it is the simplest example of a quasi-smooth scheme.

\medskip

A natural invariant of $Y$ is the differential graded (DG) algebra $\HC(A) := \on{RHom}_{A \otimes_{\kk} A}(A, A)$ of \emph{Hochschild cochains} of $A$. As explained later, the DG category of DG modules over $\HC(A)$, denoted $\H(Y)$ in the sequel, carries a monoidal structure. This monoidal DG category is important in the theory of singular support of coherent sheaves on $Y$; see \cite{AG, BIK} for the notion of singular support and \cite{shvcatHH, strong-gluing, chiral-hom} for the role of $\H(Y)$.

\sssec{}

Now, given any monoidal DG category $\CA$, it is natural to try to compute its Drinfeld center $\Z(\CA)$. In our situation, we show that the Drinfeld center of $\H(Y)$ is equivalent to the DG category of $\fD$-modules on $LY := Y \times_{Y \times_\kk Y} Y$, with the following two complications.

\medskip

The first complication is that the fiber product computing $LY$ must be taken in the derived sense. (In fact, the underived truncation of $LY$ is isomorphic to $Y$.) Practically, this means that in the formula $LY \simeq \Spec(A \otimes_{A \otimes_\kk A} A)$ the tensor product over $A \otimes_\kk A$ must be derived. In the cases of interest, $Y$ is (quasi-smooth but) not smooth. Then $LY$ is extremely derived: technically speaking, $LY$ is unbounded, which means that its DG algebra of functions $A \otimes_{A \otimes_\kk A} A$ (the DG algebra of \emph{Hochschild chains} of $A$) has \emph{infinitely many} nonzero cohomology groups. 

\medskip

The second complication is that the theory of $\fD$-modules as understood up to now (e.g. in \cite{Crys} and in \cite[Volume 2, Chapter 4]{Book}) is insufficient to deal with unbounded derived schemes. To fix this, we introduce a new (and yet natural) theory of $\fD$-modules, denoted by $\vDmod$, which makes the following statement true:
$$
\Z(\H(Y)) \simeq \vDmod(LY).
$$

\ssec{The main results}

We now proceed to explain the above situation more organically and in higher generality.

\sssec{}

As above, let $\kk$ be a ground field of characteristic zero, fixed throughout; undecorated products and tensor products are always taken over $\kk$. 
Let $\Y$ be a derived stack (over $\kk$) satisfying some conditions to be spelled out later. A good example to have in mind is $\Y = Y$ a \emph{quasi-smooth} scheme: for instance, a global complete intersection as above.

\medskip

The first task of this paper is to introduce a monoidal DG category $\H(\Y)$, which should be thought of as a \emph{categorification of the ring of differential operators on $\Y$}. 
The definition of $\H(\Y)$ is given in Section \ref{ssec:intro-defin of H}.
The second task is to compute the Drinfeld center of $\H(\Y)$. Surprisingly, we will discover a new notion of $\fD$-modules on derived stacks, to be denoted by $\vDmod$, which is much more natural than the usual notion. We refer to \cite{Crys} for the usual notion, while the definition of the new notion is the third (and last) task of this paper, performed in Section \ref{sec:beyond-ev-coc} and especially in Section \ref{ssec: derived D-modules}.

\sssec{}

The following example illuminates the difference between $\Dmod$ and $\vDmod$, showing that the latter is more natural. Again, we refer to Section \ref{ssec: derived D-modules}. (We should remark however that the discussion of $\vDmod$ in this paper is by no means complete: we limit ourselves to the definition and to the few properties needed to prove Theorem \ref{main-thm-intro} below. For future applications, a proper foundational treatment is necessary; we plan to address it elsewhere.)

\begin{example} \label{example: intro Weyl algebras}
Let $X = \Spec A$ be an affine DG scheme with $A= (\kk[x_i], d)$ quasi-free. Then we expect $\vDmod(X) \simeq W_A \mod$, where $W_A := (\kk \langle x_i, \partial_i \rangle, d)$ is the Weyl DG algebra constructed in the obvious manner from $A$. On the other hand, $\Dmod(X) \simeq W_B \mod$, where $B$ is the quasi-free DG algebra obtained from $A$ by discarding all variables of degree strictly less than $-1$.
\end{example}

Our main result, which appears in the main body of the paper as Theorem \ref{thm:Main-first-part}, goes as follows:

\begin{mainthm} \label{main-thm-intro}
Let $\Y$ be a quasi-compact derived algebraic stack that is perfect, bounded and with perfect cotangent complex.\footnote{see Section \ref{sssec:stacks reminders} for our conventions on algebraic stacks and the meaning of the adjectives} 
Then the Drinfeld center of the monoidal DG category $\H(\Y)$ is naturally equivalent to $\vDmod(\LY)$, where $\LY := \Y \times_{\Y \times \Y} \Y$ is the loop stack of $\Y$.
\end{mainthm}

\sssec{}

\nc{\cl}{{\on{cl}}}

By its very construction (see \cite{Crys}), the DG category $\Dmod(\X)$ is not sensitive to the derived structure of $\X$. Indeed, one defines $\Dmod(\X) := \QCoh(\X_\dR)$, where $\X_\dR$ is the de Rham prestack of $\X$; in turn, $\X_\dR$ is canonically isomorphic to the de Rham prestack of the classical truncation of $\X$.

On the other hand, $\vDmod(\X)$ is sensitive to the derived structure of $\X$ in general. More precisely: we will construct a canonical functor $\Upsilon^{\fD}_{\X/\pt}: \vDmod(\X) \to \Dmod(\X)$ that is an equivalence when $\X$ is \emph{bounded} but not in general. If fact, $\Upsilon^{\fD}_{\X/\pt}$ fails to be an equivalence even for the simplest unbounded affine scheme $\X$, see Proposition \ref{cor:vDmod on Yn}.

Let us recall that a derived stack is said to be \emph{bounded}\footnote{alias: \emph{eventually coconnective}} if it is smooth-locally of the form $\Spec (A)$, where $A$ is a cohomologically bounded commutative DG algebra.  A stack is said to be \emph{unbounded} if it is not bounded.

\begin{example}
Let $\Y$ be smooth, so that $\LY$ is quasi-smooth.\footnote{Recall that a derived stack is quasi-smooth if it is smooth-locally isomorphic to the derived zero locus of a map $\A^m \to \A^n$. (Equivalently, if the cotangent complex is perfect of Tor amplitude $[-1,1]$.)}
 Since quasi-smooth stacks are easily seen to be bounded, we have $\vDmod(\LY) \simeq \Dmod(\LY)$ canonically.
If moreover $\Y = Y$ is a (smooth) scheme, then $\Dmod(LY) \simeq \Dmod(Y)$. Combining these observations, the theorem states that, for $Y$ a smooth scheme, the center of $\H(Y)$ is equivalent to $\Dmod(Y)$. This assertion can be regarded as a categorification of the classical theorem that relates the Hochschild (co)homology of the ring of differential operators on $Y$ to the de Rham cohomology of $Y$, see e.g. \cite{Bry-Get}.
\end{example}

\sssec{}

For $\Y$ quasi-smooth but not smooth, one checks that $\LY$ is unbounded (see below for the simplest example). This is the easiest situation for which the full content of Theorem \ref{main-thm-intro} can be appreciated and it is, after all, our main case of interest. For instance, in the next section we will consider the quasi-smooth stack $\Y = \LS_G$ that parametrizes $G$-local systems over a smooth complete curve.

\begin{example} \label{example:loops}
Suppose $Y = \Spec(\kk[\epsilon])$, with $\epsilon$ in cohomological degree $-1$. This is arguably the simplest truly derived affine scheme. It is quasi-smooth, but not smooth. By \eqref{eqn:LG as a product}, its loop scheme $LY$ is isomorphic to the spectrum of the graded algebra $\kk[\epsilon, \eta]$, where $\epsilon$ is as before and $\eta$ is  a variable in (cohomological) degree $-2$. It follows that $LY$ is unbounded. 
\end{example}

\ssec{The DG category $\H(\Y)$ and its relatives} \label{ssec:intro-defin of H}

The origin of (relatives of) $\H(\Y)$ can be traced back to the \emph{spectral gluing theorem} occurring in geometric Langlands, where the categories 
$ \ICoh_0((\LS_G)^\wedge_{\LS_P})$
play a crucial role, see \cite{AG2}.

\sssec{} \label{sssec:notation}

Let us explain the notations:
\begin{itemize}
\item
$G$ denotes a connected reductive group over $\kk$;
\item
$P$ is one of its parabolic subgroups;
\item
$\LS_P$ (resp., $\LS_G$) denotes the \emph{quasi-smooth} stack of de Rham $P$-local systems (resp., $G$-local systems) on a smooth complete $\kk$-curve $X$;
\item
the map $\LS_P \to \LS_G$ used to construct the formal completion is the natural one, induced by the inclusion $P \subseteq G$.
\end{itemize}
Finally, and most importantly, the definition of the DG category $ \ICoh_0((\LS_G)^\wedge_{\LS_P})$ is an instance of the following general construction, applied to the above map $\LS_P \to \LS_G$.

\sssec{}

Before explaining the construction, let us point the reader to Section \ref{ssec:conventions} for our conventions regarding DG categories and higher category theory. The conventions regarding derived algebraic geometry, prestacks, formal completions and ind-coherent sheaves are explained in Section \ref{sec:IndCoh}.

\begin{defin}
Let $f: \Y \to \Z$ be a map of perfect algebraic stacks with bounded $\Y$. Recall that any quasi-smooth stack (e.g., $\LS_P$) is bounded.
Then we define the DG category $\ICoh_0(\Z^\wedge_\Y)$ as the fiber product
\begin{equation} \label{eqn:original def of ICOHzero}
\ICoh_0(\Z^\wedge_\Y)
:=
\ICoh(\Z^\wedge_\Y) \ustimes{\ICoh(\Y)}
\QCoh(\Y),
\end{equation}
where the functor $\ICoh(\Z^\wedge_\Y) \to \ICoh(\Y)$ is the pull-back along the natural map $\primef: \Y \to \Z^\wedge_\Y$, while the functor $\QCoh(\Y) \to \ICoh(\Y)$ is the standard inclusion $\Upsilon_\Y$, see \cite{ICoh}. The construction of $\ICoh_0$ will treated in more detail in Section \ref{ssec:Defn-ICohzero}.

\end{defin}

\begin{rem}
The reader might have noticed an abuse of notation here: the definition of $\ICoh_0(\Z^\wedge_\Y)$ really depends on the map $\Y \to \Z^\wedge_\Y$, not just on $\Z^\wedge_\Y$. We hope that no confusion will ever arise and refer to Section \ref{sssec:warning-abuse-notation} for further discussion.
\end{rem}

\begin{example}
For $f = \id_\Y$, we have $\ICoh_0(\Y^\wedge_\Y) \simeq \QCoh(\Y)$ tautologically.
\end{example}

\begin{example} \label{example: D-mods as Icohzero -intro}
For $f: \Y \to \pt := \Spec(\kk)$ the structure map, we have 
\begin{equation} \label{eqn:example of ICohzero as Dmodules}
\ICoh_0(\pt^\wedge_\Y) \simeq \Dmod(\Y).
\end{equation}
To see this, recall first from \cite{Crys} that $\Dmod(\Y)$ is defined either as $\QCoh(\Y_\dR)$ or as $\ICoh(\Y_\dR)$, where $\Y_\dR \simeq \pt^\wedge_\Y$. Also, thanks to \cite[Proposition 2.4.4]{Crys}, we know that $\Upsilon_{\Y_\dR} : \QCoh(\Y_\dR) \to \ICoh(\Y_\dR)$ is an equivalence, the equivalence between left and right $\fD$-modules. Then the equivalence \eqref{eqn:example of ICohzero as Dmodules} follows from the fact that $\Upsilon$ intertwines $*$-pullbacks on the $\QCoh$ side with $!$-pullbacks on the $\ICoh$ side.
 In general, $\ICoh_0(\Z^\wedge_\Y)$ is the \virg{correct} way to define the DG category of \emph{relative left} $\fD$-modules with respect to $\Y \to \Z$.
\end{example}

\begin{rem}
For  bounded $\Y$, the inclusion $\Upsilon_\Y$ admits a continuous right adjoint, denoted by $\Phi_\Y$ in this paper.
\end{rem}

\sssec{}

Without any extra assumptions on the map $f: \Y \to \Z$, it is very difficult to get a handle of $\ICoh_0(\Z^\wedge_\Y)$: for example, it is not clear whether this DG category is compactly generated and it is very difficult to exhibit compact objects.
As discussed in Proposition \ref{prop: monadic adjunction ICoh} and Section \ref{sssec: monade adjunction per intro}, the situation simplifies as soon as we restrict to perfect stacks that are \emph{locally of finite presentation} (lfp), that is, with \emph{perfect cotangent complex}.\footnote{Actually, what we really need is that the relative cotangent complex $\Tang_{\Y/\Z}$ be perfect.} 
Note that any quasi-smooth stack tautologically has this property.
This condition guarantees that each $\ICoh_0(\Z^\wedge_\Y)$ has pleasant features: it is compactly generated, self-dual and equipped with a monadic adjunction
\begin{equation} 	\label{eqn:intro-monadic-adjunction}
\begin{tikzpicture}[scale=1.5]
\node (a) at (0,1) {$\QCoh(\Y)$};
\node (b) at (4,1) {$\ICoh_0(\Z^\wedge_{\Y})$,};
\path[->,font=\scriptsize,>=angle 90]
([yshift= 1.5pt]a.east) edge node[above] {$ (\primef)_*^\ICoh \circ \Upsilon_\Y $} ([yshift= 1.5pt]b.west);
\path[->,font=\scriptsize,>=angle 90]
([yshift= -1.5pt]b.west) edge node[below] {$ \Phi_\Y \circ (\primef)^! $} ([yshift= -1.5pt]a.east);
\end{tikzpicture}
\end{equation}
so that 
\begin{equation} \label{eqn:def of ICOHzero with universal enveloping alg}
\ICoh_0(\Z^\wedge_{\Y}) \simeq (\Phi_\Y \circ \U(\Tang_{\Y/\Z}) \circ \Upsilon_\Y) \mod(\QCoh(\Y)).
\end{equation}
Here, $\U(\Tang_{\Y/\Z})$ is the \emph{universal envelope} of the Lie algebroid $\Tang_{\Y/\Z} \to \Tang_{\Y}$, which is, \emph{by definition}, the monad $(\primef)^! \circ (\primef)_*^\ICoh$ acting on $\ICoh(\Y)$, see \cite[Volume 2, Chapter 8, Section 4.2]{Book}.

\begin{rem}
In this paper, we use the notation $\UQ(\Tang_{\Y/\Z})$ for the monad $\Phi_\Y \circ \U(\Tang_{\Y/\Z}) \circ \Upsilon_\Y$. Note that parallel between \eqref{eqn:def of ICOHzero with universal enveloping alg} and the formula
$$
\ICoh(\Z^\wedge_{\Y}) \simeq \U(\Tang_{\Y/\Z}) \mod(\ICoh(\Y))
$$
featuring in \eqref{eqn: U of tang for intro}. The tangent complex appears in this context as a consequence of (a variant of) the equivalence between formal moduli problems and DG Lie algebras. For the latter equivalence, see \cite{DAG-X} and references therein.

\end{rem}

\sssec{}

The assignment
$$
[\Y \to \Z] \squigto 
\ICoh_0(\Z^\wedge_\Y)
$$
enjoys functorialities of two kinds: 
\begin{itemize}
\item
$(\infty, 1)$-categorical functorialities, where we consider $\ICoh_0(\Z^\wedge_\Y)$ as a mere DG category;

\smallskip 
\item
$(\infty, 2)$-categorical functorialities, where we consider $\ICoh_0(\Z^\wedge_\Y)$ as a left module category for $\H(\Z)$, see below.
\end{itemize} 
In this paper we treat the first item, leaving the second item to \cite{shvcatHH}.
However, we are tacitly preparing ourselves for the $(\infty, 2)$-categorical part of the theory, as we will be very much concerned with the study of the \emph{monoidal} DG category
$$
\H(\Y) 
:=
\ICoh_0((\Y \times \Y)^\wedge_\Y).
$$
The monoidal structure on $\H(\Y)$ is the one given by convolution, inherited by the standard convolution structure on $\ICoh((\Y \times \Y)^\wedge_\Y)$. Since $\Y$ is perfect, $\H(\Y)$ is compactly generated and \emph{rigid}\footnote{See Section \ref{sssec:rigidity conventions} or \cite[Section 6.1]{DGCat} for the meaning of the adjective \virg{rigid}.}. By contrast, $\ICoh((\Y \times \Y)^\wedge_\Y)$ is not rigid in general (although it is compactly generated).

\begin{rem}

The adjuction (\ref{eqn:intro-monadic-adjunction}) yields in this case a \emph{monoidal} functor 
$$
\Delta_{*,0} := (\primeDelta)_*^\ICoh \circ \Upsilon_\Y:
\QCoh(\Y) \longto \H(\Y),
$$
with continuous conservative right adjoint. In particular, $\H(\Y)$ is the DG category of modules for the inertia Lie algebra $\Tang^\QCoh_\Y[-1] := \Phi_\Y(\Tang_\Y [-1])$ in $\QCoh(\Y)$.

\end{rem}

\begin{rem} \label{rem:HC}
In the case $\Y=S$ is an affine derived scheme, $\H(S)$ is the monoidal category of modules over $\HC(S)$, the $E_2$-algebra of \emph{Hochschild cochains} of $S$. This point of view drives the study of $\H(\Y)$ carried out in \cite{shvcatHH}.
\end{rem}

\ssec{Module categories over $\H(\Y)$}

Having the monoidal DG category $\H(\Y)$ at our disposal, it is natural to search for interesting examples of module categories, that is, to look for objects of the $\infty$-category $\H(\Y) \mmod$.

\sssec{}

Objects of $\H(\Y) \mmod$ might be regarded as \virg{categorified left $\fD$-modules on $\Y$}, in the same way as objects of $\QCoh(\Y) \mmod$ might be regarded as \virg{categorified quasi-coherent sheaves on $\Y$}.
Among the various ways to justify the validity of this point of view, we mention the following. 
One can equip the $\infty$-category $\H(\Y) \mmod$ with a symmetric monoidal structure, with unit $\QCoh(\Y)$. Then, as shown in \cite{shvcatHH}, we have
$$
\Fun_{\H(\Y) \mmod}(\QCoh(\Y), \QCoh(\Y))
\simeq
\Dmod(\Y),
$$
so that $\H(\Y) \mmod$ is a \emph{delooping} of $\Dmod(\Y)$. In the formula above, the LHS denotes the DG category of $\H(\Y)$-linear continuous endofunctors of $\QCoh(\Y)$.

\sssec{} \label{sssec:analogy-with-differential-operators}

Pushing the analogy further, we may think of the monoidal DG category $\H(\Y)$  as a categorification of the ring of differential operators on $\Y$. Likewise, $\QCoh(\Y)$ corresponds to the left $\fD$-module $\O_\Y$ and the monoidal functor $\Delta_{*,0}$ corresponds to the algebra map from functions to differential operators.

\sssec{}

There are plenty of examples of DG categories carrying a natural action of $\H(\Y)$: 
\begin{itemize}
\item
it is easy to see that $\H(\Y)$ acts by convolution on $\ICoh(\Y)$, on $\QCoh(\Y)$ and on the category of singularities $\ICoh(\Y)/\QCoh(\Y)$;
\item
more generally, for $\X \to \Y$ a map of stacks as above, $\H(\Y)$ acts on $\ICoh_0(\Y^\wedge_\X)$ and on $\ICoh(\Y^\wedge_\X)$, again by convolution;
\item
If $\Y$ is quasi-smooth, the $\H(\Y)$-action on $\ICoh(\Y)$ preserves any subcategory of $\ICoh(\Y)$ cut out by a singular support condition, see \cite{AG}.
\end{itemize}

\sssec{Digression on \virg{geometric Langlands}} \label{sssec:H-action-on-Dmod(BunG)}

A much less trivial example of an $\H(\Y)$-action is given by the following. Referring to Section \ref{sssec:notation} for the notation, let $\Bun_G$ denote the stack of $G$-bundles on $X$ and by $\Gch$ the Langlands dual group of $G$.

We claim that $\H(\LSGch)$ acts on $\Dmod(\Bun_G)$ via \emph{Hecke operators}. This action and the explanation of the terminology, i.e. the connection with \emph{derived Satake}, is addressed in \cite[Section 1.4]{shvcatHH}.
Here we just mention that the datum of such action proves almost immediately the conjecture about \emph{tempered} $\fD$-modules formulated in \cite[Section 12]{AG}.

\ssec{The center of the monoidal DG category $\H(\Y)$}

Let us come back to the contents of the present paper. After having studied the functoriality of the assignment $\ICoh_0$, we turn to the computation of the Drinfeld center $\Z(\H(\Y))$ of $\H(\Y)$.
By definition, $\Z(\H(\Y))$ is the DG category
$$
\Z(\H(\Y)):=
\Fun_{(\H(\Y),\H(\Y))\bimod}(\H(\Y), \H(\Y)).
$$
In the analogy of Section \ref{sssec:analogy-with-differential-operators}, one may suggest that $\Z(\H(\Y))$ is the categorifcation of the center of the ring of differential operators on $\Y$.

\begin{rem}
As pointed out later in Remark \ref{rem:trace = center}, it  turns out that the DG category underlying $\Z(\H(\Y))$ is canonically equivalent to the \emph{trace} of $\H(\Y)$, namely the DG category defined by
$$
\Tr(\H(\Y))
:=
\H(\Y)
\usotimes{\H(\Y) \otimes \H(\Y)^\rev} \H(\Y).
$$
\end{rem}

\sssec{}

We believe that the computation of $\Z(\H(\Y))$ is interesting in its own right. However, we were brought to it by the need to make sure that the monoidal functor
$$
\Dmod(\Y) \xto{\oblv_L} \QCoh(\Y) \xto{\Delta_{*,0}} \H(\Y)
$$
factors through $\Z(\H(\Y))$.
In other words, we wanted to construct a functor 
$$
\zeta: \Dmod(\Y)
\longto
\Z(\H(\Y))
$$
making the following diagram commutative:
\begin{equation} \label{diag:Dmod Y mapping to center}
\begin{tikzpicture}[scale=1.5]
\node (00) at (0,0) {$\Dmod(\Y)$};
\node (10) at (2.8,0) {$\QCoh(\Y).$};
\node (01) at (0,1.2) {$\Z(\H(\Y))$};
\node (11) at (2.8,1.2) {$\H(\Y)$};
\path[->,font=\scriptsize,>=angle 90]
(00.east) edge node[above] {$\oblv_L$} (10.west);
\path[->,font=\scriptsize,>=angle 90]
(01.east) edge node[above] {$\ev$} (11.west);
\path[<-, dashed,font=\scriptsize,>=angle 90]
(01.south) edge node[right] {$\zeta$} (00.north);
\path[<-,font=\scriptsize,>=angle 90]
(11.south) edge node[right] {$\Delta_{*,0}$} (10.north);
\end{tikzpicture}
\end{equation}
Here, the functor $\oblv_L$ is the \virg{left forgetful} functor from $\fD$-modules to quasi-coherent sheaves, see \cite{Crys}, while $\ev$ is the tautological functor that \virg{forgets the central structure}, that is, the evaluation functor
\begin{equation} \label{eqn:oblv from Z to H}
\Z(\H(\Y))
=
\Fun_{(\H(\Y),\H(\Y))\bimod}(\H(\Y), \H(\Y))
\xto{\phi \squigto \phi(1_{\H(\Y)})}
\H(\Y).
\end{equation}

\sssec{Digression on geometric Langlands, again}

The functor $\zeta: \Dmod(\Y) \to \Z(\H(\Y))$ will be important in future applications, which come after having constructed the action of $\H(\LSGch)$ on $\Dmod(\Bun_G)$ mentioned in Section \ref{sssec:H-action-on-Dmod(BunG)}. Indeed, the datum of such action yields in particular a monoidal functor
\begin{equation} \label{eqn:from-Dmod-to-Endofunctors-of-Dmod(BunG)}
\Dmod(\LSGch)
\longto
\Fun(\Dmod(\Bun_G), \Dmod(\Bun_G)),
\end{equation}
defined as the composition
$$
\Dmod(\LSGch)
\xto{\oblv_L}
\QCoh(\LSGch)
\xto{\Delta_{*,0}}
\H(\LSGch)
\xto{\; \act \; }
\Fun(\Dmod(\Bun_G), \Dmod(\Bun_G)).
$$
Now, the commutative diagram (\ref{diag:Dmod Y mapping to center}) guarantees that (\ref{eqn:from-Dmod-to-Endofunctors-of-Dmod(BunG)}) factors through a monoidal arrow
$$
\Dmod(\LSGch)
\longto
\Fun_{\H(\LSGch)}(\Dmod(\Bun_G), \Dmod(\Bun_G)).
$$
I.e., objects of $\Dmod(\LSGch)$ give rise to endofunctors of $\Dmod(\Bun_G)$ that commute with the Hecke operators. Note, by contrast, that endofunctors of $\Dmod(\Bun_G)$ defined by objects of $\QCoh(\LSGch)$ \emph{do not} commute with the Hecke operators in general.

\sssec{} \label{sssec:D-mods map to center}

At a heuristic level, the existence of the dashed arrow in (\ref{diag:Dmod Y mapping to center}) is clear. Indeed, for $Q \in \QCoh(\Y)$ and $\F \in \H(\Y)$, we have
$$
\Delta_{*,0}(Q) \star \F 
\simeq
(\wh p_1)^! (\Upsilon_\Y Q) \stackrel ! \otimes \F,
\hspace{.4cm}
\F \star \Delta_{*,0}(Q)
\simeq
(\wh p_2)^! (\Upsilon_\Y Q) \stackrel ! \otimes \F,
$$ 
where $\star$ denotes the monoidal structure of $\H(\Y)$ and
$$
\wh p_1, \wh p_2: \Yform \rr \Y
$$
are the two projections forming the infinitesimal groupoid of $\Y$. Hence, a \virg{homotopically coherent} identification 
$$
(\wh p_1)^! (\Upsilon_\Y Q) \simeq (\wh p_2)^! (\Upsilon_\Y Q),
$$
that is, a left crystal structure on $Q$, promotes $\Delta_{*,0}(Q)$ to an object of the center of $\H(\Y)$.

\sssec{}

Rather than turning this argument into a proof, we will first compute the full center $\Z(\H(\Y))$ in geometric terms and then exhibit the natural map $\zeta$ from $\Dmod(\Y)$. In this paper we only perform the former task, leaving the latter to a sequel. Let us however anticipate that $\zeta$ is the pushforward functor $\Dmod(\Y) \simeq \vDmod(\Y) \to\vDmod(\LY)$ in the theory of $\vDmod$-modules along the inclusion $\iota:\Y \to \LY$.

\ssec{Conventions} \label{ssec:conventions}

Our conventions on higher category theory and derived algebraic geometry follow those of \cite{Book}. Let us recall the most important ones.
The reader might also consult \cite{DGCat} for a brief digest.

\sssec{}  \label{sssec:Notation DG cats}

Throughout the paper the term \virg{DG category} means \virg{stable presentable $\kk$-linear $\infty$-category}, in the sense of \cite{HA}. 
Unless otherwise stated, our DG categories are cocomplete and functors between them are colimit preserving. 
In other words, we work within the $\infty$-category $\DGCat$ of cocomplete DG categories and continuous functors. Such $\infty$-category is symmetric monoidal when equipped with the Lurie tensor product defined in \cite{HA}.

\sssec{}

Given a DG category $\C$ as above and two objects $c, c' \in \C$, we denote by $\CH om_\C(c, c')$ the DG vector space of morphisms from $c$ to $c'$.

If $\C'$ is another DG category, the symbol $\Fun(\C,\C')$ denotes the DG category of continuous functors from $\C$ to $\C'$. We also set $\End(\C) := \Fun(\C,\C)$.

\sssec{}

We assume familiarity with the notion of dualizability for objects in a symmetric monoidal $\infty$-category, as well as with the notion of ind-completion and compact generation for DG categories. Recall that a compactly generated DG category is automatically dualizable. Given $\C \in \DGCat$, we denote by $\C^{\cpt}$ its non-cocomplete full subcategory of compact objects.

\sssec{} \label{sssec:rigidity conventions}

Let $(\CA,m)$ be a monoidal (cocomplete) DG category. Following  \cite[Section 6.1]{DGCat} and \cite[Appendix D]{ShvCat}, we say that $\CA$ is  \emph{rigid} if:
\begin{itemize}
\item
 the multiplication $m: \CA \otimes \CA \to \CA$ admits a continuous right adjoint, $m^R$;
\item
the unit $u: \Vect \to \CA$ admits a continuous right adjoint, $u^R: \CA \to \Vect$;
\item
the functor $m^R: \CA \to \CA \otimes \CA$ is $(\CA, \CA)$-linear.
\end{itemize}
This notion of rigidity might appear to be non-standard for some readers: as pointed out in \cite[Section 6.1.1]{DGCat}, when $\CA$ is compactly generated, $\CA$ is rigid iff the unit is compact and the compact objects are all left and right dualizable.
The reference \cite[Sections 4.3.6-4.3.8]{DG:Finiteness}, useful in the main body of the paper, discusses the consequences of rigidity for DG categories of quasi-coherent sheaves.

\sssec{} \label{sssec:pivotality}

Assume that $(\CA, m)$ is a compactly generated rigid monoidal DG category. Given a compact object $a \in \CA^{\cpt}$, we denote by ${}^\vee a$ and $a^\vee$ its left and right dual, respectively.
Following \cite[Definition 3.8]{BN}, we say that $(\CA,m)$ is \emph{pivotal} if we are given a natural identification between the functor 
$$
\CA^{\cpt} \longto \CA^{\cpt},
\hspace{.4cm}
a \mapsto (a^\vee)^\vee
$$
and the identity functor on $\CA^{\cpt}$.
One could define the notion of pivotality even when $\CA$ is not compactly generated, see \cite[Section 6.1.2]{DGCat}; we will not need this extra level of generality.

\sssec{}

Given a monoidal (cocomplete) DG category $\CA$, the symbol $\CA \mmod$ stands for the $\infty$-category of (cocomplete) DG categories with a left action of $\CA$ and (continuous) $\CA$-linear functors among them. Similar conventions are in place for the symbol $(\CA, \CB) \bbimod$.
If $\M$ and $\N$ belong to $\CA \mod$, the symbols 
$$
\Fun_\CA(\M,\N) := \Fun_{\CA \mmod}(\M,\N)
$$
both denote the DG category of $\CA$-linear (continuous) functors from $\M$ to $\N$.

\sssec{}

The symbols $\QCoh(-)$ and $\Dmod(-)$ always stand for the DG categories of quasi-coherent sheaves and $\fD$-modules (that is, we never refer to the abelian categories).
Accordingly, for $A$ a DG algebra, the notation $A \mod$ stands for the DG category of DG modules over $A$.

The pushforward and pullback functors between DG categories of sheaves are always understood in the derived sense. Fiber products and tensor products are always derived, too.

\sssec{}

The $\infty$-category of $\infty$-groupoids is denoted by $\Grpd_\infty$. The other main notations of (derived) algebraic geometry are reviewed in Section \ref{sec:IndCoh} and especially in Section \ref{ssec: review al geom}.

\ssec{Contents of the paper}

In the first section, we give an overview of the computation of $\Z(\H(\Y))$ to explain how $\vDmod(\LY)$ comes about. Then, in the second section, we review the bit of algebraic geometry that we need and discuss ind-coherent sheaves on our prestacks of interest: algebraic stacks and formal completions thereof.
In the third section, we define the DG categories $\ICoh_0(\Z^\wedge_\Y)$ in the bounded case (that is, when $\Y$ is bounded), and we extend the assignment $[\Y \to \Z] \squigto \ICoh_0(\Z^\wedge_\Y)$ to a functor out of a category of correspondences.
In the fourth section, we extend the definition of $\ICoh_0(\Z^\wedge_\Y)$ to the case when $\Y$ is possibly unbounded. In this context, $\ICoh_0(\Z^\wedge_\Y)$ lacks some of the pleasant features present in the bounded case; we discuss the features that do generalize.
Lastly, in the fifth section, we apply the theory developed in the previous sections to identify the center of $\H(\Y)$ with $\vDmod(\LY)$.

{\ssec{Acknowledgments}
I would like to thank David Ben-Zvi, Dennis Gaitsgory, Ian Grojnowski, Kobi Kremnizer, Sam Raskin, Pavel Safronov and Bertrand To\"en for several interesting conversations. I am grateful to the anonymous referee for carefully reading the paper and for suggesting several improvements.
Research partially supported by EPSRC programme grant EP/M024830/1 Symmetries and Correspondences, and by the grant NEDAG ERC-2016-ADG-741501.
}

\sec{Outline of the center computation} \label{sec:outline}

As anticipated in Theorem \ref{main-thm-intro}, the center $\Z(\H(\Y))$ is (slightly incorrectly!) equivalent to the category of $\fD$-modules on $\LY$, the loop stack of $\Y$. Such answer is literally correct whenever $L\Y$ is \emph{bounded}, but should otherwise be modified as explained below (from Section \ref{sssec:big catch} on).
In Section \ref{ssec:overview-computing-center}, we show how to guess this incorrect answer; this will also give hints as to how correct it, which we take up in Section \ref{sssec:big catch}.

\ssec{Computing the center} \label{ssec:overview-computing-center}

Let us get acquainted with $\H(\Y)$ by explaining a natural approach to computing its Drinfeld center. 
Recall that $\Y$ is a quasi-compact derived algebraic stack that is perfect, bounded and with perfect cotangent complex. We usually denote by $m: \H(\Y) \otimes \H(\Y) \to \H(\Y)$ the multiplication functor and by $m^\rev$ the reversed multiplication (obtained from $m$ by swapping the two factors).

\sssec{}

Since $\H(\Y)$ is rigid, the conservative functor $\ev$ admits a left adjoint $\ev^L$, whence $\Z := \Z(\H(\Y))$ is equivalent to the category of modules for the monad $\ev\circ \ev^L$ acting on $\H(\Y)$.
Moreover, thanks to the pivotality of $\H(\Y)$ discussed in Section \ref{ssec:defn of H}, the functor underlying $\ev \circ \ev^L$ is naturally isomorphic to 
$$
m^\rev \circ m^R: \H(\Y) \longto \H(\Y).
$$

\sssec{}

To understand the composition $m^\rev \circ m^R$, we first need to understand the multiplication $m$ more explicitly. 
For this, it is convenient to anticipate some of the functoriality on $\ICoh_0$ that we will develop. 
To start, note that the assignment
$$
[\Y \to \Z] 
\squigto 
\ICoh_0(\Z^\wedge_\Y)
$$
extends to a functor
$$
\bigt{ \Arr(\Stkperflfp)' }^\op
\longto
\DGCat,
$$
where $\Stkperflfp$ is the $\infty$-category of perfect stacks locally of finite presentation and $\Arr(\Stkperflfp)'$ is the full subcategory of $\Arr(\Stkperflfp) := \Fun(\Delta^1, \Stkperflfp)$ spanned by those arrows $\Y \to \Z$ with bounded $\Y$.

\sssec{}

We denote by $\phi^{!,0}$ the structure functor 
$$
\ICoh_0(\Z^\wedge_\Y) \longto \ICoh_0(\V^\wedge_\U)
$$
associated to a commutative square $[\U \to \V] \to [\Y \to \Z]$. By construction, $\phi^{!,0}$ is induced by the $\ICoh$-pullback along $\phi: \V^\wedge_\U \to \Z^\wedge_\Y$.

\medskip

If $\U = \Y$, then $\phi^{!,0}$ admits a left adjoint, that we denote by $\phi_{*,0}$. On the other hand, if the commutative square $[\U \to \V] \to [\Y \to \Z]$ is cartesian, then $\phi^{!,0}$ admits a continuous right adjoint, that we denote by $\phi_{?}$.

\sssec{}

By construction, the multiplication $m: \H(\Y) \otimes \H(\Y) \longto \H(\Y)$ is the composition of the functors
$$
m: \ICoh_0 \bigt{ (\Y \times \Y)^\wedge_\Y \times (\Y \times \Y)^\wedge_\Y  }
\xto{ (p_{12} \times p_{23})^{!,0}}
\ICoh_0 \bigt{ (\Y \times \Y \times \Y )^\wedge_\Y }
\xto{(p_{13})_{*,0}}
\ICoh_0 \bigt{ (\Y \times \Y)^\wedge_\Y }.
$$
By the theory sketched above, both arrows possess continuous right adjoints, whence $m^R$ is the continuous functor
$$
m^R: \ICoh_0 \bigt{ (\Y \times \Y)^\wedge_\Y }
\xto{(p_{13})^{!,0}}
\ICoh_0 \bigt{ (\Y \times \Y \times \Y )^\wedge_\Y }
\xto{ (p_{12} \times p_{23})_{?}}
\ICoh_0 \bigt{ (\Y \times \Y)^\wedge_\Y \times (\Y \times \Y)^\wedge_\Y  }.
$$

\sssec{}

To compute $m^\rev \circ m^R$, we will resort to the \emph{horocycle diagram} (see \cite{BN}) for the map $\Y \to \Y_\dR$. In general, the horocycle diagram attached to a map $\Y \to \Z$ is the following commutative diagram with cartesian squares: 
$$ 
\begin{tikzpicture}[scale=1.5]
\node (00) at (0,0) {$\Y \times_\Z \Y$};
\node (10) at (3.3,0) {$\Z \times_{\Z \times \Z} \Y$};
\node (20) at (6.6,0) {$\Z \times_{\Z \times \Z} \Z$.};
\node (01) at (0,1) {$\Y \times_{\Z} \Y \times_{\Z} \Y$};
\node (11) at (3.3,1) {$\Y \times_{\Z \times \Z} \Y$};
\node (21) at (6.6,1) {$\Y \times_{\Z \times \Z} \Z$};
\node (02) at (0,2) {$\Y \times_\Z \Y \times \Y \times_\Z \Y$};
\node (12) at (3.3,2) {$\Y \times_{\Z} \Y \times_{\Z} \Y$};
\node (22) at (6.6,2) {$\Y \times_\Z \Y$};
\path[->,font=\scriptsize,>=angle 90]
(12.east) edge node[below] {$p_{13}$} (22.west);
\path[->,font=\scriptsize,>=angle 90]
(11.east) edge node[below] {$ $} (21.west);
\path[->,font=\scriptsize,>=angle 90]
(10.east) edge node[below] {$ $} (20.west);
\path[<-,font=\scriptsize,>=angle 90]
(02.east) edge node[below] {$p_{12} \times p_{23}$} (12.west);
\path[<-,font=\scriptsize,>=angle 90]
(01.east) edge node[below] {$ $} (11.west);
\path[<-,font=\scriptsize,>=angle 90]
(00.east) edge node[below] {$ $} (10.west);
\path[<-,font=\scriptsize,>=angle 90]
(02.south) edge node[right] {$p_{23} \times p_{12}$} (01.north);
\path[<-,font=\scriptsize,>=angle 90]
(12.south) edge node[below] {$ $} (11.north);
\path[<-,font=\scriptsize,>=angle 90]
(22.south) edge node[below] {$ $} (21.north);
\path[->,font=\scriptsize,>=angle 90]
(01.south) edge node[right] {$p_{13}$} (00.north);
\path[->,font=\scriptsize,>=angle 90]
(11.south) edge node[below] {$ $} (10.north);
\path[->,font=\scriptsize,>=angle 90]
(21.south) edge node[below] {$ $} (20.north);
\end{tikzpicture}
$$

\sssec{}

Applied to the case $\Z = \Y_\dR$, and using the tautological isomorphisms 
$$
(\Y \times \Y)^\wedge_\Y \simeq \Y \times_{\Y_\dR} \Y,
\hspace{.4cm}
\pt^\wedge_\Z \simeq \Z_\dR,
$$
the above diagram becomes
$$ 
\begin{tikzpicture}[scale=1.5]
\node (00) at (0,0) {$(\Y \times \Y)^\wedge_\Y$};
\node (10) at (3.3,0) {$\Y^\wedge_{L\Y}$};
\node (20) at (6.6,0) {$ \pt^\wedge_{L\Y}$,};
\node (01) at (0,.9) {$(\Y \times \Y \times \Y)^\wedge_\Y$};
\node (11) at (3.3,.9) {$(\Y \times \Y)^\wedge_{L\Y}$};
\node (21) at (6.6,.9) {$\Y^\wedge_{L\Y}$};
\node (02) at (0,1.8) {$(\Y \times \Y)^\wedge_\Y \times (\Y \times \Y)^\wedge_\Y$};
\node (12) at (3.3,1.8) {$(\Y \times \Y \times \Y)^\wedge_\Y$};
\node (22) at (6.6,1.8) {$(\Y \times \Y)^\wedge_\Y$};
\path[->,font=\scriptsize,>=angle 90]
(12.east) edge node[below] {$p_{13}$} (22.west);
\path[->,font=\scriptsize,>=angle 90]
(11.east) edge node[below] {$t$} (21.west);
\path[->,font=\scriptsize,>=angle 90]
(10.east) edge node[below] {$v$} (20.west);
\path[<-,font=\scriptsize,>=angle 90]
(02.east) edge node[below] {$P$} (12.west);
\path[<-,font=\scriptsize,>=angle 90]
(01.east) edge node[below] {$r$} (11.west);
\path[<-,font=\scriptsize,>=angle 90]
(00.east) edge node[below] {$s$} (10.west);
\path[<-,font=\scriptsize,>=angle 90]
(02.south) edge node[right] {$P'$} (01.north);
\path[<-,font=\scriptsize,>=angle 90]
(12.south) edge node[right] {$r'$} (11.north);
\path[<-,font=\scriptsize,>=angle 90]
(22.south) edge node[right] {$s$} (21.north);
\path[->,font=\scriptsize,>=angle 90]
(01.south) edge node[right] {$p_{13}$} (00.north);
\path[->,font=\scriptsize,>=angle 90]
(11.south) edge node[right] {$t'$} (10.north);
\path[->,font=\scriptsize,>=angle 90]
(21.south) edge node[right] {$v$} (20.north);
\end{tikzpicture}
$$
where we have set $P:=p_{12} \times p_{23}$ and $P' = p_{23} \times p_{12}$ to save space. Observe that $m^\rev$ is the composition
$$
m^\rev:
 \ICoh_0 \bigt{ (\Y \times \Y)^\wedge_\Y \times (\Y \times \Y)^\wedge_\Y  }
\xto{ (P')^{!,0}}
\ICoh_0 \bigt{ (\Y \times \Y \times \Y )^\wedge_\Y }
\xto{(p_{13})_{*,0}}
\ICoh_0 \bigt{ (\Y \times \Y)^\wedge_\Y }.
$$

\sssec{Notation} \label{ssec:notation regarding LY}

We have denoted by $\LY:= \Y \times_{\Y \times \Y} \Y$ the \emph{loop stack} of $\Y$: the fiber product of the diagonal map $\Delta: \Y \to \Y \times \Y$ with itself. There are two standard maps: the insertion of \virg{constant loops} $\iota: \Y \to \LY$ and the projection $\pi:\LY \to \Y$.

\sssec{}

Assume for the time being that $L\Y$ is bounded, so that the DG category $\ICoh_0(\W^\wedge_{L\Y})$ makes sense for any $L\Y \to \W$. Then we can consider the diagram
\begin{equation} \label{diag:crazy-diagram}
\begin{tikzpicture}[scale=1.5]
\node (00) at (0,0) {$\ICoh_0( (\Y \times \Y)^\wedge_\Y)$};
\node (10) at (4,0) {$\ICoh_0(\Y^\wedge_{L\Y})$};
\node (20) at (8,0) {$ \ICoh_0 (\pt^\wedge_{L\Y}).$};

\node (01) at (0,1.2) {$\ICoh_0( (\Y \times \Y \times \Y)^\wedge_\Y) $};
\node (11) at (4,1.2) {$\ICoh_0( (\Y \times \Y)^\wedge_{L\Y})$};
\node (21) at (8,1.2) {$\ICoh_0(\Y^\wedge_{L\Y})$};

\node (02) at (0,2.4) {$\ICoh_0((\Y \times \Y)^\wedge_\Y \times (\Y \times \Y)^\wedge_\Y)$};
\node (12) at (4,2.4) {$\ICoh_0((\Y \times \Y \times \Y)^\wedge_\Y)$};
\node (22) at (8,2.4) {$\ICoh_0((\Y \times \Y)^\wedge_\Y)$};

\path[->,font=\scriptsize,>=angle 90]
(12.east) edge node[below] {$(p_{13})_{*,0}$} (22.west);
\path[->,font=\scriptsize,>=angle 90]
(11.east) edge node[below] {$t_{*,0}$} (21.west);
\path[->,font=\scriptsize,>=angle 90]
(10.east) edge node[below] {$v_{*,0}$} (20.west);
\path[->,font=\scriptsize,>=angle 90]
(02.east) edge node[below] {$P^{!,0}$} (12.west);
\path[->, font=\scriptsize,>=angle 90]
(01.east) edge node[below] {$r^{!,0}$} (11.west);
\path[->,font=\scriptsize,>=angle 90]
(00.east) edge node[below] {$s^{!,0}$} (10.west);
\path[<-,font=\scriptsize,>=angle 90]
(02.south) edge node[right] {$(P')_{?}$} (01.north);
\path[<-,font=\scriptsize,>=angle 90]
(12.south) edge node[right] {$(r')_?$} (11.north);
\path[<-,font=\scriptsize,>=angle 90]
(22.south) edge node[right] {$(s')_?$} (21.north);
\path[<-,font=\scriptsize,>=angle 90]
(01.south) edge node[right] {$(p_{13})^{!,0}$} (00.north);
\path[<-,font=\scriptsize,>=angle 90]
(11.south) edge node[right] {$(t')^{!,0}$} (10.north);
\path[<-,font=\scriptsize,>=angle 90]
(21.south) edge node[right] {$(v')^{!,0}$} (20.north);
\end{tikzpicture}
\end{equation}

\sssec{}

As an application of the functoriality of $\ICoh_0$ (specifically, the base-change isomorphisms established in Sections \ref{sec:ICoh-zero-bounded} and \ref{sec:beyond-ev-coc}), one easily proves that these four squares are commutative. It follows that the monad $\ev \circ \ev^L \simeq m^\rev \circ m^R$ is isomorphic (as a plain functor) to the monad of the adjunction
\begin{equation} \label{adj:introduction-center-to-H-and-back}
\begin{tikzpicture}[scale=1.5]
\node (a) at (0,1) {$ \H(\Y) = \ICoh_0( (\Y \times \Y)^\wedge_\Y  )$};
\node (b) at (4,1) {$\ICoh_0(\pt^\wedge_{L\Y}) $.};
\path[->,font=\scriptsize,>=angle 90]
([yshift= 1.5pt]a.east) edge node[above] {$\beta: = v_{*,0} \circ s^{!,0}$} ([yshift= 1.5pt]b.west);
\path[->,font=\scriptsize,>=angle 90]
([yshift= -1.5pt]b.west) edge node[below] {$\beta^R := s_? \circ v^{!,0} $} ([yshift= -1.5pt]a.east);
\end{tikzpicture}
\end{equation}
We emphasize again that $\ICoh_0(\pt^\wedge_{\LY}) \simeq \Dmod(\LY)$, with $\ICoh_0(\pt^\wedge_{\LY})$ being well-defined thanks to the boundedness of $\LY$.

\sssec{}

It is not hard to check that the right adjoint in (\ref{adj:introduction-center-to-H-and-back}) is conservative. Moreover, one can show that the isomorphism $\ev \circ \ev^L \simeq \beta^R \circ \beta$ preserves the monad structures. All in all, we obtain $\Z(\H(\Y)) \simeq \Dmod(L\Y)$ whenever $\LY$ is bounded.
Moreover, in this case, the functor $\zeta: \Dmod(\Y) \to \Z(\H(\Y))$ making the diagram (\ref{diag:Dmod Y mapping to center}) commutative is simply the pushforward $\iota_{*,\dR}:\Dmod(\Y) \to \Dmod(\LY)$ along $\iota: \Y \to \LY$.

\begin{rem}
The center of a monoidal DG category comes equipped with a monoidal structure. In the case at hand, the monoidal structure on $\Dmod(\LY)$ is the one induced by \emph{composition of loops}, that is, by the correspondence
$$ 
\LY \times \LY \longleftarrow \LY \times_\Y \LY \longto \LY.
$$
We will not use such monoidal structure in this paper and therefore do not discuss it further.
\end{rem}

\begin{example}

As an example of the above computation, consider the case where $\Y = BG$, the classifying stack of an affine algebraic group $G$. Then $\LY$ is isomorphic to the adjoint quotient $G/G$, which is bounded (in fact, smooth). 
By \cite[Section 2]{Be}, we know that 
$$
\H(BG) \simeq \ICoh(G \backslash G_\dR /G)
$$
is the monoidal DG category of \emph{Harish-Chandra bimodules} for $G$. The theorem states that its center is equivalent to $\Dmod(G/G)$.

\end{example}

\ssec{Beyond the bounded case} \label{sssec:big catch}

The issue with the above argument leading to $\Z(\H(\Y)) \simeq \Dmod(\LY)$ is that boundedness of $\LY$ is rare (for instance, see Remark \ref{rem:LY boundd iff Y is smooth} below) and, for $\LY$ unbounded, the entire bottom-right square of (\ref{diag:crazy-diagram}) makes no sense.

\medskip

To remedy this, we need to search for an extension of the definition of $\ICoh_0(\Z^\wedge_\X)$ to the case of unbounded $\X$. Such definition must come with functors making the four squares of (\ref{diag:crazy-diagram}) commutative: then the above argument would go through and would show that the center of $\H(\Y)$ is equivalent to $\ICoh_0(\pt^\wedge_{L\Y})$, \emph{whatever the latter means}.

\sssec{}

To concoct this more general definition, we will try to adapt (\ref{eqn:def of ICOHzero with universal enveloping alg}), that is, we will try to define $\ICoh_0(\Z^\wedge_\X)$ as the DG category of modules over a monad acting on $\QCoh(\X)$.

\medskip

The most naive attempt is to take the same formula as in (\ref{eqn:def of ICOHzero with universal enveloping alg}); indeed, the expression $\Phi_\X \circ \U(\Tang_{\X/\Z}) \circ \Upsilon_\X$ still makes sense as a monad. This attempt fails, however, as such monad is discontinuous in general (indeed $\Phi_\X$ is continuous iff $\X$ is bounded).

\sssec{}

To fix such discontinuity, we could restrict the functor in question to $\Perf(\X)$, and then ind-complete. Let us denote the resulting (continuous) functor by 
\begin{equation} \label{eqn:renormalization of universal enveloping alg}
(\Phi_\X \U(\Tang_{\X/\Z}) \Upsilon_\X )^\ren:
\QCoh(\X) \longto \QCoh(\X).
\end{equation}
We claim that such definition is not the right one either. To see this, look at the result of this operation in the case where $\Tang_{\X/\Z}$ is an abelian Lie algebra in $\ICoh(\X)$, so that 
$$
\U(\Tang_{\X/\Z}) \simeq \Sym(\Tang_{\X/\Z}) \stackrel ! \otimes -:
\ICoh(\X) \longto \ICoh(\X).
$$
In such simple case, we expect our monad to be the functor of tensoring with the symmetric algebra of $\Tang^\QCoh_{\X/\Z}$. What we get instead is the functor of tensoring with
$$
(\Phi_\X \Sym(\Tang_{\X/\Z}) \Upsilon_\X )^\ren (\O_X)
\simeq
\Phi_\X \Upsilon_\X  \bigt{ \Sym^{\QCoh}(\Tang^\QCoh_{\X/\Z}) }.
$$
This object is the \emph{convergent renormalization}\footnote{See Section \ref{sssec:conv renormalization} for the definition in the scheme case (the stack case is similar).} 
of $\Sym(\Tang^\QCoh_{\X/\Z})$, which is different from $\Sym(\Tang^\QCoh_{\X/\Z})$ as soon as the latter is not bounded above in the t-structure of $\QCoh(\X)$.
Working with such convergent renormalizations is not pleasant: in fact, all the base-change results that we need fail.

\sssec{}

We can however turn the above failure into a positive observation. 
Note that $\Sym(\TangQ_{\X/\Z})$ and $\Sym(\Tang_{\X/\Z})$ are filtered by $\Sym^{\leq n}(\TangQ_{\X/\Z})$ and $\Sym^{\leq n}(\Tang_{\X/\Z})$, respectively. Since $\Sym^{\leq n}(\TangQ_{\X/\Z})$ is perfect for each $n$ (in particular, bounded above), the renormalization procedure (\ref{eqn:renormalization of universal enveloping alg}) applied to 
$$
\Sym^{\leq n}(\Tang_{\X/\Z}) \stackrel ! \otimes - :
\ICoh(\X) \longto \ICoh(\X)
$$ 
yields precisely the functor
$$
\Sym^{\leq n}(\Tang^\QCoh_{\X/\Z}) \otimes -:
\QCoh(\X) \longto \QCoh(\X).
$$
We thus have
\begin{equation}
\Sym (\TangQ_{\X/\Z})
\simeq 
\uscolim{n \geq 0 }
\bigt{
\Phi_\X \circ
\Sym(\Tang_{\X/\Z})^{\leq n}
\circ
\Upsilon_\X
}^\ren.
\end{equation}

\sssec{}

The general situation is analogous, thanks to the existence of a canonical filtration of $\U(\Tang_{\X/\Z})$, the PBW filtration, which specializes to the above in the case of abelian Lie algebras. See \cite[Volume 2, Chapter 9, Section 6]{Book}.
Rather than renormalizing $\U(\Tang_{\X/\Z})$ itself, we renormalize each piece of the filtration and then put them together. In symbols, we define
\begin{equation}
\UQ(\Tang_{\X/\Z})
:=
\uscolim{n \geq 0 } \bigt{
\UQ(\Tang_{\X/\Z})^{\leq n}}
\end{equation}
where
\begin{equation}
\UQ(\Tang_{\X/\Z})^{\leq n} :=
\bigt{
\Phi_\X \circ
\U(\Tang_{\X/\Z})^{\leq n}
\circ
\Upsilon_\X
}^\ren
\end{equation}
is the only continuous functor whose restriction to $\Perf(\X)$ is given by $\Phi_\X
\U(\Tang_{\X/\Z})^{\leq n}
\Upsilon_\X$.

\sssec{}

In Section \ref{ssec: defin of Icohzero unbdd case}, we will prove that $\UQ(\Tang_{\X/\Z})$ comes equipped with the structure of a monad on $\QCoh(\X)$. This allows to extend the definition of $\ICoh_0$ to the case of $\X$ unbounded as
$$
\ICoh_0(\Z^\wedge_\X)
:=
\UQ(\Tang_{\X/\Z}) \mod (\QCoh(\X)),
$$
see Definition \ref{defn:ICohzero unbounded case} where this is done officially.
In the later parts of Section \ref{sec:beyond-ev-coc}, we show that the assignment
$$
[\X \to \Z]
\squigto
\ICoh_0(\Z^\wedge_\X)
:=
\UQ(\Tang_{\X/\Z}) \mod (\QCoh(\X))
$$
possesses all the functorialities that we need for the computation of $\Z(\H(\Y))$. 
Specifically, our main Theorems \ref{thm:Main-first-part} and \ref{MAIN-thm} will assert that
$$
\Z(\H(\Y))
\simeq
\ICoh_0(\pt^\wedge_\LY) =: \vDmod(\LY),
$$
sitting in a monadic adjunction $\H(\Y) \rightleftarrows \Z(\H(\Y))$ defined exactly as in (\ref{adj:introduction-center-to-H-and-back}).

\sssec{} \label{sssec:list-of-equivalences}

Let us comment on the relationship between $\Dmod(\LY)$ and $\vDmod(\LY)$ in the general case. As shown in Section \ref{sssec: passage from left to right unbdd setting}, there always exists a tautological functor
\begin{equation} \label{eqn:tau-arrow}
\Upsilon^\fD_{\LY/\pt}: 
\vDmod(\LY)
\longto
\Dmod(\LY),
\end{equation}
which can be regarded as a passage from left to right $\fD$-modules in our setting.
Contrarily to the usual setting, this functor is not an equivalence in general (see Corollary \ref{cor:vDmod on Yn} for an example).

\begin{rem} \label{rem:LY boundd iff Y is smooth}
For $\Y$ quasi-smooth, the following assertions are equivalent:
\begin{enumerate}
\item
$\Y$ is smooth;
\item
$\LY$ is quasi-smooth;
\item
$\LY$ is bounded;
\item
the above functor $\Upsilon^\fD_{\LY/\pt}$ is an equivalence.
\end{enumerate} 
The implications $(1) \Rightarrow (2) \Rightarrow (3)$ are obvious. 
The implication $(3) \Rightarrow (4)$ will be a consequence of the fact that $\UQ(\Tang_\X) \simeq \Phi_\X \U(\Tang_\X) \Upsilon_\X$ whenever $\X$ is bounded, combined with Example \ref{example: D-mods as Icohzero -intro}.
The implication  $(4) \Rightarrow (1)$ not needed in the present paper, will be discussed elsewhere.
\end{rem}

\section{Ind-coherent sheaves on formal completions} \label{sec:IndCoh}

This section is devoted to recalling the theory of ind-coherent sheaves. We are particularly interested in ind-coherent sheaves on formal completions of perfect stacks. The main references are \cite{ICoh} and \cite{Book}.

\ssec{Some notions of derived algebraic geometry} \label{ssec: review al geom}

\sssec{}

Denote by $\Aff$ the $\infty$-category of affine (DG) schemes over $\kk$. We usually omit the adjective \virg{DG}, so our schemes (and then prestacks) are derived unless stated otherwise.
An affine scheme $\Spec(A)$ is \emph{bounded} if $H^i(A)$ is zero except for finitely many $i \leq 0$.

Denote by $\Aff_\aft \subseteq \Aff$ the $\infty$-category of affine schemes \emph{almost of finite type}: by definition, $\Spec(A)$ belongs to $\Aff_\aft$ if and only if $H^0(A)$ is of finite type over $\kk$ and each cohomology $H^{i}(A)$ is finitely generated as a module over $H^0(A)$.

We shall often consider affine schemes that are both bounded and almost of finite type: we denote by $\Affaftevcoc$ the $\infty$-category they form.

Denote by $\Sch_{\aft}$ the $\infty$-category of quasi-compact (DG) schemes almost of finite type. Such a scheme $S \in \Sch_{\aft}$ is bounded if it is Zariski locally so (hence we reduce to the definition for affine schemes). Moreover, $S$ is the colimit of its truncations $S^{\leq n}$, as $n \to \infty$.

\sssec{}

A prestack is an arbitrary functor $\Y: \Aff^\op \to \inftyGrpd$. Denote by $\PreStk$ the $\infty$-category of prestacks. 
Important for us is the subcategory $\PreStk_\laft$ of prestacks that are \emph{locally almost of finite type (laft)}, see \cite[Volume 1, Chapter 2, Section 1.7]{Book}.
Rather than the actual definition, what we need to know about $\PreStk_\laft$ are its following properties:
\begin{itemize}
\item
it is closed under fiber products;

\item
it is closed under the operation $\Y \squigto \Y_\dR$ (the \emph{de Rham} prestack of $\Y$);
\item
it contains all perfect stacks (see below).
\end{itemize}

\begin{example}
In particular, for $\Y \to \Z$ in $\PreStk_\laft$, the \emph{formal completion of $\Y$ in $\Z$}, i.e. the fiber product
$$
\Z^\wedge_\Y := \Z \ustimes{\Z_\dR} \Y_\dR,
$$
is laft.
\end{example}

\begin{rem}
The point of the condition $\laft$ is that $\PreStk_\laft$ is equivalent to the $\infty$-category of arbitrary functors from $(\Affevcocaft)^\op$ to $\inftyGrpd$.
\end{rem} 

\sssec{Algebraic stacks} \label{sssec:stacks reminders}

We will be quite restrictive on the kinds of stacks that we deal with. Namely, we denote by $\Stk \subset \PreStk$ the full subcategory consisting of those (quasi-compact) algebraic stacks with affine diagonal and with an atlas in $\Aff_\aft$. We will just call them \emph{stacks}. 

\sssec{}

We say that $\Y \in \Stk$ is \emph{bounded} if for some (equivalently: any) atlas $U \to \Y$, the affine scheme $U$ is bounded.
Denote by $\Stkevcoc \subset \Stk$ the full subcategory of bounded stacks. It is closed under products, but not under fiber products.  
We say that a map $\Y \to \Z$ in $\Stk$ is \emph{bounded} if, for any $S \in (\Affevcocaft)_{/\Z}$, the fiber product $S \times_\Z \Y$ belongs to $\Stkevcoc$.

\medskip

Following \cite{BFN}, we say that  $\Y \in \Stk$ is \emph{perfect} if the DG category $\QCoh(\Y)$ is compactly generated by its subcategory $\Perf(\Y)$ of perfect objects.

\medskip

We say that  $\Y \in \Stk$ is \emph{locally finitely presented (lfp)} if its cotangent complex $\LL_{\Y}  \in \QCoh(\Y)$ is perfect. In that case, we denote by $\TangQ_{\Y} \in \Perf(\Y)$ its monoidal dual.

\medskip

We denote by $\Stkperfevcoclfp \subseteq \Stk$ the full subcategory of stacks that are perfect, bounded and locally of finite presentation. Similarly, the notations $\Stkperflfp$ and $\Stkperf$ have the evident meaning. By \cite[Proposition 3.24]{BFN}, $\Stkperf$ is closed under fiber products (this is because our stacks have affine diagonal by assumption).

\ssec{Ind-coherent sheaves on schemes}

This section is a recapitulation of parts of \cite{ICoh}, \citep{Book} and \cite{DG:Finiteness}. It is included for the reader's convenience and to fix the notation.

\sssec{}

For a scheme $S \in \Sch_\aft$, consider the non-cocomplete DG category $\Coh(S)$, the DG category of cohomologically bounded complexes with coherent cohomology. 
We define $\IndCoh(S) := \Ind(\Coh(S))$ to be its ind-completion. The latter comes equipped with an action of $\QCoh(S)$ and a tautological $\QCoh(S)$-linear functor $\Psi_S: \IndCoh(S) \to \QCoh(S)$.

\begin{prop}
$\Psi_S$ is an equivalence iff $S$ is a smooth classical scheme.
\end{prop}

For the proof, see \cite[Lemma 1.1.6 and Proposition 1.4.6]{ICoh}.

\sssec{} \label{sssec:Psi for ICoh}

Boundedness of $S$ is equivalent to $\Psi_S$ having a fully faithful left adjoint $\Xi_S: \QCoh(S) \to \IndCoh(S)$. Indeed, for $\Perf(S)$ to be contained in $\Coh(S)$, the structure sheaf has to be bounded. (When $\Xi_S$ exists, it is automatically $\QCoh(S)$-linear.)
Thus, for bounded schemes, $\IndCoh$ is an enlargement of $\QCoh$; more precisely, $\Psi$ is a colocalization. 

\begin{rem}
For unbounded schemes, the situation is unwieldy. For instance, consider the scheme $S= \Spec (\Sym V^*[2])$, with $V$ a finite dimensional ordinary vector space over $\kk$.
In this case, $\Psi_S$ is fully faithful (but not an equivalence): indeed, the augmentation module Karoubi generates $\Coh(S)$ and it is perfect by Example \ref{example:Koszul duality} below.
\end{rem}

\sssec{} \label{sssec:ICOH functoriality on schemes}

The assignment $S \squigto \ICoh(S)$ underlies an $\2$-functor
$$
\Corr(\Sch_\aft)_{\all;\all}^{\proper} \longto \DGCat^{2-\Cat},
$$
where $\DGCat^{2-\Cat}$ denotes the $\2$-category of DG categories and the notation $\Corr(\C)_{\vert;\horiz}^{\adm}$ is taken from \cite[Volume 1, Chapter 7]{Book}.
In any case, the above $\2$-functor is a fancy way to encode the following data:
\begin{itemize}
\item
 for any map $f: S \to T$ in $\Sch_\aft$, we have a push-forward functor $f_*^\IndCoh: \ICoh(S) \to \ICoh(T)$ and a pullback functor  $f^!:\IndCoh(T) \to \IndCoh(S)$;
\item
 push-forwards and pull-backs are equipped with base-change isomorphisms along Cartesian squares;
 \item
 if $f$ is proper, then $f_*^\ICoh$ is \emph{left adjoint} to $f^!$.
\end{itemize}

\sssec{}

The action of $\QCoh(S)$ on $\ICoh(S)$ and the canonical object $\omega_S := (p_S)^!(\kk) \in \ICoh(S)$ yield the functor 
$$
\Upsilon_S := - \otimes \omega_S 
:
\QCoh(S) \longto \ICoh(S).
$$
The latter admits a continuous right adjoint if and only if $\omega_S \in \Coh(S)$, which in turn is equivalent to $S$ being bounded. Since such right adjoint does not have a notation in the original paper \cite{ICoh}, we shall call it $\Phi_S$.

\begin{example} \label{example:Koszul duality}
Let $Y_n = \Spec (\Sym V^*[n])$, with $V$ a nonzero finite dimensional ordinary vector space and $n \geq 1$. Then 
\begin{equation} \label{eqn:QCoh on Yn }
\QCoh(Y_n) \simeq  (\Sym V^*[n]) \mod,
\end{equation}
tautologically. On the other hand, we have the \emph{Koszul duality} equivalence
\begin{equation} \label{eqn:ICoh on Yn }
\ICoh(Y_n) \simeq  \Sym (V[-n-1])\mod.
\end{equation}
Indeed, $\Coh(Y_n)$ is generated by a single object, the augmentation module, and a standard Koszul resolution yields
$$
\CH om_{\Sym V^*[n]}(\kk, \kk) \simeq \Sym (V[-n-1]).
$$
Under \eqref{eqn:QCoh on Yn } and \eqref{eqn:ICoh on Yn }, the functor $\Upsilon_{Y_n}$ is the tensor product with the augmentation $\kk$, the latter viewed as a $(\Sym (V^*[n]), \Sym (V[-n-1])$-bimodule.

\end{example}

\sssec{} \label{sssec:conv renormalization}

As in \cite[Lemma F.5.8]{AG}, one shows that the composition $\Phi_S \circ \Upsilon_S$ is the functor of \emph{convergent renormalization}, computed explicitly as follows:
$$
M \squigto 
M^\conv:= \lim_{n \geq 0} \, 
(i_n)_* (i_n)^* M
\in \QCoh(S),
$$
where $i_n: S^{\leq n} \hto S$ is the inclusion of the $n$-connective truncation of $S$.
This shows that $\Upsilon_S$ is fully faithful when $S$ is bounded: indeed, in that case $S \simeq S^{\leq n}$ for some $n$ and the limit stabilizes.

\sssec{} \label{sssec: Ups fully fiath on perf}

More generally, regardless of whether $S$ is bounded or not, the unit of the adjunction $M \to M^\conv$ is the identity on $\Perf(S)$ and more generally on $\QCoh(S)^-$ (the full subcategory of $\QCoh(S)$ consisting of objects bounded from above in the usual t-structure on $\QCoh(S)$).
In particular, whether $S$ is bounded or not, we can consider $\Upsilon_S(\Perf(S))$ and $\Upsilon_S(\QCoh(S)^-)$ as (non-cocomplete) full subcategories of $\ICoh(S)$.

\begin{example}
The functor $\Upsilon_S$ fails to be fully faithful for the simplest unbounded scheme $S = \Spec \kk[u]$ (here $u$ is a variable in cohomological degree $-2$). To see this, it suffices to prove that the functor $\Phi_S \circ \Upsilon_S : M \squigto M^\conv$ is not conservative: indeed, the $\kk[u]$-module $\kk[u, u^{-1}]$ is obviously nonzero and yet $(\kk[u, u^{-1}])^\conv \simeq 0$ for degree reasons.
\end{example}

\begin{rem}
The same idea should prove that $\Upsilon_S$ is not fully faithful as soon as $S$ is unbounded.
\end{rem}

\sssec{}

Let $\CA$ be a monoidal DG category acting on $\C$. Recall our conventions on DG categories explained in Section \ref{ssec:conventions}. For $c \in \C$, consider the \emph{possibly discontinuous} functor
$$
\ul\Hom_\CA(c,-): \C \longto \CA,
$$
the right adjoint to the functor of action on $c$. 
For istance, $\Phi_S \simeq \ul\Hom_{\QCoh(S)} (\omega_S,-)$, where we are of course using the standard action of $\QCoh(S)$ on $\ICoh(S)$.

\sssec{} \label{sssec:Serre duality}

Consider instead the functor
$$
\ul\Hom_{\QCoh(S)}(-, \omega_S): \Coh(S)^\op \longto \QCoh(S).
$$
It is shown in \cite[Lemma 9.5.5]{ICoh} that the above yields an involutive equivalence
$$
\bbD^\Serre_S: \Coh(S)^\op \longto \Coh(S),
$$
which is the usual Serre duality. Such equivalence exhibits $\ICoh(S)$ as its own dual.

\sssec{} \label{sssec:Serre duality on Coh and Perf}

For bounded $S$, it is easy to see that $\bbD^\Serre_S$ exhanges the two subcategories $\Perf(S) \subseteq \Coh(S)$ and $\Upsilon_S(\Perf(S)) \subseteq \Coh(S)$. Indeed, the definition of $\bbD^\Serre_S$ immediately yields that
$$
\bbD^\Serre_S (\Upsilon_S(P)) \simeq \bbD^{\QCoh}_S(P),
\hspace{.4cm} P \in \Perf(S),
$$
where $\bbD_S^\QCoh$ is the standard duality involution on $\Perf(S)$. 

\begin{rem} \label{rem: Serre duality extended}

The Serre duality isomorphism
\begin{equation}\label{eqn:Serre duality generalized}
\CH om_{\ICoh(S)}(M,N)
\simeq
\CH om_{\ICoh(S)} (\bbD^\Serre_S(N), \bbD^\Serre_S(M)),
\hspace{.4cm} \mbox{for $M, N \in \Coh(S)$,}
\end{equation}
holds true slightly more generally. 
Indeed, observe that $\bbD^\Serre_S := 
\ul\Hom_{\QCoh(S)}(-, \omega_S)$ extends to a contravariant functor of $\ICoh(S)$ that sends colimits to limits. Then it is easy to check that \eqref{eqn:Serre duality generalized} is valid for $N \in \Coh(S)$ and $M \in \ICoh(S)$ arbitrary. 
\end{rem}

\ssec{Ind-coherent sheaves on laft prestacks}

\sssec{}

Recall that a laft prestacks are by definition presheaves of spaces on $\Affaftevcoc$.  %
Ind-coherent sheaves are defined for arbitrary laft prestacks: one simply right-Kan extends the functor 
$$
\ICoh^!: (\Affevcocaft)^\op  \to \DGCat
$$
along $\Affevcocaft \hto \PreStk_\laft$.
In particular, for any $\Y \in \PreStk_\laft$, the $!$-pullback along $\Y \to \pt$ yields a canonical object $\omega_\Y \in \IndCoh(\Y)$. 
Since, as in the case of schemes, $\IndCoh(\Y)$ admits an action of $\QCoh(\Y)$, we have the canonical functor 
$$
\Upsilon_\Y: \QCoh(\Y) \longto \IndCoh(\Y)
$$
corresponding to the action on $\omega_\Y$.


\sssec{}

Let us now discuss ind-coherent sheaves on stacks (recall our convention of the term \virg{stack}: our stacks are all algebraic, quasi-compact, with affine diagonal, and laft).

\begin{prop}
For $\Y \in \Stk$, the DG category $\ICoh(\Y)$ is compactly generated by $\Coh(\Y)$, and self-dual by Serre duality.
\end{prop}

\begin{prop} [Corollary 4.3.8 of \cite{DG:Finiteness}]
If $\Y \in \Stk$ is bounded, $\QCoh(\Y)$ is rigid and in particular self-dual. Therefore, an object of $\QCoh(\Y)$ is compact if and only if it is perfect. 
\end{prop}

In the situation above, it is not known whether $\QCoh(\Y)$ is compactly generated. However, such condition (almost always satisfied in practice!) is convenient for many manipulations, hence we include it \virg{by hand} in our main results by requiring our stacks to be perfect.

\ssec{Base-change}

Next, one would like to define push-forward functors of ind-coherent sheaves on laft prestacks, together with base-change isomorphisms. For this, one needs to find the correct $\infty$-category of correspondences of laft prestacks.
Indeed, unlike $!$-pullbacks, push-forwards are not to be expected to be defined (and continuous) for all maps between laft prestacks.

\sssec{}

The situation is neatly summarized by the following theorem, see \cite[Volume 2, Chapter 3, Theorem 5.4.3]{Book}.

\begin{thm} \label{thm:ICOH-corresp-book}
The assignment $\Y \squigto \ICoh(\Y)$ extends uniquely to an $\2$-functor 
\begin{equation} \label{eqn:ICOH-out-of-Corr-Prestacks-laft}
\ICoh: 
\Corr(\PreStk_\laft)_{\mathit{ind \on- inf \on- schem}; \all}^{\mathit{ind \on- inf \on- schem \& ind \on- proper}} 
\longto 
\DGCat^{2-\Cat}, 
\end{equation}
where the abbreviation $\virg{\mathit{ind \on- inf \on- schem}}$ stands for \emph{ind-inf-schematic}.
\end{thm}
Translated into plain language, the theorem states that:
\begin{itemize}
\item
$(*,\ICoh)$-push-forwards are defined only for \emph{ind-inf-schematic} maps and have base-change isomorphisms against $!$-pullbacks;
\item
 if $f$ ind-inf-schematic and ind-proper, then $f_*^\ICoh$ is \emph{left} adjoint to $f^!$.
\end{itemize}  

\sssec{}

Luckily, in this paper we do not need such high level of generality (whence, we will not need to recall the definitions of those words). We only need to be aware of the following fact: if $\X \to \Y \to \Z$ is a string in $\Stk$ with $\X \to \Y$ schematic (and proper), then the resulting map $\Z^\wedge_\X \to \Z^\wedge_\Y$ is ind-inf-schematic (and ind-proper).
This yields:

\begin{lem}
The assignment
$$
\Arr(\Stk)
\longto
\PreStk_\laft, 
\hspace{.4cm}
[\Y \to \Z]
\squigto \Z^\wedge_\Y
$$
extends to an $\2$-functor 
\begin{equation} \label{eqn:funny infty 2 functor}
\Corr(\Arr(\Stk))_{\schem; \all}^{\schem \& \proper}
\longto
\Corr(\PreStk_\laft)_{\mathit{ind \on- inf \on- schem}; \all}^{\mathit{ind \on- inf \on- schem \& ind \on- proper}},
\end{equation}
where a morphism $[\X \to \Y] \to [\U \to \V]$ in $\Arr(\Stk)$ is said to be schematic (or proper) if so is the map $\X \to \U$. 
\end{lem}

\begin{proof}
The only thing to check is the $1$-categorical composition of correspondences. This boils down to the following fact (whose proof, left to the reader, is a diagram chase): for a cospan
$$
[\X \to \Y] \to [\U \to \V] \leftto [\X' \to \Y']
$$
in $\Arr(\PreStk)$, the resulting map
$$
\Y^\wedge_\X \ustimes{\V^\wedge_\U} (\Y')^\wedge_{\X'}
\longleftarrow
(\Y \times_\V \Y')^\wedge_{\X \times_{\U} \X'}
$$
is an isomorphism.
\end{proof}

\begin{thm} \label{thm:ICOH-corresp-formal completions}
The assignment $[\Y \to \Z] \squigto \ICoh(\Z^\wedge_\Y)$ underlies an $\2$-functor 
\begin{equation} \label{eqn:ICOH-out-of-Corr-formal-completions}
\ICoh: 
\Corr(\Arr(\Stk))_{\schem; \all}^{\schem \& \proper}
\longto 
\DGCat^{2-\Cat}.
\end{equation}
\end{thm}

\begin{proof}
This $\2$-functor is the composition of \eqref{eqn:funny infty 2 functor} with \eqref{eqn:ICOH-out-of-Corr-Prestacks-laft}.
\end{proof}

\begin{example}
For $\Y \in \Stk$, consider the tautological proper arrow
$[\Y \to \Y] \to [\Y \to \pt]$.
Under the equivalence $\ICoh(\pt^\wedge_\Y) \simeq \Dmod(\Y)$, the resulting adjunction
$\ICoh(\Y) \rightleftarrows \Dmod(\Y)$, coming from the $\2$-categorical structure on correspondences, is the usual induction/forgetful adjunction $(\ind_R, \oblv_R)$.
\end{example}

\ssec{Nil-isomorphisms and self-duality}

\sssec{}

A map of laft prestacks is said to be a \emph{nil-isomorphism} if it is an isomorphism at the reduced level. If a map is a nil-isomorphism, then the resulting $\ICoh$-pullback is conservative, see \cite[Volume 2, Chapter 3, Proposition 3.1.2]{Book}.

\medskip

As a main example consider the following: for $f: \Y \to \Z$ a map in $\Stk$, the natural map $\primef: \Y \to \Z^\wedge_\Y$ is a nil-isomorphism. Thus, we obtain the following statement.

\begin{prop} \label{prop: monadic adjunction ICoh}
In the situation above, the functors $(\primef)_*^\ICoh$ and $(\primef)^!$ form a monadic adjunction
\begin{equation} \label{eqn:adjunction-ICoh-on-formal-completion}
\begin{tikzpicture}[scale=1.5]
\node (a) at (0,1) {$\ICoh(\Y)$};
\node (b) at (3,1) {$\ICoh(\Z^\wedge_{\Y})$.};
\path[->,font=\scriptsize,>=angle 90]
([yshift= 1.5pt]a.east) edge node[above] {$(\primef)_*^\ICoh $} ([yshift= 1.5pt]b.west);
\path[->,font=\scriptsize,>=angle 90]
([yshift= -1.5pt]b.west) edge node[below] {$ (\primef)^! $} ([yshift= -1.5pt]a.east);
\end{tikzpicture}
\end{equation}
In particular, $\ICoh(\Z^\wedge_\Y)$ is compactly generated by $(\primef)_*^\ICoh(\Coh(\Y))$. 
\end{prop}

\sssec{}

By the definition of the universal envelope of a Lie algebroid (see \cite[Volume 2, Chapter 8, Section 4.2]{Book}), we may write:
\begin{equation} \label{eqn: U of tang for intro}
\ICoh(\Z^\wedge_{\Y})
\simeq
\U(\Tang_{\Y/\Z}) \mod (\ICoh(\Y)).
\end{equation}
Recall that a compactly generated DG category is automatically dualizable (for a proof, see \cite[Proposition 2.3.1]{DGCat}). The next result describes the dual of $\ICoh(\Z^\wedge_{\Y})$.

\begin{prop} \label{prop:ICOH-formal-completion-self-dual}
In the situation above, the DG category $\ICoh(\Z^\wedge_\Y)$, which is automatically dualizable by the above proposition, is self-dual.
\end{prop}

\begin{proof}
We will exhibit two functors and prove they form a self-duality datum for $\ICoh(\Z^\wedge_\Y)$.
We set:
$$
\coev: \Vect 
\xto{\; (p_{\Z^\wedge_\Y})^! \; } 
\ICoh(\Z^\wedge_\Y) 
\xto{ \; \Delta_*^{\ICoh} \; } 
\ICoh(\Z^\wedge_\Y \times \Z^\wedge_\Y)
\simeq
 \ICoh(\Z^\wedge_\Y)  \otimes \ICoh(\Z^\wedge_\Y),
$$
where the second functor is continuous as $\Delta: \Z^\wedge_\Y \to \Z^\wedge_\Y \times \Z^\wedge_\Y$ is inf-schematic\footnote{The notion of inf-schematic morphism is introduced in \cite[Volume 2, Part 1,Chapter 2]{Book}.} (since $\Y$ has schematic diagonal) and the last equivalence holds because $\ICoh(\Z^\wedge_\Y) $ is dualizable.

As for the functor going the opposite direction, we set:
$$
\ev: 
 \ICoh(\Z^\wedge_\Y)  \otimes \ICoh(\Z^\wedge_\Y)
\xto{ \; \Delta^! \; } 
 \ICoh(\Z^\wedge_\Y)  
\xto{\; \pi_*^\ICoh \; } 
\Dmod(\Y)
\xto{\; \Gamma(\Y_\dR, -)_{\ren} \; } 
\Vect,
$$
where $\pi: \Z^\wedge_\Y \to \Y_\dR$ is the tautological inf-schematic map and $\Gamma(\Y_\dR, -)_{\ren}$ is the functor of \emph{renormalized de Rham global sections}, see \citep{DG:Finiteness}. By definition, $\Gamma(\Y_\dR, -)_{\ren}$ is the dual of $(p_{\Y_\dR})^!$ under the standard self-duality of $\Dmod(\Y)$ and $\Vect$.

After a straightforward diagram chase, proving that these two functors yield a self-duality datum boils down to proving that the functor
$$
\ICoh(\Z^\wedge_\Y)  
\xto{\;\;(\pi \times \id)_*^{\ICoh}\;\;}
\Dmod(\Y) \otimes \ICoh(\Z^\wedge_\Y)   
\xto{\;\Gamma(\Y_\dR, -)_{\ren} \otimes \id \;}
\ICoh(\Z^\wedge_\Y)  
$$
is the identity. It suffices to check this smooth-locally on $\Z$. Then we can assume that $\Z=Z$ is a scheme, in which case 
\begin{equation} \label{eqn:pippo}
\ICoh(Z^\wedge_\Y) \simeq \ICoh(Z) \otimes_{\Dmod(Z)} \Dmod(\Y)
\end{equation} 
and the assertion is obvious (the functor in question being dual to the identity).
To prove \eqref{eqn:pippo}, first use the dualizability of $\ICoh(Z)$ as a $\Dmod(Z)$-module (proven in \cite[Corollary 4.2.2]{shvcatHH}) to reduce to the case where $\Y$ is also a scheme; then combine \cite[Proposition 3.1.2]{AG2} with the $1$-affineness of $Z_\dR$, which allows us to use \cite[Proposition 3.1.9]{ShvCat}.
\end{proof}

\sssec{}

Unraveling the construction, the two functors (\ref{eqn:adjunction-ICoh-on-formal-completion}) are dual to each other under the self-duality of the above proposition and the standard self-duality of $\ICoh(\Y)$.
Consequently, the Serre involution 
$$
\bbD_{\Z^\wedge_\Y}^\Serre: 
(\ICoh(\Z^\wedge_\Y)^{\cpt})^\op
\xto{\;\;\simeq\;\;}
\ICoh(\Z^\wedge_\Y)^{\cpt}
$$
sends $(\primef)_*^\ICoh(C) \squigto (\primef)_*^\ICoh (\bbD_\Y^\Serre C )$.

\section{$\ICoh_0$ in the bounded case} \label{sec:ICoh-zero-bounded}

We start this section by officially defining the DG category $\ICoh_0(\Y \to \Z^\wedge_\Y)$ attached to a map of stacks $\Y \to \Z$, with $\Y \in \Stkevcoc$. A crucial condition to make this DG category manageable is the perfection of the relative cotangent complex $\LL_{\Y/\Z}$. Another useful condition to impose is the perfection of $\Y$ itself: as we show below, this makes $\ICoh_0(\Y \to \Z^\wedge_\Y)$ compactly generated.

Thus, for simplicity, we will restrict our attention to stacks that are bounded perfect and lfp. It will then be immediately clear that the assignment $[\Y \to \Z] \squigto \ICoh_0(\Y \to \Z^\wedge_\Y)$ underlies a functor 
\begin{equation} \label{funct: ICOHzero as a functor out of formalmod-opposite}
\ICoh_0:
\bigt{ 
\Arr(\Stkperfevcoclfp) 
}^\op
\longto
\DGCat.
\end{equation}
We shall extend such functor to an $\2$-functor out of an $ \2$-category of correspondences, see (\ref{eqn:ICOH-zero-out-of-corresp-ev-coc-case}).
We will also discuss descent for $\ICoh_0$, as well as its behaviour under tensoring up over $\QCoh$.

\ssec{Definition and first properties} \label{ssec:Defn-ICohzero}

In this section, we define the DG category $\ICoh_0(\Y \to \Z^\wedge_\Y)$ attached to the nil-isomorphism $\primef: \Y \to \Z^\wedge_\Y$ and discuss some generalities.

\begin{defin} \label{sssec:the first definition}
Let $f: \Y \to \Z$ be a morphism in $\Stk$ with $\Y$ bounded.
We define $\ICoh_0(\Y \to \Z^\wedge_\Y)$ as the DG category sitting in the pull-back square
\begin{equation} \label{diag:definition-ICoh0}
\begin{tikzpicture}[scale=1.5]
\node (M) at (0,1.2) {$\ICoh_0(\Y \to \Z^\wedge_\Y)$ };
\node (IM) at (0,0) {$\ICoh(\Z^\wedge_\Y)$};
\node (N) at (3,1.2) {$\QCoh(\Y)$};
\node (IN) at (3,0) {$\ICoh(\Y).$};
\path[right hook->,font=\scriptsize,>=angle 90]
(M.south) edge node[right] { $\iota$ } (IM.north);
\path[right hook->,font=\scriptsize,>=angle 90]
(N.south) edge node[right] { $\Upsilon_\Y$ } (IN.north);
\path[->,font=\scriptsize,>=angle 90]
(M.east) edge node[above] { $(\primef)^{!,0}$ } (N.west);
\path[->,font=\scriptsize,>=angle 90]
(IM.east) edge node[above] { $(\primef)^!$ } (IN.west);
\end{tikzpicture}
\end{equation}

\end{defin}

\begin{rem}
The definition is taken from \cite{AG2}, with the proviso that \cite{AG2} assumed quasi-smoothness of the stacks involved and used the functor $\Xi_\Y$ in place of $\Upsilon_\Y$ (those two functors differ by a shifted line bundle in the quasi-smooth case).
\end{rem}

\begin{rem}
Since $\Upsilon_\Y$ is (symmetric) monoidal, it is clear that the tensor product $\stackrel ! \otimes$ on $\ICoh(\Z^\wedge_\Y)$ preserves the subcategory $\ICoh_0(\Z^\wedge_\Y)$. In other words, $\iota$ is the inclusion of a symmetric monoidal subcategory.
\end{rem}

\sssec{Warning} \label{sssec:warning-abuse-notation}

We will abuse notation and write $\ICoh_0(\Z^\wedge_\Y)$ instead of the more precise $\ICoh_0 (\Y \to \Z^\wedge_\Y)$. Observe that the latter category really depends on the formal moduli problem $\Y \to \Z^\wedge_\Y$, and not just on $\Z^\wedge_\Y$: indeed, $\Z^\wedge_\Y$ is insensitive to any derived or non-reduced structure on $\Y$, while $\ICoh_0 (\Y \to \Z^\wedge_\Y)$ is not.

\begin{prop} \label{prop:monadic-descrption-ICoh_0}

Let $f: \Y \to \Z$ be a map in $\Stk$, with $\Y$ bounded. Assume that the relative cotangent complex $\LL_{\Y/\Z}$ is \emph{perfect}.
Then the pullback diagram (\ref{diag:definition-ICoh0}) is left adjointable, i.e., the horizontal arrows admit left adjoints and the resulting lax-commutative diagram 
\begin{equation} \label{diag:definition-ICoh0-left-adjoints}
\begin{tikzpicture}[scale=1.5]
\node (M) at (0,1.2) {$\ICoh_0(\Z^\wedge_\Y)$ };
\node (IM) at (0,0) {$\ICoh(\Z^\wedge_\Y)$};
\node (N) at (3,1.2) {$\QCoh(\Y)$};
\node (IN) at (3,0) {$\ICoh(\Y)$};
\path[right hook->,font=\scriptsize,>=angle 90]
(M.south) edge node[right] { $\iota$ } (IM.north);
\path[right hook->,font=\scriptsize,>=angle 90]
(N.south) edge node[right] { $\Upsilon_\Y$ } (IN.north);
\path[<-,font=\scriptsize,>=angle 90]
(M.east) edge node[above] { $(\primef)_{*,0}$ } (N.west);
\path[<-,font=\scriptsize,>=angle 90]
(IM.east) edge node[above] { $(\primef)_*^\ICoh$ } (IN.west);
\end{tikzpicture}
\end{equation}
is actually commutative.
\end{prop} 

\begin{proof}
We just need to verify that the functor
$$
(\primef)_*^\ICoh \circ \Upsilon_\Y:
\QCoh(\Y)
\longto
\ICoh(\Z^\wedge_\Y)
$$
lands inside $\ICoh_0(\Z^\wedge_\Y)$. 
It suffices to show that the monad $\U(\Tang_{\Y/\Z}) := (\primef)^! \circ (\primef)_*^\ICoh$ on $\ICoh(\Y)$ preserves the subcategory $\Upsilon_\Y(\QCoh(\Y))$. 
We proceed as in \cite[Proposition 3.2.3]{AG2}.
Since $\U(\Tang_{\Y/\Z}) $ admits a canonical non-negative filtration (the PBW filtration, see \cite[Volume 2, Chapter 9, Theorem 6.1.2]{Book}), it suffices to check the assertion for each $n^{th}$ associated graded piece. Since $\LL_{\Y/\Z}$ is perfect, the latter is the functor of tensoring with 
$$
\Sym^n(\Tang_{\Y/\Z})
\simeq
\Upsilon_\Y \bigt{ \Sym^n( \TangQ_{\Y/\Z} ) },
$$
where $\TangQ_{\Y/\Z}$ is the dual of $\LL_{\Y/\Z}$ in $\QCoh(\Y)$.
It is clear that tensoring with $\Upsilon_\Y \bigt{ \Sym^n( \TangQ_{\Y/\Z} ) }$ preserves $\Upsilon_\Y(\QCoh(\Y))$. 
\end{proof}

\sssec{} \label{sssec: monade adjunction per intro}

Let us assume, as in the above proposition, that $f: \Y \to \Z$ is a map in $\Stk$ with $\Y$ bounded and with $\LL_{\Y/\Z}$ perfect.
Since $(\primef)^{!,0}$ is continuous and conservative, the monadic adjunction
\begin{equation} 	\label{adj:monadic adjunction for ICOHzero}
\begin{tikzpicture}[scale=1.5]
\node (a) at (0,1) {$\QCoh(\Y)$};
\node (b) at (4,1) {$\ICoh_0(\Z^\wedge_{\Y})$.};
\path[->,font=\scriptsize,>=angle 90]
([yshift= 1.5pt]a.east) edge node[above] {$(\primef)_{*,0} \simeq (\primef)_*^\ICoh \circ \Upsilon_\Y $} ([yshift= 1.5pt]b.west);
\path[->,font=\scriptsize,>=angle 90]
([yshift= -1.5pt]b.west) edge node[below] {$(\primef)^{!,0} \simeq \Phi_\Y \circ (\primef)^!$} ([yshift= -1.5pt]a.east);
\end{tikzpicture}
\end{equation}
yields an equivalence 
$$
\ICoh_0(\Z^\wedge_\Y) \simeq \UQ (\Tang_{\Y/\Z}) \mod (\QCoh(\Y)),
$$ 
where $\UQ (\Tang_{\Y/\Z}) $ is, by definition, the monad $\Phi_\Y \circ \U(\Tang_{\Y/\Z}) \circ  \Upsilon_\Y$.

\medskip

This monadic description implies the compact generation of $\ICoh_0(\Z^\wedge_\Y)$ as follows.

\begin{cor} \label{cor:finale}
With the notation above, assume furthermore that $\Y$ is perfect. Then the DG category $\ICoh_0(\Z^\wedge_\Y)$ is compactly generated by objects of the form 
$$
(\primef)_{*,0}(P) \simeq
(\primef)_*^\IndCoh (\Upsilon_\Y(P)),
\hspace{.4cm}
\text{for} \; P \in \Perf(\Y).
$$  
\end{cor} 

\begin{rem} \label{rem:cute}
The lax commutative diagram
\begin{equation} 
\begin{tikzpicture}[scale=1.5]
\node (M) at (0,1.2) {$\ICoh_0(\Z^\wedge_\Y)$ };
\node (IM) at (0,0) {$\ICoh(\Z^\wedge_\Y)$};
\node (N) at (3,1.2) {$\QCoh(\Y)$};
\node (IN) at (3,0) {$\ICoh(\Y),$};
\path[<-,font=\scriptsize,>=angle 90]
(M.south) edge node[right] { $\iota^R$ } (IM.north);
\path[<-,font=\scriptsize,>=angle 90]
(N.south) edge node[right] { $\Phi_\Y$ } (IN.north);
\path[<-,font=\scriptsize,>=angle 90]
(M.east) edge node[above] { $(\primef)_{*,0}$ } (N.west);
\path[<-,font=\scriptsize,>=angle 90]
(IM.east) edge node[above] { $(\primef)_*^\ICoh$ } (IN.west);
\end{tikzpicture}
\end{equation}
obtained from (\ref{diag:definition-ICoh0-left-adjoints}) by changing the vertical arrows with their right adjoints, is commutative. 
Checking this boils down to proving that $\iota^R$ sends $(\primef)_*^\ICoh(C)$, with $C \in \ICoh(\Y)$, to the object $(\primef)_*^\ICoh(\Upsilon_\Y \Phi_\Y  C)$.
This is a simple computation, which uses the fact that $\Phi_\Y$ is $\QCoh(\Y)$-linear.  %
\end{rem}

\ssec{Duality} \label{ssec:duality of ICohzero}

Let $\Y \to \Z$ be a morphism in $\Stk$, with $\Y$ bounded and perfect. Assume also that $\LL_{\Y/\Z}$ perfect. We show that the DG category $\ICoh_0(\Z^\wedge_\Y)$ is naturally self-dual. 

\sssec{}

Denote by $\ICoh_0^{(\Xi)}(\Z^\wedge_\Y)$ the DG category defined as in diagram (\ref{diag:definition-ICoh0}), but with the inclusion $\Xi_\Y$ in place of $\Upsilon_\Y$.
Reasoning as above, we see that there is a monadic adjunction
\begin{equation} 	\label{adj:monadic adjunction for ICOHzero Xi}
\begin{tikzpicture}[scale=1.5]
\node (a) at (0,1) {$\QCoh(\Y)$};
\node (b) at (4,1) {$\ICoh_0^{(\Xi)}(\Z^\wedge_{\Y})$.};
\path[->,font=\scriptsize,>=angle 90]
([yshift= 1.5pt]a.east) edge node[above] {$(\primef)_*^\ICoh \circ \Xi_\Y $} ([yshift= 1.5pt]b.west);
\path[->,font=\scriptsize,>=angle 90]
([yshift= -1.5pt]b.west) edge node[below] {$\Psi_\Y \circ (\primef)^! $} ([yshift= -1.5pt]a.east);
\end{tikzpicture}
\end{equation}
In particular, since $\Y$ is perfect, $\ICoh_0^{(\Xi)}(\Z^\wedge_\Y)$ is compactly generated by $(\primef)_*^\ICoh(\Xi_\Y \Perf(\Y))$.

\begin{lem}
The DG categories $\ICoh_0^{(\Xi)}(\Z^\wedge_\Y)$ and $\ICoh_0(\Z^\wedge_\Y)$ are mutually dual.
\end{lem}

\begin{proof}
Both categories are compactly generated, hence dualizable. Furthermore, they are retracts of the dualizable DG category $\ICoh(\Z^\wedge_\Y)$ by means of the right adjoints of the structure inclusions 
$$
\iota: \ICoh_0(\Z^\wedge_\Y) \hto \ICoh(\Z^\wedge_\Y),
\hspace{.6cm}
\ICoh_0^{(\Xi)}(\Z^\wedge_\Y) \hto \ICoh(\Z^\wedge_\Y).
$$
Observe that these right adjoints are continuous in view of Remark \ref{rem:cute}.

Now, using the self-duality of $\ICoh(\Z^\wedge_\Y)$ proven in Proposition \ref{prop:ICOH-formal-completion-self-dual}, we see that the dual of $\ICoh_0(\Z^\wedge_\Y)$ is the full subcategory of $\ICoh(\Z^\wedge_\Y)$ consisting of those objects $\F$ for which the natural arrow $(\iota^R)^\vee \iota^\vee (\F) \to \F$ is an isomorphism. This happens if and only if
$$
\langle \F, (\primef)_*^\ICoh(\Upsilon_\Y \Phi_\Y  C) \rangle
\longto
\langle \F, (\primef)_*^\ICoh(C) \rangle
$$
is an isomorphism for any $C \in \Coh(\Y)$, which in turn is equivalent to  
$$
\Xi_\Y \Psi_\Y \bigt{ (\primef)^!\F } \xto{\simeq} (\primef)^!\F.
$$
In other words, $(\primef)^!\F$ must belong to $\Xi_\Y(\QCoh(\Y))$, which means precisely that $\F \in \ICoh_0^{(\Xi)}(\Z^\wedge_\Y)$.
\end{proof}

\begin{rem}
One easily checks that the dual of the inclusion $\iota: \ICoh_0(\Z^\wedge_\Y) \hto \ICoh(\Z^\wedge_\Y)$ is the functor
$$
\iota^\vee: \ICoh(\Z^\wedge_\Y) \tto \ICoh_0^{(\Xi)}(\Z^\wedge_\Y)
$$
that sends $(\primef)_*^\ICoh(\F) \squigto (\primef)_*^\ICoh(\Xi_\Y \Psi_\Y \F)$.
\end{rem}

\begin{prop}  \label{prop: ICOHzero-for-stacks-self-dual}
In the situation above, there exists an equivalence 
$$
\sigma:\ICoh_0^{(\Xi)}(\Z^\wedge_\Y)
\xto{\simeq}
\ICoh_0  (\Z^\wedge_\Y)
$$
that renders the triangle
\begin{equation} \label{diag:triangle-with-two-ICoh0}
\begin{tikzpicture}[scale=1.5]
\node (Ups) at (4,0) {$\ICoh_0(\Z^\wedge_\Y)$ };
\node (Xi) at (0,0) {$\ICoh_0^{(\Xi)}(\Z^\wedge_\Y)$ };
\node (Q) at (2,1) {$\QCoh(\Y)$};
\path[->,font=\scriptsize,>=angle 90]
(Q.south west) edge node[right] {  } (Xi.north east);
\path[->,font=\scriptsize,>=angle 90]
(Q.south east) edge node[right] {  } (Ups.north west);
\path[->,font=\scriptsize,>=angle 90]
(Xi.east) edge node[above] { $\sigma$ } (Ups.west);
\end{tikzpicture}
\end{equation}
commutative. In particular, $\ICoh_0(\Z^\wedge_\Y)$ is self-dual in the only way that makes $(\primef)_{*,0}$ and $(\primef)^{!,0}$ dual to each other. 
\end{prop}

\begin{proof}
The adjunction (\ref{adj:monadic adjunction for ICOHzero Xi}) is monadic, and the monad is easily seen to coincide with the one of the monadic adjunction (\ref{adj:monadic adjunction for ICOHzero}): indeed, it suffices to show that each functor $\U^{\leq n}(\Tang_{\Y/\Z})$ preserves the subcategory $\Xi(\Perf(\Y)) \subseteq \ICoh(\Y)$, which is immediately checked at the level of the associated graded.

This fact implies the existence of the equivalence $\sigma$ fitting in the above triangle.
As for the duality statement, let us compute the evaluation between $f_{*,0}(Q)$ and an arbitrary $\F \in \ICoh_0(\Z^\wedge_\Y)$. We have 
\begin{eqnarray}
\nonumber
\langle (\primef)_*^\ICoh \bigt{ \Upsilon_\Y Q },  \F \rangle_{\ICoh_0(\Z^\wedge_\Y)}
& \simeq &
\langle \sigma^{-1} (\primef)_*^\ICoh \bigt{ \Upsilon_\Y Q },  \F \rangle_{\ICoh(\Z^\wedge_\Y)}
\\
\nonumber
& \simeq &
\langle (\primef)_*^\ICoh \bigt{ \Xi_\Y Q },  \F \rangle_{\ICoh(\Z^\wedge_\Y)}
\\
\nonumber
& \simeq &
\langle Q, \Phi_\Y  (\primef)^!\F \rangle_{\QCoh(\Y)}
\\
\nonumber
& \simeq &
\langle Q, (\primef)^{!,0} \F \rangle_{\QCoh(\Y)},
\end{eqnarray}
as claimed.
\end{proof}

\ssec{Functoriality}

The results of the previous sections show that the DG category $\ICoh_0(\Z^\wedge_\Y)$ is particularly well-behaved in the case $\Y \in \Stkperfevcoc$ and the relative contangent complex $\LL_{\Y/\Z}$ is perfect. Hence, it makes sense to restrict $\ICoh_0$ to arrows $\Y \to \Z$ in $\Stkperfevcoclfp$.
In this section, we upgrade the functor
$$
\bigt{
\Arr(\Stkperfevcoclfp)
}^\op
\longto \DGCat,
\hspace{.4cm}
[\Y \to \Z] \squigto \ICoh_0(\Z^\wedge_\Y)
$$
to a functor out a certain category of correspondences of $\Arr(\Stkperfevcoclfp)$.
To do so, we shall reduce the question to the functoriality of 
$$
[\Y \to \Z] \squigto \ICoh(\Z^\wedge_\Y),
$$
which is known thanks to Theorem \ref{thm:ICOH-corresp-book}.

\sssec{}

To simplify the notation, denote by $\Arr := \Arr(\Stkperfevcoclfp)$ the $\infty$-category of arrows in $\Stkperfevcoclfp$. A $1$-morphism in $\Arr$, say between $[\Y_1 \to \Z_1]$ and $[\Y_2 \to \Z_2]$, is represented by a commutative diagram
\begin{equation} \label{diag:map-in-Arr}
\begin{tikzpicture}[scale=1.5]
\node (00) at (0,0) {$\Z_1$};
\node (10) at (1.5,0) {$\Z_2$,};
\node (01) at (0,.8) {$\Y_1$};
\node (11) at (1.5,.8) {$\Y_2$};
\path[->,font=\scriptsize,>=angle 90]
(00.east) edge node[above] {$ $} (10.west);
\path[->,font=\scriptsize,>=angle 90]
(01.east) edge node[above] {$ $} (11.west);
\path[->,font=\scriptsize,>=angle 90]
(01.south) edge node[right] {$ $} (00.north);
\path[->,font=\scriptsize,>=angle 90]
(11.south) edge node[right] {$ $} (10.north);
\end{tikzpicture}
\end{equation}
where, by convention, objects of $\Arr$ are always drawn as vertical arrows.

\sssec{}

Note that $\Arr$ is not closed under fiber products, as $\Stkevcoc$ is not (see Example \ref{example:loops}). Hence, we cannot define the $\infty$-category $\Corr(\Arr)$ without restricting the class of vertical arrows in some appropriate way.
To this end, let us make the following definitions. We say that a morphism (\ref{diag:map-in-Arr}) is \emph{schematic and bounded} (respectively, \emph{schematic and proper}) if so is the top horizontal map.

It is then clear that the $\2$-category
$$
\Corr(\Arr)_{\schem \& \evcoc;\all}^{\schem \& \evcoc \& \proper}
$$
is well defined. We will show that the assignment $[\Y \to \Z] \squigto \ICoh_0(\Z^\wedge_\Y)$ can be upgraded to an $\2$-functor out of the above $\2$-category of correspondences, with pushforward functor directly induced by the $(*, \ICoh)$-pushforward. More precisely, we will establish the following theorem.

\begin{thm} \label{thm:functoriality of ICOHzero}
The functor $\ICoh$ of (\ref{eqn:ICOH-out-of-Corr-formal-completions}) restricts to an $\2$-functor
\begin{equation} \label{eqn:ICOH-zero-out-of-corresp-ev-coc-case}
\ICoh_0:
\Corr(\Arr)_{\schem \& \evcoc;\all}^{\schem \& \evcoc \& \proper}
\longto
\DGCat^{2 - \Cat}.
\end{equation}
\end{thm}

\begin{proof}
It is clear that $!$-pullbacks always preserve the $\ICoh_0$ subcategories. It remains to check that, for a diagram (\ref{diag:map-in-Arr}) with schematic and bounded top arrow, the $\ICoh$-pushforward functor
$$
\ICoh((\Z_1)^\wedge_{\Y_1})
\longto
\ICoh((\Z_2)^\wedge_{\Y_2})
$$
preserves the $\ICoh_0$-subcategories. We can write the map $(\Z_1)^\wedge_{\Y_1} \to (\Z_2)^\wedge_{\Y_2}$ as the composition 
\begin{equation} \label{eqn:composition of maps of formal completions}
(\Z_1)^\wedge_{\Y_1}
\xto{\; \alpha \;}
(\Z_2)^\wedge_{\Y_1}
\xto{\; \beta \;}
(\Z_2)^\wedge_{\Y_2}
\end{equation}
and analyze the two resulting functors $\alpha_*^\ICoh$ and $\beta_*^\ICoh$ separately. Let us show that $\alpha_*^\ICoh$ preserves the $\ICoh_0$-subcategories. As we know by Corollary \ref{cor:finale}, the DG category $\ICoh_0((\Z_1)^\wedge_{\Y_1})$ is generated under colimits by the essential image of the induction map
$$
\QCoh(\Y_1) \to \ICoh_0((\Z_1)^\wedge_{\Y_1}),
$$
and $\alpha_*^\ICoh$ obviously sends each such generator to an object of $\ICoh_0((\Z_2)^\wedge_{\Y_1})$.
It remains to discuss the pushforward $\beta_*^\ICoh$ along the rightmost map in (\ref{eqn:composition of maps of formal completions}). The question is settled by the following more general result.
\end{proof}

\begin{lem} \label{lem:evcoc-push-forwards-preserve-icoh0}
Consider a string $\X \to \Y \to \Z$ in $\Stklfp$, with both $\X$ and $\Y$ bounded. Assume that the first map $f: \X \to \Y$ is schematic and bounded. Denoting by $\beta: \Z^\wedge_\X \to \Z^\wedge_\Y$ the induced map, the functor $\beta_*^{\ICoh}: \ICoh(\Z^\wedge_\X) \to \ICoh(\Z^\wedge_\Y)$ sends $\ICoh_0(\Z^\wedge_\X)$ to $\ICoh_0(\Z^\wedge_\Y)$.
\end{lem}

\begin{proof} 
It suffices to check that the image of the functor 
$$
f_*^\ICoh \circ \Upsilon_\X: 
\QCoh(\X)
 \longto 
 \ICoh(\Y)
$$ 
is contained in $\Upsilon_\Y(\QCoh(\Y)) \subseteq \ICoh(\Y)$. 
Since the question is smooth local in $\Y$, we may pullback to an atlas of $\Y$, thereby reducing the assertion to the case of $X =\X$ and $Y=\Y$ schemes.

We need to prove that the natural transformation
$$
\Upsilon_Y \Phi_Y  f_*^\ICoh  \Upsilon_X
\longto
f_*^\ICoh  \Upsilon_X
$$
is an isomorphism. Passing to duals, this is equivalent to showing that  
\begin{equation} \label{eqn:counit of some adjunction}
\Psi_X f^! \Xi_Y \Psi_Y (\F)
\longto
\Psi_X  f^! (\F)
\end{equation}
is an isomorphism for any $\F \in \Coh(Y)$.
We now use \cite[Lemma 7.2.2]{ICoh}: since $f$ is bounded (eventually coconnective in the terminology of \cite[Section 7.2]{ICoh}), there exists a continuous functor $f^{!, \QCoh}$ equipped with as isomorphism $f^{!, \QCoh} \circ \Psi_Y \simeq \Psi_X \circ f^!$.
The conclusion follows from the fully faithfulness of $\Xi_Y$.
\end{proof}

\sssec{}

Let $\xi : [\Y_1 \to \Z_1] \to [\Y_2 \to \Z_2]$ be a morphism in $\Arr$ as in (\ref{diag:map-in-Arr}), which is schematic and bounded. We denote by 
$$
\xi_{*,0}: 
\ICoh_0((\Z_1)^\wedge_{\Y_1})
\longto
\ICoh_0((\Z_2)^\wedge_{\Y_2})
$$
the push-forward functor of the above theorem.
Such notation matches the usage of the $(*,0)$-pushforwards that appeared earlier in the text. Indeed, if $\xi$ is proper, $\xi_{*,0}$ is left adjoint to $\xi^{!,0}$.

\sssec{}

Let us spell out the base-change isomorphism for $\ICoh_0$ stated in Theorem \ref{thm:functoriality of ICOHzero}. A pair of maps
$$
f: [\W \to \X]
\longto 
[\Y \to \Z],
\hspace{.4cm}
g: [\U \to \V]
\longto 
[\Y \to \Z]
$$
in $\Arr$ corresponds to a commutative diagram
$$ 
\begin{tikzpicture}[scale=1.5]
\node (00) at (0,0) {$\X$};
\node (10) at (1.5,0) {$\Z$};
\node (01) at (0,.8) {$\W$};
\node (11) at (1.5,.8) {$\Y$};
\node (21) at (3,.8) {$\U$};
\node (20) at (3,0) {$\V$};
\path[->,font=\scriptsize,>=angle 90]
(00.east) edge node[above] {$ $} (10.west);
\path[->,font=\scriptsize,>=angle 90]
(01.east) edge node[above] {$ $} (11.west);
\path[<-,font=\scriptsize,>=angle 90]
(11.east) edge node[above] {$ $} (21.west);
\path[<-,font=\scriptsize,>=angle 90]
(10.east) edge node[above] {$ $} (20.west);
\path[->,font=\scriptsize,>=angle 90]
(01.south) edge node[right] {$ $} (00.north);
\path[->,font=\scriptsize,>=angle 90]
(11.south) edge node[right] {$ $} (10.north);
\path[->,font=\scriptsize,>=angle 90]
(21.south) edge node[right] {$ $} (20.north);
\end{tikzpicture}
$$
in $\Stkevcoclfp$. Such commutative diagram yields the commutative diagram
$$ 
\begin{tikzpicture}[scale=1.5]
\node (00) at (0,0) {$\X$};
\node (10) at (1.5,0) {$\X \times_\Z \V$ };
\node (01) at (0,.8) {$\W$};
\node (11) at (1.5,.8) {$\W \times_\Y \U$};
\node (21) at (3,.8) {$\U$};
\node (20) at (3,0) {$\V$};
\path[<-,font=\scriptsize,>=angle 90]
(00.east) edge node[above] {$ $} (10.west);
\path[<-,font=\scriptsize,>=angle 90]
(01.east) edge node[above] {$ $} (11.west);
\path[->,font=\scriptsize,>=angle 90]
(11.east) edge node[above] {$ $} (21.west);
\path[->,font=\scriptsize,>=angle 90]
(10.east) edge node[above] {$ $} (20.west);
\path[->,font=\scriptsize,>=angle 90]
(01.south) edge node[right] {$ $} (00.north);
\path[->,font=\scriptsize,>=angle 90]
(11.south) edge node[right] {$ $} (10.north);
\path[->,font=\scriptsize,>=angle 90]
(21.south) edge node[right] {$ $} (20.north);
\end{tikzpicture}
$$
which we regard as a correspondence
$$
[\W \to \X]
\xleftarrow{ \; \; G \; \; }
[\W \times_\Y \U \to \X \times_\Z \V]
\xto{\; \; F \; \; }
[\U \to \V]
$$
in $\Arr$, \emph{provided} that $f$ is bounded (so that $\W \times_\Y \U$ is bounded too). The theorem states that, if $f$ is moreover schematic, the diagram
$$ 
\begin{tikzpicture}[scale=1.5]
\node (00) at (0,0) {$\ICoh_0(\X^\wedge_\W)$};
\node (10) at (3.5,0) {$\ICoh_0(\Z^\wedge_\Y)$};
\node (11) at (3.5,1.2) {$\ICoh_0(\V^\wedge_\U)$};
\node (01) at (0,1.2) {$\ICoh_0((\X \times_\Z \V)^\wedge_{\W \times_\Y \U})$};
\path[->,font=\scriptsize,>=angle 90]
(00.east) edge node[above] {$f_{*,0}$} (10.west);
\path[->,font=\scriptsize,>=angle 90]
(01.east) edge node[above] {$F_{*,0}$} (11.west);
\path[<-,font=\scriptsize,>=angle 90]
(01.south) edge node[right] {$ G^{!,0}$} (00.north);
\path[<-,font=\scriptsize,>=angle 90]
(11.south) edge node[right] {$g^{!,0}$} (10.north);
\end{tikzpicture}
$$
is naturally commutative.

\sssec{} 

We now use the above functoriality to prove descent of $\ICoh_0$ \virg{in the second variable}.

\begin{cor} \label{cor:descent for ?-pullback in second variable}
For any $\W \in \Stkevcoclfp$, the functor
$$
((\Stkevcoclfp)_{\W/})^\op
\longto
\DGCat,
\hspace{.4cm}
[\W \to \Z] \squigto \ICoh_0(\Z^\wedge_\W)
$$
satisfies descent along \emph{any} map.
\end{cor}

\begin{proof}
Let $\W \to \Y \to \Z$ be a string in $\Stkevcoclfp$, giving rise to a \virg{nil-isomorphism} $\xi: [\W \to \Y] \to [\W \to \Z]$. Denote by $\fZ^\bullet$ the Cech resolution of $\xi$.
The $(!,0)$-pullback functors yield a cosimplicial category $\ICoh_0(\fZ^{\bullet})$ and a functor
$$
\ICoh_0(\Z^\wedge_\W)
\longto 
\Tot \Bigt{  \ICoh_0(\fZ^{\bullet}) },
$$
which we need to prove to be an equivalence. The cosimplicial category in question satisfies the left Beck-Chevalley condition: this follows from the adjunction 
\begin{equation} \label{adj:easy-adjunction-nil-iso-ev-coc-case}
\begin{tikzpicture}[scale=1.5]
\node (a) at (0,1) {$\alpha_{*,0} : \ICoh_0(\Y^\wedge_\W)$};
\node (b) at (4,1) {$\ICoh_0(\Z^\wedge_\W) : \alpha^{!,0}$};
\path[->,font=\scriptsize,>=angle 90]
([yshift= 1.5pt]a.east) edge node[above] {$ $} ([yshift= 1.5pt]b.west);
\path[->,font=\scriptsize,>=angle 90]
([yshift= -1.5pt]b.west) edge node[below] {$ $} ([yshift= -1.5pt]a.east);
\end{tikzpicture}
\end{equation}
induced by the nil-isomorphism $\Y^\wedge_\W \to \Z^\wedge_\W$ (recall that the ind-coherent pushforward along such map preserves the $\ICoh_0$ subcategories).
Then we conclude by using the paradigm of \cite[Section C.1]{ShvCat}, which is based on \cite[Theorem 4.7.5.2]{HA}.
\end{proof}

\ssec{Exterior tensor products}

Consider a diagram $\U \to \V \to \Z \leftto \Y \leftto \X$ in $\Stkperfevcoclfp$. 
In this situation, $\QCoh(\Z)$ acts on $\ICoh_0(\V^\wedge_\U)$ and $\ICoh_0(\Y^\wedge_\X)$ by pullback.

\begin{prop} \label{prop:exterior-product-ICohzero-over-QCoh-ev-coc-case}
In the above situation, assume furthermore that $\U \times_\Z \X$ is bounded. Then the exterior tensor product descends to an equivalence 
\begin{equation} \label{eqn:exterior-product-ICohzero-over-QCoh-ev-coc-case}
\ICoh_0(\V^\wedge_\U)
\usotimes{\QCoh(\Z)}
\ICoh_0(\Y^\wedge_\X)
\xto{\; \; \simeq \; \; }
\ICoh_0 \Bigt{ 
(\V \times_{\Z} \Y)
^\wedge_
{(\U \times_{\Z} \X)}
}.
\end{equation}
\end{prop}

\begin{proof}
Both DG categories are modules for monads acting on $\QCoh(\U \times_\Z \X)$.
Note that $\QCoh(\U \times_\Z \X)$ is generated by objects of the form $p^*P \otimes q^*Q$ for $P \in \QCoh(\U)$ and $Q \in \QCoh(\X)$, where $p: \U \times_\Z \X \to \U$ and $q: \U \times_\Z \X \to \X$ are the two projections.
We will identify the values of the two monads acting on such generators.

\medskip

The monad on the LHS is given by
$$
p^*P \otimes q^*Q
\squigto
p^* (\UQCoh(\Tang_{\U/\V})(P))
\otimes
q^* (\UQCoh(\Tang_{\X/\Y})(Q)),
$$
while the monad on the RHS by
$$
p^*P \otimes q^*Q
\squigto
 \UQCoh(\Tang_{\U \times_\Z \X/ \V \times_\Z \Y})
(p^*P \otimes q^*Q).
$$
Now, the elementary isomorphism
$$
\LL_{\U \times_\Z \X/ \V \times_\Z \Y}
\simeq
p^* \LL_{\U/\V} \oplus q^* \LL_{ \X/ \Y},
$$
taking place in $\QCoh(\U \times_\Z \X)$, yields the assertion upon dualization.
\end{proof}

\sssec{} \label{sssec:exceptional pull and push in bounded case}

As a consequence of the above exterior product formula, we obtain another kind of functor, the $?$-pushforward, for $\ICoh_0$. To construct it, consider maps $\X \to \Z \leftto \Y$ in $\Stkperfevcoclfp$, with the property that $\X \times_\Z \Y$ is also bounded.
We view the resulting cartesian diagram
\begin{equation} 
\nonumber
\begin{tikzpicture}[scale=1.5]
\node (00) at (0,0) {$\X$};
\node (10) at (1.5,0) {$\Z$};
\node (01) at (0,.8) {$\X \times_\Z \Y $};
\node (11) at (1.5,.8) {$\Y$};
\path[->,font=\scriptsize,>=angle 90]
(00.east) edge node[above] {$h$} (10.west);
\path[->,font=\scriptsize,>=angle 90]
(01.east) edge node[above] {$ $} (11.west);
\path[->,font=\scriptsize,>=angle 90]
(01.south) edge node[right] {$ $} (00.north);
\path[->,font=\scriptsize,>=angle 90]
(11.south) edge node[right] {$ $} (10.north);
\end{tikzpicture}
\end{equation}
as a morphism 
$$
\eta:
[\X \times_\Z \Y \to \X]
\longto
[\Y \to \Z]  
$$ 
in $\Arr = \Arr(\Stkperfevcoclfp)$.
Then the equivalence
\begin{equation} \nonumber
\QCoh(\X)
\usotimes{\QCoh(\Y)}
\ICoh_0(\Z^\wedge_\Y)
\xto{\; \; \simeq \; \; }
\ICoh_0 \Bigt{ 
\X
^\wedge_
{(\X \times_{\Z} \Y)}
}
\end{equation}
of Proposition \ref{prop:exterior-product-ICohzero-over-QCoh-ev-coc-case}, together with the usual adjunction $h^*: \QCoh(\Z) \rightleftarrows \QCoh(\X): h_*$, yields the adjunction
\begin{equation} \label{adj:exceptional-adjunction}
\begin{tikzpicture}[scale=1.5]
\node (a) at (0,1) {$\eta^{!,0}: \ICoh_0(\Z^\wedge_\Y)   $};
\node (b) at (3,1) {$  \ICoh_0(\X^\wedge_{\X \times_\Z \Y}) : \eta_?$.};
\path[->,font=\scriptsize,>=angle 90]
([yshift= 1.5pt]a.east) edge node[above] {  } ([yshift= 1.5pt]b.west);
\path[->,font=\scriptsize,>=angle 90]
([yshift= -1.5pt]b.west) edge node[below] { } ([yshift= -1.5pt]a.east);
\end{tikzpicture}
\end{equation}

\ssec{The monoidal category $\H(\Y)$} \label{ssec:defn of H}

In this short section, we officially introduce the main object of this paper: the monoidal category $\H(\Y)$ attached to $\Y \in \Stkevcoclfp$.

\sssec{}

Let $\Y$ be as above and recall the convolution monoidal structure on $\ICoh(\Yform)$ defined by pull-push along the correspondence 
$$
(\Y \times \Y)^\wedge_\Y \times (\Y \times \Y)^\wedge_\Y 
\xleftarrow{\wh p_{12} \times \wh p_{23}}
(\Y \times \Y \times \Y)^\wedge_{\Y} 
\xto{\wh p_{13}}
(\Y \times \Y)^\wedge_\Y.
$$
It is clear that
$$
\iota:
\ICoh_0((\Y \times \Y)^\wedge_\Y )
\hto
\ICoh( (\Y \times \Y)^\wedge_\Y  )
$$
is the inclusion of a monoidal sub-category: indeed, it suffices to show that convolution preserves the subcategory $\primeDelta_{*,0}(\QCoh(\Y)) \subseteq \ICoh( (\Y \times \Y)^\wedge_\Y )$, which is a simple diagram chase.

\sssec{}

The same reasoning shows that the functor
$$
\primeDelta_{*,0}:
\QCoh(\Y)
\longto
\ICoh_0((\Y \times \Y)^\wedge_\Y  )
$$
is monoidal: indeed, it can be written as the composition 
$$
\QCoh(\Y)
\xto{\Upsilon_\Y \,}
\ICoh(\Y)
\xto{\primeDelta_*^\ICoh}
\ICoh((\Y \times \Y)^\wedge_\Y )
$$
of monoidal functors.

\begin{defin}\label{sssec:official-definition of H(Y)}
We set $(\H(\Y), \star)$ to be the monoidal DG category
$$
\H(\Y)
:=
\ICoh_0( (\Y \times \Y)^\wedge_\Y),
$$
with the monoidal structure given by convolution, that is, the product induced by the correspondence
$$ 
\begin{tikzpicture}[scale=1.5]
\node (00) at (0,0) {$\Y^4$};
\node (10) at (1.5,0) {$\Y^3$ };
\node (01) at (0,.8) {$\Y^2$};
\node (11) at (1.5,.8) {$\Y$};
\node (21) at (3,.8) {$\Y$};
\node (20) at (3,0) {$\Y^2$};
\path[<-,font=\scriptsize,>=angle 90]
(00.east) edge node[above] {$p_{12} \times p_{23}$} (10.west);
\path[<-,font=\scriptsize,>=angle 90]
(01.east) edge node[above] {$\Delta$} (11.west);
\path[->,font=\scriptsize,>=angle 90]
(11.east) edge node[above] {$\id_\Y$} (21.west);
\path[->,font=\scriptsize,>=angle 90]
(10.east) edge node[above] {$p_{13}$} (20.west);
\path[->,font=\scriptsize,>=angle 90]
(01.south) edge node[right] {$\Delta \times \Delta$} (00.north);
\path[->,font=\scriptsize,>=angle 90]
(11.south) edge node[right] {$\Delta$} (10.north);
\path[->,font=\scriptsize,>=angle 90]
(21.south) edge node[right] {$\Delta$} (20.north);
\end{tikzpicture}
$$
in $\Arr$. The identity for convolution is of course the object $\bbone_{\H(\Y)} := {\primeDelta}_*^\ICoh(\omega_\Y)$.

\end{defin}

\begin{lem} \label{lem:pivotal}
Suppose that $\Y$ is perfect. Then the monoidal DG category $(\H(\Y), \star)$ is compactly generated, rigid and pivotal (see Sections \ref{sssec:rigidity conventions} and \ref{sssec:pivotality} for the definitions).
\end{lem}

\begin{proof}
Compact generation and rigidity follow immediately from the fact that
$$
\primeDelta_{*,0}: \QCoh(\Y) \longto \ICoh_0( (\Y \times \Y)^\wedge_\Y  )
$$ 
is monoidal and generates the target under colimits.
As for pivotality, we need to show that left and right duals of compact objects can be functorially identified. Note that $\H(\Y)$ is self-dual in two different looking ways. The first way is a consequence of rigidity: any rigid monoidal DG category is self-dual with evaluation given by $u^R \circ m$.
In our case, this reads as
$$
\ev_\star: \H(\Y) \otimes \H(\Y) \longto
\H(\Y), 
\hspace{.4cm}
(\F, \G) \mapsto \HHom_{\H(\Y)}(\bbone_{\H(\Y)}, \F \star \G).
$$
The second self-duality datum comes from the general theory of $\ICoh_0$, as proven in Section \ref{ssec:duality of ICohzero}: it determines a second evaluation functor that we will denote by $\ev_\otimes$.
In view of \cite[Lemma 4.6.1.10]{HA}, the two duality data are canonically identified: $\ev_\star \simeq \ev_\otimes$. Denote by $\bbD$ the contravariant involution on $\H(\Y)^\cpt$ induced by $ \ev_\otimes$ (explicitly, $\bbD$ is the composition of Serre duality for $\ICoh(\Yform)$ and $\sigma^{-1}$). Then, for $\F \in \H(\Y)^\cpt$, we have
$$
\HHom_{\H(\Y)}(\DD(\F), -)
\simeq
\ev_\otimes (\F \otimes -)
\simeq
\ev_\star(\F \otimes -)
=
\HHom_{\H(\Y)}(\bbone_{\H(\Y)}, \F \star -)
\simeq
\HHom_{\H(\Y)}({}^\vee\F,  -).
$$
It follows that ${}^\vee\F \simeq \DD(\F)$ naturally: this fact yields the desired pivotal structure.
\end{proof}

\sec{Beyond the bounded case} \label{sec:beyond-ev-coc}

The computation of $\Z(\H(\Y))$ sketched in Section \ref{sec:outline} showed the need to extend the definition of $\ICoh_0(\Z^\wedge_\Y)$ to the case where $\Y \in \Stk_\lfp$ is not necessarily bounded, but still perfect. 
This task is the goal of the present section.

Inspired by the equivalence (\ref{eqn:def of ICOHzero with universal enveloping alg}), we will define $\ICoh_0(\Z^\wedge_\Y)$ as the DG category of modules for a monad $\UQ(\Tang_{\Y/\Z})$ acting on $\QCoh(\Y)$, where $\UQ(\Tang_{\Y/\Z})$ is defined so that:
\begin{itemize}
\item
its meaning coincides with the already established one when $\Y$ is bounded;
\item
it is equivalent to $\UQ(\TangQ_{\Y/\Z}) \otimes - $, when $\Tang_{\Y/\Z}$ is a Lie algebra.
\end{itemize}
To define such monad, we will use the PBW filtration of the universal envelope of a Lie algebroid.

\medskip

After this is done, we will discuss the functoriality of $\ICoh_0$ in the unbounded context. Such functoriality is not as rich as the one discussed in the previous chapter, the issue being that we can no longer rely on the functoriality of ind-coherent sheaves on formal completions.
For instance, in the present context, the $(*,0)$-pushforward will be defined only for maps in $\Arr(\Stkperflfp)$ coming from diagrams of the form
$$
[\W \to \X] \longto [\W \to \Y].
$$
Similarly, the $(!,0)$-pullback will be defined only for maps with \emph{cartesian} associated square.
These two kinds of functors are \virg{good} because they preserve compact objects, whence they admit continuous right adjoints: the so-called $?$-pullback and $?$-pushforward, respectively.

We also discuss descent and tensoring up with $\QCoh$. The material of this section is essential to carry out the computation $\Z(\H(\Y)) \simeq \vDmod(\LY)$, performed in Section \ref{sec:center}.

\ssec{The definition} \label{ssec: defin of Icohzero unbdd case}

Consider a Lie algebroid $\fL$ on $\Y$ and its universal envelope $\U(\fL)$, which is a monad acting on $\ICoh(\Y)$. By \cite[Volume 2, Chapter 9, Section 6.1]{Book},  the assignment $\fL \squigto \U(\fL)$ upgrades to a functor
\begin{equation}  \label{funct:Lie algebroids to monads wihth filtration on ICOH}
\LieAlgbd(\Y) \longto \Alg(\End(\ICoh(\Y))^{\Fil, \geq 0}).
\end{equation}
The target $\infty$-category will be referred to as the $\infty$-category of \emph{monads (acting on $\ICoh(\Y)$) with non-negative filtration}.

We will assume familiarity with the Rees-Simpson point of view that filtered objects are objects that lie over the stack $\A^1/\bbG_m$: we refer to \cite[Volume 2, Chapter 9, Section 1.3]{Book} for the main results in the subject and for the notation.

\sssec{}

Let $\fL \in \LieAlgbd(\Y)$ be such that its underlying ind-coherent sheaf $\oblv_{\LieAlgbd}(\fL)$ belongs to the subcategory $\Upsilon_\Y(\Perf(\Y)) \subset \ICoh(\Y)$. In this situation, we will show that the monad $\U(\fL)$ induces a canonical monad acting on $\QCoh(\Y)$, denoted $\UQ(\fL)$.
We need the following paradigm, whch goes under the slogan: \emph{the filtered renormalization of a filtered monad is also a filtered monad}.

\begin{lem}
Let $\mu$ be a monad with non-negative filtration acting on a (cocomplete) DG category $\C$. Let $\C_0 \subset \C$ be a non-cocomplete subcategory with the property that, for each $n \geq 0$, the $n^{th}$ piece of the filtration $\mu^{\leq n}$ preserves $\C_0$.\footnote{Note that we do not require that $\mu$ have this property.}
Let $\Ind(\C_0)$ be the ind-completion of $\C_0$ and $\wt \mu^{\leq n}$ the ind-completion of the functor $\mu^{\leq n}: \C_0 \to \C_0$. Then the non-negatively filtered functor
$$
\wt \mu := \uscolim{n \geq 0} \, \wt\mu^{\leq n}: 
\Ind(\C_0) \longto \Ind(\C_0)
$$
admits a canonical structure of monad with filtration.
\end{lem}

\begin{proof} 
Consider the DG category
$$
\End(\C)^{\Fil} := \Fun(\ZZ, \End(\C)),
$$
equipped with the monoidal structure is given by Day convolution. 
A monad with filtration on $\C$ (for instance, $\mu$) is an algebra object of the above DG category.
Let us express $\End(\C)^{\Fil} $ using the stack $\A^1/\GG_m$.
By \cite[Volume 2, Chapter 9, Section 1.3]{Book},
$$
\End(\C)^{\Fil}
\simeq
\End(\C) \otimes {\QCoh(\A^1/\GG_m)}.
$$
Since $ {\QCoh(\A^1/\GG_m)}$ is dualizable (and self-dual), we further obtain that
$$
\End(\C)^{\Fil}
\simeq
\Fun(\C, \C \otimes \QCoh(\A^1/\GG_m) )
\simeq
\End_{\QCoh(\A^1/\GG_m)}
\bigt{\C \otimes \QCoh(\A^1/\GG_m) },
$$
in such a way that the Day convolution on the LHS corresponds to the obvious monoidal structure on the RHS.

Now, denote by $\C_0'$ the Karoubi completion of $\C_0$.
The assumption on $\mu$ means that its restriction along $\C_0' \otimes \Perf(\A^1/\GG_m)$ gives rise to an (automatically algebra) object $\mu'$ of 
$$
\End^{\on{non-cocomplete}}_{\Perf(\A^1/\GG_m)}
\bigt{\C_0' \otimes \Perf(\A^1/\GG_m) }.
$$
In the above formula, contrarily to our usage, we have considered exact endofunctors of the non-cocomplete DG category $\C_0' \otimes \Perf(\A^1/\GG_m) $.
By ind-extending $\mu'$, we obtain an object
$$
\wt\mu
\in 
\Alg
\bigt{
\End_{\QCoh(\A^1/\GG_m)}
\bigt{\Ind(\C_0) \otimes \QCoh(\A^1/\GG_m) }
}.
$$
Since $\Ind(\C_0) = \Ind(\C_0')$, the latter is the monad with filtration we were looking for.
\end{proof}

\sssec{}

We apply the above paradigm to $\mu = \U(\fL)$ acting on $\C = \ICoh(\Y)$ and we choose $\C_0$ to be the subcategory $\ICoh(\Y)^{\on{dualiz}}$ of dualizable objects in $\ICoh(\Y)$, with respect to the $\stackrel ! \otimes$ symmetric monoidal structure. By \cite[Remark 5.4.4]{ker-adj-duality}, the adjunction
$$
\Upsilon_\Y : \Perf(\Y)
\rightleftarrows
\ICoh(\Y)^{\on{dualiz}}:
\Phi_\Y := (\Upsilon_\Y)^R
$$
is an equivalence.
Thus, we just need to check that each $\U(\fL)^{\leq n}: \ICoh(\Y) \to \ICoh(\Y)$ preserves the subcategory $\ICoh(\Y)^{\on{dualiz}} \simeq \Upsilon_\Y(\Perf(\Y))$. This is obvious: the $n$-associated graded piece is the functor
$$
\Sym^n(\oblv_{\on{Lie-algbd}}(\fL)) \stackrel ! \otimes - :
\ICoh(\Y) \longto \ICoh(\Y)
$$ 
and $\oblv_{\on{Lie-algbd}}(\fL)$ belongs to $\Upsilon_\Y(\Perf(\Y))$ by assumption.

\sssec{}

Through the above construction, the functor (\ref{funct:Lie algebroids to monads wihth filtration on ICOH}) yields a functor
\begin{equation}  \label{funct: !-perfect Lie algebroids to monads wihth filtration on QCOH}
\UQ: 
\LieAlgbd(\Y) 
\ustimes{\ICoh(\Y)}
\ICoh(\Y)^{\on{dualiz}}
 \longto 
 \Alg(\End(\QCoh(\Y))^{\Fil, \geq 0}).
\end{equation}
Explicitly, the functor underlying the monad $\UQ(\fL)$ is the unique endofunctor of $\QCoh(\Y)$ whose restriction to $\Perf(\Y)$ is given by
$$
\Perf(\Y) \longto \QCoh(\Y),
\hspace{.4cm}
P \squigto 
\colim_{n\geq 0 } \Phi_\Y \circ \U^{\leq n}(\fL) \circ \Upsilon_\Y(P).
$$

\begin{rem}
Since $\Phi_\Y$ is discontinuous for $\Y$ unbounded, $\UQ(\fL)$ is different (in general) from the ind-extension of $P \squigto  \Phi_\Y \U(\fL) \Upsilon_\Y(P)$.
\end{rem}

The construction of $\UQ$ allows us to extend the definition of $\ICoh_0$ to the unbounded case as follows.

\begin{defin}[$\ICoh_0$ in the unbounded case] \label{defn:ICohzero unbounded case}

For $\Y \to \Z$ in $\Stk_\lfp$, with $\Y$ perfect but not necessarily bounded, we define
\begin{equation} \label{eqn:ICOHzero-unbounded case}
\ICoh_0(\Z^\wedge_\Y)
:=
\UQ(\Tang_{\Y/\Z}) \mod (\QCoh(\Y)).
\end{equation}
Clearly, such DG category is compactly generated and hence dualizable.

\end{defin}

\ssec{Basic functoriality}

\sssec{}

By construction, the assignment $\ICoh_0(-^\wedge_\W)$ underlies the structure of a covariant functor $\Stk_{\W/} \to \DGCat$.
Explicitly, a string $\W \to \X \to \Y$ in $\Stklfp$ gives a canonical map of Lie algebroids $\Tang_{\W/\X} \to \Tang_{\W/\Z}$ on $\W$ (equivalently, a map of $\xi: \X^\wedge_\W \to \Y^\wedge_\W$ of nil-isomorphisms under $\W$). Induction along the resulting algebra arrow
$$
\UQCoh(\Tang_{\W/\X} )
\longto
\UQCoh(\Tang_{\W/\Z} )
$$
yields the structure functor
$$
\xi_{*,0}:
\ICoh_0(\X^\wedge_\W) 
\longto 
\ICoh_0(\Y^\wedge_\W).
$$

\sssec{}

Since the functor $\xi_{*,0}$ preserves compact objects, it admits a continuous right adjoint that we shall denote by
$$
\xi^{?}: 
\ICoh_0(\Y^\wedge_\W) 
\longto 
\ICoh_0(\X^\wedge_\W).
$$
This is just the forgetful functor along the above maps of universal enveloping algebras.

\begin{rem}
If $\W$ is bounded, then $\xi^?$ is simply the $(!,0)$-pullback, while $\xi_{*,0}$ reduces to its already established meaning.
\end{rem} 

\begin{example}
For $\W= \X$, we just have a map $f: \X \to \Z$ in $\Stkperflfp$ and we rediscover the defining monadic adjunction
\begin{equation} \label{adj:monadic in general case}
\begin{tikzpicture}[scale=1.5]
\node (a) at (0,1) {$(\primef)_{*,0} : \QCoh(\X)$};
\node (b) at (3,1) {$\ICoh_0(\Y^\wedge_{\X}) : (\primef)^{?}$.};
\path[->,font=\scriptsize,>=angle 90]
([yshift= 1.5pt]a.east) edge node[above] {$ $} ([yshift= 1.5pt]b.west);
\path[->,font=\scriptsize,>=angle 90]
([yshift= -1.5pt]b.west) edge node[below] {$ $} ([yshift= -1.5pt]a.east);
\end{tikzpicture}
\end{equation}

\end{example}

\sssec{}

For $\Y$ bounded, $\ICoh_0(\Z^\wedge_\Y)$ has been defined as a full subcategory of $\ICoh(\Z^\wedge_\Y)$. 
For $\Y$ unbounded, we show there still is a canonical arrow $\ICoh_0(\Z^\wedge_\Y) \to \ICoh(\Z^\wedge_\Y)$, which is no longer an inclusion.

\begin{lem}
Let $f: \Y \to \Z$ be a morphism in $\Stkperflfp$. There is a natural isomorphism
$$
 \Upsilon_\Y \circ \UQ(\Tang_{\Y/\Z})
 \longto
\U(\Tang_{\Y/\Z}) \circ \Upsilon_\Y
$$
of non-negatively filtered functors from $\QCoh(\Y)$ to $\ICoh(\Y)$.
\end{lem}

\begin{proof}
If suffices to construct a filtered isomorphism 
$$
  \restr{\Upsilon_\Y 	\circ \UQ(\Tang_{\Y/\Z}) }{\Perf(\Y)}
  \longto
  \restr{\U(\Tang_{\Y/\Z}) \circ \Upsilon_\Y}{\Perf(\Y)}
$$
between the restrictions of the above functors to $\Perf(\Y)$. By definition, this amounts to giving a compatible $\NN$-family of isomorphisms
$$
 \restr{\Upsilon_\Y 	\circ \Phi_\Y\circ \U^{\leq n}(\Tang_{\Y/\Z})  \circ \Upsilon_\Y}{\Perf(\Y)}
 \longto
 \restr{\U^{\leq n}(\Tang_{\Y/\Z}) \circ \Upsilon_\Y}{\Perf(\Y)}
.
$$
These isomorphisms are manifest since $\U^{\leq n}(\Tang_{\Y/\Z})$ preserves $\Upsilon_\Y(\Perf(\Y))$.
\end{proof}

\sssec{} \label{sssec: passage from left to right unbdd setting}

This lemma shows that $\Upsilon_\Y: \QCoh(\Y) \to \ICoh(\Y)$ upgrades to a functor
$$
\Upsilon_{\Y/\Z}^{\fD}:
\ICoh_0(\Z^\wedge_\Y) \longto \ICoh(\Z^\wedge_\Y)
$$
sitting in the commutative diagram
\begin{equation} 
\nonumber
\begin{tikzpicture}[scale=1.5]
\node (00) at (0,0) {$\QCoh(\Y)$};
\node (10) at (2.3,0) {$\ICoh(\Y)$.};
\node (01) at (0,1) {$\ICoh_0(\Z^\wedge_\Y) $};
\node (11) at (2.3,1) {$\ICoh(\Z^\wedge_\Y)$};
\path[->,font=\scriptsize,>=angle 90]
(00.east) edge node[above] {$\Upsilon_\Y$} (10.west);
\path[->,font=\scriptsize,>=angle 90]
(01.east) edge node[above] {$\Upsilon^\fD_{\Y/\Z}$} (11.west);
\path[->,font=\scriptsize,>=angle 90]
(01.south) edge node[left] {$\oblv = (f)^{?}$} (00.north);
\path[->,font=\scriptsize,>=angle 90]
(11.south) edge node[right] {$\oblv = ('f)^!$} (10.north);
\end{tikzpicture}
\end{equation}
In view of the formulas
$$
\ICoh_0(\Z^\wedge_\Y) 
:= \UQ(\Tang_{\Y/\Z}) \mod (\QCoh(\Y))
\hspace{.4cm}
\ICoh(\Z^\wedge_\Y) 
:= \U(\Tang_{\Y/\Z}) \mod (\ICoh(\Y)),
$$
the functor $\Upsilon^\fD_{\Y/\Z}$ might regarded as the passage from left to right relative $\fD$-modules. When $\Z = \pt$, we write $\Upsilon^\fD_{\Y}$ rather than $\Upsilon^\fD_{\Y/\pt}$.

\sssec{}

As we see in the section, a special case (the absolute case where $\Z = \pt$) of Definition \ref{defn:ICohzero unbounded case} yields a new notion of $\fD$-module on a derived stack.

\begin{defin} [Derived D-modules]
For  $\Y \in \Stkperflfp$, we set
$$
\vDmod(\Y) := \ICoh_0(\pt^\wedge_\Y),
$$
that is, 
$$
\vDmod(\Y)  := \UQ(\Tang_{\Y}) \mod (\QCoh(\Y)).
$$
\end{defin}

\begin{rem}

The notation $\vDmod$ is in place to distinguish $\vDmod(\Y)$ from the usual DG category
$$
\Dmod(\Y) := \ICoh(\pt^\wedge_\Y) \simeq \U(\Tang_{\Y}) \mod (\ICoh(\Y))
$$
of $\fD$-modules on $\Y$.

\end{rem}

\ssec{D-modules and derived D-modules} \label{ssec: derived D-modules}

In this section, we compute the DG category $\vDmod(Y)$ for $Y = \Spec(A)$ an affine DG scheme locally of finite presentation.
We will first perform the computation in general and then apply it to the affine schemes of Example \ref{example:Koszul duality}.

\sssec{}

Since $Y$ is affine, $\vDmod(Y)$ is the DG category of modules over the DG algebra 
$$
\Gamma(Y, \UQ(\Tang_{Y})(\O_{Y})).
$$
To describe this DG algebra, we need to compute the monad $\U(\Tang_{Y})$, understand its PBW filtration, and then do the filtered renormalization.

\sssec{}

For any $m \geq 1$, denote by 
$$
\Delta_m \xto{s_m \times t_m} Y \times Y
$$
the $m^{th}$ infinitesimal neighbourhood of the diagonal, so that $(Y \times Y)^\wedge_Y \simeq \colim_m \Delta_m$. Each $\Delta_m$ is a formal groupoid acting on $Y$; we form the groupoid quotient $Y^{(m)} := Y/\Delta_m$ and let $p^{(m)} : Y \to Y^{(m)}$ be the structure projection map.

\sssec{}

By definition, the monad $\U(\Tang_{Y})$ is the push-pull $p^! p_*^\ICoh$ along $p: Y \to Y_\dR$. According to \cite[Volume 2, Chapter 9, Section 6.5]{Book}, the filtration on $\U(\Tang_Y)$ is induced by the filtration of the groupoid $(Y \times Y)^\wedge_Y$ by the infinitesimal neighbourhoods of the diagonal. Namely, it is given by the sequence
$$
m \squigto  (p^{(m)})^! (p^{(m)})_*^\ICoh
\simeq
(s_m)_*^{\ICoh} (t_m)^!.
$$
Thus, to compute $\UQ(\Tang_Y)$, we need to compute the filtered monad
$$
m \squigto \Phi_Y(s_m)_*^{\ICoh} (t_m)^!   \Upsilon_Y.
$$
Similarly, for the filtered algebra $\Gamma(Y, \UQ(\Tang_{Y})(\O_{Y}))$, we need to understand the filtered monad structure on
$$
m \squigto
W^{\leq m}
:=
 \CHom_{\ICoh(Y)}
\bigt{ 
\omega_Y, (s_m)_*^{\ICoh} (t_m)^! (\omega_Y)
}
$$

\sssec{}

In all cases, such structure is induced via base-change by the structure maps
$$
\alpha_{m,p}:\Delta_m \times_Y \Delta_p \to \Delta_{m+p}
$$
over $Y \times Y$, together with the similar ones for higher compositions.\footnote{We are just making explicit the fact that the functor $m \squigto \Delta_m$ is an algebra object in the (Day convolution) monoidal $\infty$-category of sequences of groupoids over $Y$.}
Now observe that each $\Delta_m$ is an affine DG scheme, say $\Delta_m \simeq \Spec(A_m)$: indeed each of them is a square-zero extension of the preceding one. Hence the maps $\alpha_{m,p}$ correspond to $(A,A)$-bilinear algebra maps $A_{m+p} \to A_m \otimes_A A_p$ going the other way.

\sssec{}

Let us now begin the computation of $W^{\leq m}$. We have
$$
W^{\leq m}:=
 \CHom_{\ICoh(Y)}
\bigt{ 
\omega_Y, (s_m)_*^{\ICoh} (t_m)^! (\omega_Y)
}
\simeq
 \CHom_{\ICoh(Y)}
\bigt{ 
\omega_Y, (s_m)_*^{\ICoh}(\omega_{\Delta_m})
},
$$
where we recall that $s_m: \Delta_m \to Y$ is the first of the two structure maps. It is clear that $s_m$ is proper and in fact finite. Before continuing, we need to establish another property of $s_m$.

\begin{lem}
The map $s_m$ is bounded (or eventually coconnective, in the terminology of \cite[Section 7.2]{ICoh}).
\end{lem}

\begin{proof}
In view of the reformulations given immediately after \cite[Definition 3.5.2]{ICoh}, it is enough to show that $R \otimes_A A_m$ is cohomologically bounded for any algebra map $A \to R$ with $R \simeq H^0(R)$. Since $A_m$ is constructed inductively as a sequence of square-zero extensions by symmetric powers of $\LL_A$, it suffices to show that $R \otimes_A \Sym^m \LL_A$ is cohomologically bounded. This is clear: the cotangent complex $\LL_A$ is perfect by assumption and so $R \otimes_A \Sym^m \LL_A$ is a perfect $R$-module, whence bounded (as $R$ is classical).
\end{proof}

\sssec{}

Since $s_m$ is proper and bounded, we can invoke \cite[Propositions 7.2.2 and 7.2.9(a)]{ICoh} to conclude that $(s_m)_*: \QCoh(\Delta_m) \to \QCoh(Y)$ admits a continuous right adjoint $(s_m)^{!,\QCoh}$ sitting in the commutative diagram
\begin{equation} 
\nonumber
\begin{tikzpicture}[scale=1.5]
\node (00) at (0,0) {$\ICoh(\Delta_m)$};
\node (10) at (2.3,0) {$\QCoh(\Delta_m)$.};
\node (01) at (0,1) {$\ICoh(Y)$};
\node (11) at (2.3,1) {$\QCoh(Y)$};
\path[->,font=\scriptsize,>=angle 90]
(00.east) edge node[above] {$ \Psi_{\Delta_m}$} (10.west);
\path[->,font=\scriptsize,>=angle 90]
(01.east) edge node[above] {$\Psi_Y$} (11.west);
\path[->,font=\scriptsize,>=angle 90]
(01.south) edge node[left] {$(s_m)^{!}$} (00.north);
\path[->,font=\scriptsize,>=angle 90]
(11.south) edge node[right] {$(s_m)^{!,\QCoh}$} (10.north);
\end{tikzpicture}
\end{equation}
By passing to dual functors, we deduce that $(s_m)^*$ admits a \emph{left} adjoint, to be denoted 
$$
(s_m)_!: \QCoh(\Delta_m) \to \QCoh(Y),
$$
which gets intertwined to $(s_m)_*^{\ICoh}$ by the $\Upsilon$ functors. In other words,
$$
(s_m)_*^{\ICoh}(\omega_{\Delta_m})
\simeq
(s_m)_*^{\ICoh}\Upsilon_{\Delta_m}(\O_{\Delta_m})
\simeq
\Upsilon_Y (s_m)_! (\O_{\Delta_m}).
$$

\sssec{}

We obtain that 
$$
W^{\leq m} 
\simeq
 \CHom_{\ICoh(Y)}
\bigt{ 
\omega_Y, \Upsilon_Y (s_m)_! (\O_{\Delta_m}))
}
\simeq
 \CHom_{\QCoh(Y)}
\bigt{ 
\O_Y, (s_m)_! (\O_{\Delta_m})
},
$$
where the last step used the fact that $(s_m)_! (\O_{\Delta_m})$ is automatically perfect and that $\Upsilon_Y$ is fully-faithful on $\Perf(Y)$ even if $Y$ is not bounded (see Section \ref{sssec: Ups fully fiath on perf}).

\sssec{}

The naive duality on $\Perf(Y)$ further yields
$$
W^{\leq m}
\simeq
 \CHom_{\QCoh(Y)}
\bigt{ 
((s_m)_! \O_{\Delta_m})^\vee, \O_Y
}
\simeq
 \CHom_{\QCoh(Y)}
\bigt{ 
(s_m)_* \O_{\Delta_m}, \O_Y
}
\simeq
 \CH om_{A\mod}
\bigt{ 
A_m, A
} :=(A_m)^*.
$$
Hence
$$
\Gamma(Y, \UQ(\Tang_{Y})(\O_{Y}))
\simeq 
W_A :=
\uscolim{m \geq 1} \, W^{\leq m}
\simeq
\uscolim{m \geq 1} \,
(A_m)^*,
$$
with algebra structure induced by the $A$-linear duals of the system of maps $A_{m+p} \to A_m \otimes_A A_p$.

\begin{example}

Consider $Y = \A^1 = \Spec(\kk[x])$. Then $\Delta_m \simeq \Spec A_m$ with $A_m = \kk[u,v]/(u-v)^{m+1}$. Let us declare that $s_m$ is the map corresponding to the natural map $\kk[v] \to A_m$. Then, changing variables $(u,v) \mapsto (u-v, v)$, we obtain that
$$
(A_m)^* \simeq \bigoplus_{i=0}^{m}  \partial^i \cdot \kk[v],
$$
as right $\kk[v]$-modules. An easy check shows that the left $\kk[v]$-module structure is the one giving rise to the Weyl algebra.
By taking products, an analogous result holds for $\vDmod(\A^p)$.
\end{example}

\begin{example} \label{example: Weyl algebras on Yn}

Consider now the derived affine scheme $Y_n = \Spec (\kk[u])$, where $n \geq 1$ and $u$ is a variable in cohomological degree $-n$.
In this case, $\Dmod(Y_n) \simeq \Vect$ since the classical truncation of $Y_n$ is just a point. 
One shows that $\vDmod(Y_n)$ is equivalent to $\kk \langle u, \partial_u \rangle \mod$, where $\kk \langle u, \partial_u \rangle$ is the Weyl graded algebra built on $\kk[u]$. Here $\partial_u$ has cohomological degree $n$.
This follows, as above, from the knowledge of the inifinitesimal neighbourhoods of the diagonal: they are obtained from $\kk[u_1, u_2] = H^*(Y_n \times Y_n, \O)$ by attaching a cell that kills off $(u_1 - u_2)^{m+1}$.
By taking products, an analogous result holds true for $\vDmod(\Spec(\Sym V^*[n]))$.
\end{example}

\sssec{}

In the setup of Example \ref{example: Weyl algebras on Yn}, there is a huge difference between the even and the odd cases. Set $W_n := \kk \langle u, \partial_u \rangle$.
Unraveling the constructions, the natural functor 
$$
\Upsilon_{Y_n}^\fD:
\vDmod(Y_n) \longto \Dmod(Y_n) \simeq \Vect
$$
goes over to the functor 
$$
\phi_{n}:
W_n \mod \xto{  \kk[\partial_u] \usotimes{W_n} -  } \kk \mod,
$$
where the right $W_n$-module structure on $\kk[\partial_u] $ is the one induced by the isomorphism $\kk[\partial_u] 
  \simeq \kk \otimes_{\kk[u]} W_n$.

\begin{cor} \label{cor:vDmod on Yn}
The functor $\Upsilon_{Y_n}^\fD$ is an equivalence if and only if $n$ is odd, that is, if and only if $Y_n$ is bounded.
\end{cor}

\begin{proof}
When $n$ is odd, $Y_n$ is bounded and thus the functor in question is an equivalence by the general theory. In particular this implies that the Weyl algebra on an odd vector space is Morita equivalent to $\kk$.
If $n = 2m$ is even, then we claim that the functor is not an equivalence. While this can proven directly, we prefer to give a quick argument that uses the shift of grading trick introduced in \cite[Appendix A]{AG}.\footnote{See \cite[Section 13.4]{ShvCat}, \cite[Section 1]{strong-gluing}, \cite[Sections 3.2, 3.3]{antitemp}, \cite[Section 3.2]{DL-spectral} for other applications of this trick.}
Since $\phi_n$ is $\GG_m$-equivariant, we can apply the shift of grading $m$ times to cancel the shifts. This shows that $\phi_{n}$ is an equivalence if and only if so is $\phi_{0}$. But the latter is the functor $\Dmod(\A^1) \to \Vect$ of $!$-restriction at $0$, which is obviously not an equivalence.
\end{proof}

\sssec{} \label{sssec:Weyl alg da fare}

Let us return to the case of a general affine scheme $Y = \Spec(A)$.
In regard to the formula $\vDmod(Y) \simeq W_A \mod$ established above, we always expect $W_A$ to be a Weyl algebra. For instance, if $A = (\kk[x_i], d)$ is a quasi-free commutative DG algebra with finitely many variables, we expect $W_A$ to look as follows:
\begin{itemize}
\item
as graded vector space, $W_A := \kk[x_i, \partial_i]$, with $\deg(\partial_i) = -\deg(x_i)$;
\item
the differential $d_W$ of $W_A$ is determined by
$$
d_W(x_i) = d(x_i) =: f_i, 
\hspace{.4cm}
d_W(\partial_i) = \Sum_j \frac{\partial f_j}{\partial x_i} \partial_j;
$$
\item
the algebra structure is determined by the super-commutation relations $[x_i, x_j]= 0$, $[\partial_i, \partial_j] = 0$,  $[\partial_i, x_j] = \delta_{ij}$.
\end{itemize}
We defer this computation to a future work.

\begin{cor}
As above, let $Y=\Spec(A)$ with $A = (\kk[x_i], d)$ is a quasi-free commutative DG algebra. Then we have $\Dmod(Y) \simeq W_B \mod$, where $B$ is the quasi-free commutative DG algebra obtained from $A$ by discarding all $x_i$ of degrees $\leq -2$. 
\end{cor}

\begin{proof}
By definition, $\Dmod(Y) \simeq \Dmod(Y')$ for any map $Y \to Y'$ that is an isomorphism on the classical truncations of $Y$ and $Y'$. This is the case for the map $\Spec(B) \to Y$ induced by the projection $A \to B$, whence $\Dmod(Y) \simeq \Dmod(\Spec(B))$. Since $\Spec(B)$ is quasi-smooth, it is bounded; we deduce that $\Dmod(\Spec(B)) \simeq \vDmod(\Spec(B))$. It remains to apply Section \ref{sssec:Weyl alg da fare}.
\end{proof}

\ssec{Descent in the \virg{second variable}}

In this section, we generalize Corollary \ref{cor:descent for ?-pullback in second variable}.

\begin{prop} \label{prop:descent for ICOHzero-general-case}
For any fixed $\W \in \Stkperflfp$, the contravariant functor $\ICoh_0((-)^\wedge_\W)$, under $?$-pullbacks, satisfies descent along \emph{any} map in $(\Stkperflfp)_{\W/}$. 
\end{prop}

\begin{proof}
Let $\W \to \Y \to \Z$ be a map in $(\Stkperflfp)_{\W/}$.
Consider the Cech complex $(\Y^{\times_\Z(\bullet+1)})^\wedge_\W$ of the map $\Y^\wedge_\W \to \Z^\wedge_\W$ and the resulting pullback functor
$$
\ICoh_0(\Z^\wedge_\W)
\longto 
\Tot \Bigt{ \ICoh_0
\bigt{ (\Y^{\times_\Z(\bullet+1)})^\wedge_\W 
}
}.
$$
We need to show that such functor is an equivalence. By passing to left adjoints, this amounts to showing that the arrow
$$
\uscolim{[n] \in \bDelta} \Bigt{ \ICoh_0
\bigt{ (\Y^{\times_\Z(\bullet+1)})^\wedge_\W 
}
}
\longto
\ICoh_0(\Z^\wedge_\W)
$$
is an equivalence, where now the structure maps forming the colimit are given by the $(*,0)$-pushforward functors (that is, induction along the maps between the universal envelopes).
Hence, it suffices to show that the natural arrow
$$
\uscolim{[n] \in \bDelta} \,
\UQ(\Tang_{\W/ {\Y^{\times_\Z(n+1)}}   }) 
\longto
\UQ(\Tang_{\W/ {\Z}}  ),
$$
taking place in $\Alg(\End(\QCoh(\W))^{\Fil, \geq 0})$, is an isomorphism. Forgetting the monad structure is conservative, whence we will just prove that the arrow above is an isomorphism in $\End(\QCoh(\W))^{\Fil, \geq 0}$ (i.e., that it is an isomorphism of filtered endofunctors).

\medskip

Since the filtrations in questions are \emph{non-negative}, it is enough to prove the isomorphism separately for each component of the associated graded. 
Recall that the $j^{th}$-associated graded of $\UQ(\fL)$ is the functor $\Sym^j(\Phi_\W(\fL)) \otimes -: \QCoh(\W) \to \QCoh(\W)$.
Thus, we are to prove that the natural map
$$
\uscolim{[n] \in \bDelta} \,
\Sym^j(\TangQ_{\W/ {\Y^{\times_\Z(n+1)}}   }) 
\longto
\Sym^j (\TangQ_{\W/Z})
$$
is an isomorphism in $\QCoh(\W)$ for each $j \geq 0$. Since $\Sym$ commutes with colimits, it suffices to show that 
$$
\uscolim{[n] \in \bDelta} \,
\TangQ_{\W/ {\Y^{\times_\Z(n+1)}}   }
\longto
\TangQ_{\W/\Z}
$$
is an isomorphism. We will show that the cone of such map is zero. First, with no loss of generality, we may assume that $\W = \Y$. Then we compute the cone in question as
$$
\uscolim{[n] \in \bDelta^\op} \,  (\TangQ_{\Y/\Z})^{\oplus(n+1)},
$$
and this expression is manifestly isomorphic to the zero object of $\QCoh(\Y)$: indeed, the simplicial object in question is the Cech nerve of the map $\TangQ_{\Y/\Z} \to 0$ in $\QCoh(\Y)$.
\end{proof}

\ssec{Tensor products of $\ICoh_0$ over $\QCoh$}

In this section, we show that formation of $\ICoh_0$ behaves well with respect to fiber products.

\begin{lem} \label{lem: UQCoh compatible with pullbacks}
Let $\X \to \Z \leftto \Y$ be a diagram in $\Stkperflfp$ and denote by $p: \X \times_\Z \Y \to \Y$ the natural map. There is a natural isomorphism
\begin{equation} \label{eqn:pullback-of-UQCoh}
\UQ(\Tang_{\X \times_\Z \Y/\X}) \circ p^*
\longto
p^* \circ \UQ(\Tang_{\Y/\Z})
\end{equation}
in the DG category $\Fun(\QCoh(\Y),\QCoh(\X \times_\Z \Y))^{\Fil, \geq 0}$.
\end{lem}

\begin{proof}
We need to exhibit a compatible $\NN$-family of isomorphisms
\begin{equation} \nonumber
\UQ(\Tang_{\X \times_\Z \Y/\X})^{\leq n } \circ p^*
\longto
p^* \circ \UQ(\Tang_{\Y/\Z})^{\leq n}.
\end{equation}
By the continuity of these functors and perfection of $\Y$, it suffices to exhibit a compatible $\NN$-family of isomorphisms
\begin{equation} \nonumber
\restr{ \UQ(\Tang_{\X \times_\Z \Y/\X})^{\leq n } \circ p^*}{\Perf(\Y)}
\longto
p^* \circ \restr{\UQ(\Tang_{\Y/\Z})^{\leq n}}{\Perf(\Y)}.
\end{equation}
When restricted to $\Perf(\Y) \subset \QCoh(\Y)$, the LHS can be rewritten as 
$$
\Phi_{\X \times_\Z \Y}
\circ
\U(\Tang_{\X \times_\Z \Y/\X})^{\leq n}
\circ
p^!
\circ
\Upsilon_\Y
$$
and the RHS as
$$
\Phi_{\X \times_\Z \Y}
\circ
p^!
\circ
\U(\Tang_{\Y/\Z})^{\leq n}
\circ
\Upsilon_\Y.
$$
It then suffices to give a compatible $\NN$-family of isomorphisms
\begin{equation} \label{eqn:family-functors-truncated-univ-env}
p^! \circ 
\U(\Tang_{\Y/\Z})^{\leq n}
\longto
\U(\Tang_{\X \times_\Z \Y/\X})^{\leq n} \circ p^!
\end{equation}
of functors $\ICoh(\Y) \to \ICoh(\X \times_\Z \Y)$.

\medskip

By (\cite[Volume 2, Chapter 9, Section 6.5]{Book}), for a Lie algebroid $\L$ in $\ICoh(\Y)$, the functor $\U(\fL)^{\leq n}$ can be written using the $n^{th}$ infinitesimal neighbourhood of the formal groupoid associated to $\fL$.
In our case, let $\wt\V^{(n)}$ the $n^{th}$ infinitesimal neighbourhood attached to $\Tang_{\X \times_\Z \Y/\X} \to \Tang_{\X \times_\Z \Y}$, equipped with its two structure maps $\wt p_s, \wt p_t: \wt\V^{(n)} \rr \X \times_\Z \Y$; then
$$
\U(\Tang_{\X \times_\Z \Y/\X})^{\leq n}  
\simeq
(\wt p_s)_*^\ICoh \circ (\wt p_t)^!.
$$
Similarly, let $p_s,  p_t: \V^{(n)} \rr \Y$ be the same data for $\Tang_{\Y/\Z}m\to \Tang
_\Y$, so that
$$
\U(\Tang_{\Y/\Z})^{\leq n}  
\simeq
( p_s)_*^\ICoh \circ (p_t)^!.
$$
By the very construction of $n^{th}$ infinitesimal neighbourhoods, we have canonical isomorphisms
$$ 
\wt\V^{(n)} 
\simeq
(\X \times_\Z \Y) \times_{\Y, p_s} \V^{(n)}
\simeq
\V^{(n)} \times_{p_t, \Y} (\Y \times_\Z \X),
$$
which are compatible with varying $n$. 
Hence, the compatible isomorphisms (\ref{eqn:family-functors-truncated-univ-env}) come from base-change for ind-coherent sheaves.
\end{proof}

\begin{cor} \label{cor:commutation of UQCOH with push}
With the notation of the above lemma, assume furthermore that at least one of the following two requirements is satisfied:
\begin{itemize}
\item
the map $\X \to \Z$ is affine (more generally, we just need that $p_*: \QCoh(\X \times_\Z \Y) \to \QCoh(\Y)$ be right t-exact up to a finite shift);
\item
$\Y$ is bounded.
\end{itemize}
Then the arrow
$$
\UQ(\Tang_{\Y/\Z})  \circ p_* 
\longto
p_*  \circ \UQ(\Tang_{\X \times_\Z \Y/\X}),
$$
obtained from (\ref{eqn:pullback-of-UQCoh}) by adjunction, is an isomorphism of filtered functors from $\QCoh(\X \times_\Z \Y)$ to $\QCoh(\Y)$.
\end{cor}

\begin{proof}
As the arrow in question is the colimit of the $\NN$-family
$$
\UQ(\Tang_{\Y/\Z})^{\leq n} \circ p_* 
\longto
p_*  \circ \UQ(\Tang_{\X \times_\Z \Y/\X})^{\leq n},
$$
it suffices to prove the assertion separately for each piece of the associated graded.
For each $n\geq 0$, the map in question is
$$
\Phi_\Y \Bigt{  \Sym^n(\Tang_{\Y/\Z}) \overset ! \otimes \Upsilon_\Y ( p_*(-)) } 
\longto
p_* \Phi_{\X \times_\Z \Y} \Bigt{ \Sym^n(\Tang_{\X \times_\Z \Y/\X}) \overset ! \otimes \Upsilon_{\X \times_\Z \Y}  (-)}.
$$
Let us now finish the proof in the situation of the first assumption, the argument for the second one is easier.
It suffices to check the isomorphism after restricting both sides to $\Perf(\X \times_\Z \Y)$, in which case we are dealing with the arrow
$$
\restr{\Phi_\Y \Upsilon_\Y \Bigt{  \Sym^n(\TangQ_{\Y/\Z}) \otimes  p_*(-) } }{\Perf(\X \times_\Z \Y)}
\longto
\restr{p_* \Bigt{ \Sym^n(\TangQ_{\X \times_\Z \Y/\X}) \otimes  - }
}{\Perf(\X \times_\Z \Y)}.
$$
Now the assertion follows from the projection formula and the fact (see Section \ref{sssec: Ups fully fiath on perf}) that $\Upsilon_\Y$ is fully faithful on the full subcategory of $\QCoh(\Y)$ consisting of eventually connective objects.
\end{proof}

\begin{cor} \label{cor:commuting monads.}
In situation of Corollary \ref{cor:commutation of UQCOH with push}, the two monads $\UQ(\Tang_{\Y/ \Z})$ and $p_* p^*$ commute.
\end{cor}

\sssec{}

We now generalize the simplest instance of Proposition \ref{prop:exterior-product-ICohzero-over-QCoh-ev-coc-case}.

\begin{prop} \label{prop:exterior-product-ICohzero-over-QCoh}
Let $\Y \to \Z \leftto \V \leftto \U$ be a diagram in $\Stkperfevcoclfp$. Note that we do not assume that $\Y \times_\Z \U$ be bounded. Then the exterior product yields an equivalence
\begin{equation} 
\QCoh(\Y)
\usotimes{\QCoh(\Z)}
\ICoh_0(\V^\wedge_\U)
\xto{\;\;\simeq\;\; }
\ICoh_0 \Bigt{ 
(\Y \times_{\Z} \V)
^\wedge_
{(\Y \times_{\Z} \U)}
}.
\end{equation}
\end{prop}

\begin{proof}
Without loss of generality, we may assume that $\Z= \V$. Thus, for a diagram $\Y \to \Z \leftto \U$, we need to construct a $\QCoh(\Y)$-linear equivalence
 \begin{equation} \label{eqn:exterior-product-ICohzero-over-QCoh-unbounded}
\QCoh(\Y)
\usotimes{\QCoh(\Z)}
\ICoh_0(\Z^\wedge_\U)
\xto{\;\;\simeq\;\; }
\ICoh_0 \Bigt{ 
\Y
^\wedge_
{(\Y \times_{\Z} \U)}
}.
\end{equation}

Both categories are modules for monads acting on $\QCoh(\U)$ (this is true thanks to the hypothesis of affineness), so it suffices to construct a map between those monads and check it is an isomorphism.

\medskip

Let $p: \U \times_\Z \Y \to \U$ denote the obvious projection.
The two monads in questions are 
$$
p_* \circ p^*  \circ \UQ(\Tang_{\U/\V})
$$
and
$$
 p_* \circ \UQCoh(\Tang_{\U \times_\Z \X/ \V \times_\Z \Y}) \circ p^*
$$
Note that the monad structure on the former functor has been discussed in Corollary \ref{cor:commuting monads.}.

\medskip

By assumption,  $\QCoh(\U \times_\Z \Y)$ is compactly generated by objects of the form $p^*P$ for $P \in \Perf(\U)$.
Now the assertion follows from Lemma \ref{lem: UQCoh compatible with pullbacks}.
\end{proof}

\ssec{The exceptional pull-back and push-forward functors}

Let us now generalize Section \ref{sssec:exceptional pull and push in bounded case}. 

\sssec{}

Given maps $\X \to \Z \leftto \Y$ in $\Stkperflfp$, we regard the resulting cartesian diagram
\begin{equation} 
\nonumber
\begin{tikzpicture}[scale=1.5]
\node (00) at (0,0) {$\X$};
\node (10) at (1.5,0) {$\Z$,};
\node (01) at (0,.8) {$\X \times_\Z \Y $};
\node (11) at (1.5,.8) {$\Y$};
\path[->,font=\scriptsize,>=angle 90]
(00.east) edge node[above] {$g$} (10.west);
\path[->,font=\scriptsize,>=angle 90]
(01.east) edge node[above] {$G$} (11.west);
\path[->,font=\scriptsize,>=angle 90]
(01.south) edge node[right] {$F$} (00.north);
\path[->,font=\scriptsize,>=angle 90]
(11.south) edge node[right] {$f$} (10.north);
\end{tikzpicture}
\end{equation}
as a morphism 
$$
\eta:
[\X \times_\Z \Y \to \X]
\longto
[\Y \to \Z]  
$$ 
in $\Arr(\Stkperflfp)$. We emphasize that none of the stacks in question is required to be bounded.
In this situation, we define the adjunction
\begin{equation} \label{adj:exceptional-adjunction-unbounded}
\begin{tikzpicture}[scale=1.5]
\node (a) at (0,1) {$\eta^{!,0}: \ICoh_0(\Z^\wedge_\Y)   $};
\node (b) at (3,1) {$  \ICoh_0(\X^\wedge_{\X \times_\Z \Y}) : \eta_?$.};
\path[->,font=\scriptsize,>=angle 90]
([yshift= 1.5pt]a.east) edge node[above] {  } ([yshift= 1.5pt]b.west);
\path[->,font=\scriptsize,>=angle 90]
([yshift= -1.5pt]b.west) edge node[below] { } ([yshift= -1.5pt]a.east);
\end{tikzpicture}
\end{equation}
exactly as in Section \ref{sssec:exceptional pull and push in bounded case}, using the equivalence
$$
\ICoh_0(\X^\wedge_{\X \times_\Z \Y})
\simeq
\QCoh(\X)
\usotimes{\QCoh(\Z)}
\ICoh_0(\Z^\wedge_\Y)
$$
proven in Proposition \ref{prop:exterior-product-ICohzero-over-QCoh}.

\sssec{}
Tautologically, the functors $\eta^{!,0}$ and $\eta_?$ fit in the commutative diagrams
\begin{equation} 
\nonumber
\begin{tikzpicture}[scale=1.5]
\node (00) at (0,0) {$\QCoh(\Y)$};
\node (10) at (3,0) {$\ICoh_0(\Z^\wedge_\Y)$};
\node (01) at (0,1) {$\QCoh(\X \times_\Z \Y )$};
\node (11) at (3,1) {$\ICoh_0 \bigt{ \X^\wedge_{\X \times_\Z \Y} }$};
\path[<-,font=\scriptsize,>=angle 90]
(00.east) edge node[above] {$(\primef)^{?}$} (10.west);
\path[<-,font=\scriptsize,>=angle 90]
(01.east) edge node[above] {$('\! F)^{?}$} (11.west);
\path[<-,font=\scriptsize,>=angle 90]
(01.south) edge node[right] {$ G^*$} (00.north);
\path[<-,font=\scriptsize,>=angle 90]
(11.south) edge node[right] {$\eta^{!,0}$} (10.north);
\end{tikzpicture}
\end{equation}
\begin{equation} \label{diag:commutation of ?-pushforward with sharp pullbacks}
\begin{tikzpicture}[scale=1.5]
\node (00) at (0,0) {$\QCoh(\Y)$};
\node (10) at (3,0) {$\ICoh_0(\Z^\wedge_\Y)$.};
\node (01) at (0,1) {$\QCoh(\X \times_\Z \Y )$};
\node (11) at (3,1) {$\ICoh_0 \bigt{ \X^\wedge_{\X \times_\Z \Y} }$};
\path[<-,font=\scriptsize,>=angle 90]
(00.east) edge node[above] {$(\primef)^{?}$} (10.west);
\path[<-,font=\scriptsize,>=angle 90]
(01.east) edge node[above] {$(' \!  F)^{?}$} (11.west);
\path[->,font=\scriptsize,>=angle 90]
(01.south) edge node[right] {$ G_*$} (00.north);
\path[->,font=\scriptsize,>=angle 90]
(11.south) edge node[right] {$\eta_?$} (10.north);
\end{tikzpicture}
\end{equation}

\begin{example} \label{example:loop Omega of a scheme}

Let us illustrate the adjuction $(\eta^{!,0}, \eta_?)$ in the simple example where $\X = \Y = \pt$, both mapping to a marked point of $\Z$. We further assume that $\Z =Z$ is a bounded affine scheme locally of finite presentation. In this case, $\X \times_\Z \Y = \Omega Z := \pt \times_Z \pt$ and the adjunction takes the form
\begin{equation} \label{adj:exceptional-adjunction-bounded}
\nonumber
\begin{tikzpicture}[scale=1.5]
\node (a) at (0,1) {$\eta^{!,0}: \ICoh(Z^\wedge_\pt)   $};
\node (b) at (4,1) {$  \ICoh_0(\pt^\wedge_{\Omega Z})
=: \vDmod(\Omega Z): \eta_?$.};
\path[->,font=\scriptsize,>=angle 90]
([yshift= 1.5pt]a.east) edge node[above] {  } ([yshift= 1.5pt]b.west);
\path[->,font=\scriptsize,>=angle 90]
([yshift= -1.5pt]b.west) edge node[below] { } ([yshift= -1.5pt]a.east);
\end{tikzpicture}
\end{equation}
Now, the LHS is equivalent to $\U(L) \mod$, where $L$ is the DG Lie algebra $\Tang_{Z,z}[-1]$, see \cite{DAG-X}.
On the other hand, since $\Omega Z$ is a formal group DG scheme with Lie algebra $L$, we have $\Gamma(\Omega Z, \O) \simeq \U(L)^* \simeq \Sym (L^*)$ as commutative DG algebras; the RHS is thus equivalent to 
$$
(\Sym (L^*) 	\rtimes \U(L) ) \mod.
$$
The adjunction in question becomes
$$ 
\begin{tikzpicture}[scale=1.5]
\node (a) at (0,1) {$ \U(L) \mod$};
\node (b) at (4,1) {$ (\Sym (L^*) 	\rtimes \U(L) ) \mod$,};
\path[->,font=\scriptsize,>=angle 90]
([yshift= 1.5pt]a.east) edge node[above] {$\ind_\WW$ } ([yshift= 1.5pt]b.west);
\path[->,font=\scriptsize,>=angle 90]
([yshift= -1.5pt]b.west) edge node[below] { $\oblv_\WW$} ([yshift= -1.5pt]a.east);
\end{tikzpicture}
$$
that is, the induction/restriction adjunction along the algebra map $\U(L) \to \Sym L^* \rtimes \U(L)$ defined by $\phi \squigto 1 \otimes \phi$.

\end{example}

\ssec{Base-change}

We now construct base-change isomorphisms between $(*,0)$-pushforwards and $(!,0)$-pullbacks. 

\sssec{}

As we have seen, in the unbounded context we have defined $(*,0)$-pushforwards only for morphisms of the form 
$$
[\X \to \Y]
\longto
[\X \to \Z]
$$
in $\Arr(\Stkperflfp)$, and $(!,0)$-pullbacks only for arrows of the form
$$
[\W \times_\Z \X \to \W]
\longto
[\X \to \Z].
$$
Let us call \emph{nil-isomorphisms} the arrows of the first type and \emph{cartesian} the arrows of the second type. It is straightforward to check that the associated $\infty$-category $\Corr(\Arr(\Stkperflfp))_{\niliso;\cart}$ of correspondences is well-defined.

\begin{prop}
The $(*,0)$-pushforwards and $(!,0)$-pullbacks assemble to a functor
$$
\ICoh_0:
\Corr(\Arr(\Stkperflfp))_{\niliso;\cart}
\longto
\DGCat.
$$
\end{prop}

\begin{proof}
A diagram
$$
[\X \to \Y]
\xto{\; \; \xi \; \; }
[\X \to \Z]
\xleftarrow{ \; \; \eta \; \; }
[\W \times_\Z \X \to \W],
$$
with all stacks in $\Stkperflfp$, gives rise to a correspondence
$$
[\X \to \Y]
\xleftarrow{ \; \; \wt\eta \; \; }
[\W \times_\Z \X \to \W \times_\Z \Y]
\xto{\; \; \wt\xi \; \; }
[\X \times_\Z \W \to \W]
$$
and to a square
$$ 
\begin{tikzpicture}[scale=1.5]
\node (00) at (0,0) {$\ICoh_0(\Y^\wedge_\X)$};
\node (10) at (3.5,0) {$\ICoh_0(\Z^\wedge_\X)$.};
\node (11) at (3.5,1.2) {$\ICoh_0(\W^\wedge_{\W \times_\Z \X})$};
\node (01) at (0,1.2) {$\ICoh_0((\W \times_\Z \Y)^\wedge_{\W \times_\Z \X})$};
\path[->,font=\scriptsize,>=angle 90]
(00.east) edge node[above] {$\xi_{*,0}$} (10.west);
\path[->,font=\scriptsize,>=angle 90]
(01.east) edge node[above] {$\wt \xi_{*,0}$} (11.west);
\path[<-,font=\scriptsize,>=angle 90]
(01.south) edge node[right] {$ {\wt\eta}^{!,0}$} (00.north);
\path[<-,font=\scriptsize,>=angle 90]
(11.south) edge node[right] {$\eta^{!,0}$} (10.north);
\end{tikzpicture}
$$
The latter is canonically commutative: to see this, use Proposition \ref{prop:exterior-product-ICohzero-over-QCoh} to rewrite the two DG categories on the top as relative tensor products.
\end{proof}


\section{The center of $\H(\Y)$} \label{sec:center}

The goal of this final section is to compute the center of the monoidal DG category $\H(\Y)$ associated to $\Y \in \Stkperfevcoclfp$.

\ssec{The center of a monoidal DG category}

In this preliminary section, we will recall some general facts on the center of a monoidal DG category $A$. For instance, we will recall why, for $A$ rigid and pivotal, the center of $A$ (with values in a bimodule category $M$) sits in a monadic adjunction
$$ 
\begin{tikzpicture}[scale=1.5]
\node (a) at (0,1) {$M$};
\node (b) at (2,1) {$\Z_A(M)$.};
\path[->,font=\scriptsize,>=angle 90]
([yshift= 1.5pt]a.east) edge node[above] { $ \ev^L $ } ([yshift= 1.5pt]b.west);
\path[->,font=\scriptsize,>=angle 90]
([yshift= -1.5pt]b.west) edge node[below] { $ \ev$ } ([yshift= -1.5pt]a.east);
\end{tikzpicture}
$$

\sssec{}

Let $A$ be a monoidal DG category and $M$ an $(A,A)$-bimodule. Then, the \emph{center of $M$ with respect to $A$} is the DG category 
$$
\Z_A(M) := \Fun_{(A,A)\bbimod}(A,M).
$$
When $M=A$, we write $\Z(A)$ in place of $\Z_A(A)$.

\sssec{}

Denote by
$$
\ev: \Z_A(M) \longto \Fun_{A \mmod}(A, M) \simeq M
$$
the tautological (continuous and) conservative functor that forgets the right $A$-action. Thus, we see that $\Z_A(M)$ consists of elements $m \in M$ with extra structure: this extra structure captures precisely the commutation of $m$ with elements of $A$.

\sssec{} \label{sssec:formula for cosimplicial arrows}

To express $\Z_A(M)$ more explicitly, we use the bar resolution of the $(A,A)$-bimodule $A \simeq A \otimes_A A$ to obtain:
\begin{eqnarray} \label{eqn: Bar stuff for center}
\nonumber
\Z_A(M) 
& \simeq &
\Tot \, \Fun_{A \otimes A^\rev} (A^{\otimes (\bullet+2)}, M) \\
\nonumber
& \simeq &
\Tot \, \Fun(A^{\otimes \bullet}, M) \\
& = &
\lim
(
M 
\rr 
\Fun(A, M)
\rrr
\Fun(A \otimes A, M) \cdots ).
\end{eqnarray}
Unraveling the construction, for $n \geq 0$ and $i = 0, \ldots, n+1$, the structure arrow
$$
\partial_n^i: \Fun ( A^{\otimes n}, M)
\longto
\Fun ( A^{\otimes (n+1)}, M)
$$
sends $f$ to the functor
$$
\partial_n^i(f): 
(a_1, \ldots a_{n+1})
\squigto
\begin{cases} 
a_1 \star f(a_2, \ldots, a_{n+1}) , & \mbox{if } i=0 \\
f(a_1, \ldots, a_i \star a_{i+1}, \ldots,  a_{n+1}) , & \mbox{if } 1 \leq i \leq n \\
 f(a_1, \ldots, a_{n}) \star a_{n+1} , & \mbox{if } i=n+1, \\
 \end{cases}
$$
with the obvious reinterpretation in the case $n=0$.

\sssec{} \label{sssec:description of ZA for A rigid}

Assume from now on that $A$ is \emph{rigid} and \emph{pivotal} (in particular, $A$ is compactly generated), see Sections \ref{sssec:rigidity conventions} and \ref{sssec:pivotality} for the definitions. Any rigid monoidal category is self-dual is such a way that $m^\vee \simeq m^R$. Pivotality further implies that such self-duality $A^\vee \simeq A$ is an equivalence of $(A,A)$-bimodules. Using this and \cite[Proposition D.2.2]{ShvCat}, we can turn the limit of \eqref{eqn: Bar stuff for center} into a colimit by taking left adjoints. We obtain that $\Z_A(M)$ can also be computed as the colimit of the simplicial DG category
\begin{equation} \label{eqn: simplicial cat computing ZAM with A rigid}
\Z_A(M)
\simeq
\uscolim{[n] \in \bDelta^\op} \;
\Bigt{
\cdots
 A \otimes A \otimes M
\rrr
 A \otimes M 
\rr
M
},
\end{equation}
with arrows induced by multiplication, action and reversed action as usual.

\begin{rem}
If $A$ is rigid but not pivotal, then the equivalence $A^\vee \simeq A$ is $(A,A)$-bilinear provided that we twist one of the two actions on $A$ by the monoidal automorphism of $A$ induced by $a \mapsto (a^\vee)^\vee$ at the level of compact objects. We are grateful to Lin Chen for pointing out this issue, which is discussed in \cite[Section 3.2]{BN}.
\end{rem}

\begin{rem} \label{rem:trace = center}
The colimit on the RHS computes the relative tensor product $A \otimes_{A \otimes A^\rev} M$: this is called the \emph{trace} of $M$ with respect to $A$, denoted $\Tr(A,M)$.
Thus, center and \emph{trace} are canonically identified whenever $A$ is rigid and pivotal. See \cite[Section 5.1]{BFN} for more details on this general situation.
\end{rem}

\sssec{}

Thanks to \cite[Corollary D.4.9 (a)]{ShvCat}, we obtain that the cosimplicial diagram featuring in (\ref{eqn: Bar stuff for center}) satisfies the left Beck-Chevalley condition.\footnote{Alternatively, the simplicial diagram appearing in (\ref{eqn: simplicial cat computing ZAM with A rigid}) satisfies the right Beck-Chevalley condition.} 
We then use \cite[Theorem 4.7.5.2]{HA} to deduce that:
\begin{itemize}
\item
 the evaluation functor $\ev: \Z_A(M) \to M$ is monadic: it is conservative, continuous (hence the split colimit condition is obviously satisfied) and it admits a left adjoint $\ev^L$ yielding the equivalence
$$
\Z_A(M) \simeq 
\bigt{ \ev \circ \ev^L } \mod (M);
$$
\item
the functor underlying the monad $ \ev \circ \ev^L $ is isomorphic to the composition $(\partial_0^0)^L \circ (\partial_0^1) \simeq m \circ (m^\rev)^R$.
\end{itemize}

\ssec{The computation}

In this section, we let $A$ be the monoidal DG category $\H(\Y)$ associated to $\Y \in \Stkperfevcoclfp$. It is rigid and pivotal in view of Lemma \ref{lem:pivotal}.
To reduce cuttler, let us set 
$$
Q:=\QCoh(\Y),
\hspace{.4cm}
H:= \H(\Y)= \ICoh_0((\Y \times \Y)^\wedge_\Y),
\hspace{.4cm}
Z := \Z(H).
$$
To compute $Z$, we will use a variation of the equivalence (\ref{eqn: simplicial cat computing ZAM with A rigid}) which takes into account the monoidal functor $\primeDelta_{*,0}: Q \longto H$.
This will allow us to use the equivalence $\Z(Q) \simeq \QCoh(L\Y)$ established in \cite{BFN}.

\sssec{} \label{sssec:structure-maps-for-Bar-stuff}

Let $M$ be an $(H, H)$-bimodule category.
For any $m \geq 0$, we shall use the notation
$$
H^{\otimes_Q m}
:=
\underset{m \, \text{times}}{\underbrace{H \otimes_Q \cdots \otimes_Q H}},
$$
with the understanding that $H^{\otimes_Q 0} =Q$.
The bar resolution of $H$ in the symmetric monoidal $\infty$-category $Q \mmod$ allows to write $\Z_H(M)$ (or rather, $\Tr(H,M)$) as the colimit
\begin{equation} \label{eqn: simplicial cat for H}
\Z_H(M)
\simeq
\uscolim{[n] \in \bDelta^\op} \;
\Bigt{
\cdots
(H \otimes_Q H) \otimes_{Q \otimes Q^\rev} M
\rrr
H \otimes_{Q \otimes Q^\rev} M
\rr
Q \otimes_{Q \otimes Q^\rev} M
}.
\end{equation}
of the simplicial category 
$$
 H^{\otimes_Q (\bullet+2)} 
 \usotimes{{H \otimes H^\rev}} 
 M
\simeq
 H^{\otimes_Q (\bullet)} 
 \usotimes{{Q \otimes Q^\rev}} 
 M,
$$
with the obvious structure arrows.
In  the case $M=H$, we have
\begin{equation} \label{eqn:express the center using Q}
Z 
\simeq
\colim 
\Bigt{
H^{\otimes_Q \bullet}
\usotimes{Q \otimes Q^\rev} H
}.
\end{equation}
Let us proceed to express the DG catergory on the RHS in geometric terms.
\begin{prop}
There is a natural equivalence
\begin{equation} \label{eqn:compute the center geometrically}
H^{\otimes_Q \bullet}
\usotimes{Q \otimes Q^\rev} H
\longto
\ICoh_0((\Y^{\bullet +1})^\wedge_{L\Y})
\end{equation} 
of simplicial categories, where 
$$
\ICoh_0((\Y^{\bullet +1})^\wedge_{L\Y})
$$
is the simplicial category induced, under $(*,0)$-pushforwards, by the Cech resolution of $\Y \to \pt$.
\end{prop}

\begin{proof}
For fixed $n \geq 0$, we have the obvious equivalence:
$$
H^{\otimes_Q n}
\usotimes{Q \otimes Q^\rev} H
\simeq
H^{\otimes_Q (n+1)}
\usotimes{Q \otimes Q^\rev} Q.
$$
Proposition \ref{prop:exterior-product-ICohzero-over-QCoh-ev-coc-case} gives the equivalence 
$$
H^{\otimes_Q (n+1)}
\xto{\; \simeq \; }
\ICoh_0((\Y^{n+2})^\wedge_\Y).
$$
Combining this with Proposition \ref{prop:exterior-product-ICohzero-over-QCoh}, we obtain
$$
H^{\otimes_Q n}
\usotimes{Q \otimes Q^\rev} H
\simeq
\ICoh_0((\Y^{n+2})^\wedge_\Y )
\usotimes{Q \otimes Q^\rev}
Q
\xto{\; \simeq \; }
\ICoh_0((\Y^{n+1})^\wedge_\LY ).
$$
It is a routine exercise (left to the reader) to unravel the structure functors.
\end{proof}

We can now state and prove the first part of our main theorem.

\begin{thm} \label{thm:Main-first-part}
For $\Y \in \Stkperfevcoclfp$, there is a canonical equivalence
$$
\Z(\H(\Y))
\simeq
\ICoh_0(\pt^\wedge_\LY).
$$
\end{thm}

\begin{proof} 
The above proposition yields the equivalence
$$
Z 
\simeq 
\Tot 
\Bigt{
\ICoh_0((\Y^{\bullet+1})^\wedge_\LY) 
},
$$
where the totalization is taken with respect to the $?$-pullbacks. It remains to apply descent on the RHS: we use Proposition \ref{prop:descent for ICOHzero-general-case} to obtain
\begin{equation} \label{eqn:descent for Dmods on LY}
\ICoh_0(\pt^\wedge_{L\Y})
\xto{\; \simeq \; }
 \Tot \Bigt{ \ICoh_0((\Y^{\bullet + 1})^\wedge_{L\Y})},
\end{equation}
as claimed.
\end{proof}

\ssec{Relation between $\H(\Y)$ and its center}

In the previous section, we have constructed an equivalence $Z \simeq \ICoh_0(\pt^\wedge_{\LY})$. Our next task is to describe what the adjunction $H \rightleftarrows Z$ becomes under such equivalence.

\sssec{} \label{sssec: s-adjucntions}

Let
$$
s: [\LY \to \Y]
\longto
[\Y \to \Y^2]
$$
be the map in $\Arr(\Stkperflfp)$ defined by the cartesian diagram
\begin{equation}  \nonumber
\begin{tikzpicture}[scale=1.5]
\node (00) at (0,0) {$\Y$};
\node (10) at (1.5,0) {$\Y \times \Y$.};
\node (01) at (0,.8) {$\LY $};
\node (11) at (1.5,.8) {$\Y$};
\path[->,font=\scriptsize,>=angle 90]
(00.east) edge node[above] {$ $} (10.west);
\path[->,font=\scriptsize,>=angle 90]
(01.east) edge node[above] {$ $} (11.west);
\path[->,font=\scriptsize,>=angle 90]
(01.south) edge node[right] {$ $} (00.north);
\path[->,font=\scriptsize,>=angle 90]
(11.south) edge node[right] {$ $} (10.north);
\end{tikzpicture}
\end{equation}
By its very construction, the functor $s^{!,0}: \ICoh_0((\Y \times \Y)^\wedge_\Y) \to \ICoh_0(\Y^\wedge_\LY)$ is given by
$$
H
\simeq
(Q \otimes Q) \usotimes{Q \otimes Q^\rev} H
\xto{ \Delta^* \otimes \id}
Q \usotimes{Q \otimes Q^\rev} H
\xto{\simeq}
 \ICoh_0(\Y^\wedge_\LY).
$$ 
By adjuction, $s_?$ is the continuous functor induced by the $\QCoh$-pushforward $\Delta_*$. Since $\Delta$ is affine, the functor $s_?$ is monadic.

\sssec{} 

Denote by 
$v: [\LY \to \Y] \to [\LY \to \pt]$
the obvious morphism in $\Arr(\Stkperflfp)$. As seen above, the adjunction
$$ 
\begin{tikzpicture}[scale=1.5]
\node (a) at (0,1) {$\QCoh(\LY) $};
\node (b) at (3,1) {$ \ICoh_0(\Y^\wedge_\LY)$};
\path[->,font=\scriptsize,>=angle 90]
([yshift= 1.5pt]a.east) edge node[above] { $v_{*,0}$} ([yshift= 1.5pt]b.west);
\path[->,font=\scriptsize,>=angle 90]
([yshift= -1.5pt]b.west) edge node[below] {$v^?$ } ([yshift= -1.5pt]a.east);
\end{tikzpicture}
$$
is obtained from the adjunction
$$ 
\begin{tikzpicture}[scale=1.5]
\node (a) at (0,1) {$H \otimes_{Q} H $};
\node (b) at (2,1) {$H $};
\path[->,font=\scriptsize,>=angle 90]
([yshift= 1.5pt]a.east) edge node[above] {$\wt m$ } ([yshift= 1.5pt]b.west);
\path[->,font=\scriptsize,>=angle 90]
([yshift= -1.5pt]b.west) edge node[below] { $(\wt m)^R$} ([yshift= -1.5pt]a.east);
\end{tikzpicture}
$$
by tensoring up with $H \otimes_{H \otimes H^\rev} -$.

\sssec{}

We are now ready to combine all the constructions above in the second part of our main theorem, which describes the relationship of $\Z(\H(\Y))$ and $\H(\Y)$ in geometric terms.

\begin{thm} \label{MAIN-thm}
Under the equivalence 
$\ICoh_0(\pt^\wedge_{L\Y})
\simeq
\Z(\H(\Y))$
of Theorem \ref{thm:Main-first-part}, the adjunction 
$$ 
\begin{tikzpicture}[scale=1.5]
\node (a) at (0,1) {$\H(\Y)$};
\node (b) at (2,1) {$\Z(\H(\Y))$};
\path[->,font=\scriptsize,>=angle 90]
([yshift= 1.5pt]a.east) edge node[above] {$\ev^L$} ([yshift= 1.5pt]b.west);
\path[->,font=\scriptsize,>=angle 90]
([yshift= -1.5pt]b.west) edge node[below] {$\ev$} ([yshift= -1.5pt]a.east);
\end{tikzpicture}
$$ 
goes over to the (monadic) adjunction
\begin{equation} \nonumber
\begin{tikzpicture}[scale=1.5]
\node (a) at (0,1) {$\ICoh_0((\Y \times \Y)^\wedge_\Y)$};
\node (b) at (4,1) {$\ICoh_0(\pt^\wedge_{L\Y})$.};
\path[->,font=\scriptsize,>=angle 90]
([yshift= 1.5pt]a.east) edge node[above] {$v_{*,0} \circ s^{!,0}$} ([yshift= 1.5pt]b.west);
\path[->,font=\scriptsize,>=angle 90]
([yshift= -1.5pt]b.west) edge node[below] {$s_? \circ v^{!,0} $} ([yshift= -1.5pt]a.east);
\end{tikzpicture}
\end{equation}
\end{thm}

\begin{proof}
It suffices to show that the composition
$$
\ICoh_0(\pt^\wedge_{L\Y})
\xto{\; \simeq \; }
Z
\xto{\ev}
H
$$
corresponds to the functor $s_? \circ v^{!,0}$. By (\ref{eqn:express the center using Q}), the structure functor $\ev$ is the composition
$$
Z
:=
H
\usotimes{H \otimes H^\rev} 
H
\xto{(\wt m)^R \otimes \id_H}
(H \otimes_Q H ) 
\usotimes{H \otimes H^\rev} 
H
\simeq
Q
\usotimes{Q \otimes Q^\rev} 
 H
\xto{\Delta_* \otimes \id_H}
H,
$$
so that the assertion follows the remarks above. 
\end{proof}

\begin{example}

Let us illustrate the above adjunction in the case where $\Y = G$ is a group DG scheme (as usual, bounded and lfp).
The automorphism of $G \times G$ given by $(x, y) \mapsto (x, xy)$ yields an equivalence
\begin{equation} \label{eqn: H for group}
\H(G) 
\simeq 
\QCoh(G) \otimes \ICoh(G^\wedge_\pt)
\end{equation}
of DG categories. Now set $\Omega G := \pt \times_G \pt \simeq \Spec (\Sym \g^*[1])$. Since
\begin{equation} \label{eqn:LG as a product}
LG \simeq G \times \Omega G
\end{equation}
canonically, we obtain
\begin{equation} \label{eqn:Z of H(G)}
\Z(\H(G)) \simeq
\ICoh_0(\pt^\wedge_{LG})
\simeq
\Dmod(G) \otimes \vDmod(\Omega G),
\end{equation}
\begin{equation} \label{eqn: intermediate cat for  H(G)}
\ICoh_0(G^\wedge_{LG})
\simeq
\ICoh_0(G^\wedge_{G \times \Omega G})
\simeq
\QCoh(G) \otimes \vDmod(\Omega G).
\end{equation}
Recalling the example treated in Section \ref{example:loop Omega of a scheme}, the adjunction $\H(G) \rightleftarrows \Z(\H(G))$ reads:
\begin{equation} 	\nonumber
\begin{tikzpicture}[scale=1.5]
\node (a) at (0,1) {$\QCoh(G) \otimes \ICoh(G^\wedge_\pt)$};
\node (b) at (4,1) {$\QCoh(G) \otimes \vDmod(\Omega G)$};
\node (c) at (7.7,1) {$\Dmod(G) \otimes \vDmod(\Omega G)$,};
\path[ ->,font=\scriptsize,>=angle 90]
([yshift= 1.5pt]a.east) edge node[above] { $\id \otimes \ind_\WW$} ([yshift= 1.5pt]b.west);
\path[->,font=\scriptsize,>=angle 90]
([yshift= -1.5pt]b.west) edge node[below] { $\id \otimes \oblv_\WW$ } ([yshift= -1.5pt]a.east);
\path[ ->,font=\scriptsize,>=angle 90]
([yshift= 1.5pt]b.east) edge node[above] {$ \ind_L \otimes \id$} ([yshift= 1.5pt]c.west);
\path[->,font=\scriptsize,>=angle 90]
([yshift= -1.5pt]c.west) edge node[below] {$\oblv_L \otimes \id$} ([yshift= -1.5pt]b.east);
\end{tikzpicture}
\end{equation}
where $\ind_L$ is the induction functor for left $\fD$-modules (which makes sense as $G$ is bounded).
\end{example}


\end{document}